\documentclass[10pt,a4paper,reqno]{amsart}
\usepackage{hyperref}

\usepackage[margin=1in,marginparwidth=0.8in]{geometry}
\usepackage{tikz}
\usepackage{graphicx}
\usepackage{todonotes}
\usepackage{mathtools}
\usepackage{shuffle}
\usepackage{color}
\usepackage{comment}

\usepackage{ amssymb, mathrsfs }

\def\ceq{\:{}\coloneqq{}\:}
\def\={\:{}={}\:}
\newcommand{\modsh}{\:\mathrel{\overset{\shuffle}{=}}\:}

\DeclareMathOperator{\Li}{Li}
\DeclareMathOperator{\Sym}{Sym}
\DeclareMathOperator{\Alt}{Alt}
\DeclareMathOperator{\id}{id}

\newcommand{\abs}[1]{|#1|}

\def\Symb{\operatorname{Symb}}
\def\Symbsh{\operatorname{Symb}^\shuffle}
\def\LiCS{\mathcal{L}^\shuffle}

\def\SCS{\mathcal{S}^\shuffle}

\def\LiZS{\mathscr{L}}

\DeclareMathOperator{\CR}{cr}
\DeclareMathOperator{\R}{r}
\DeclareMathOperator{\sv}{sv}
\DeclareMathOperator{\Ss}{s}

\renewcommand{\Im}{\operatorname{Im}}
\renewcommand{\Re}{\operatorname{Re}}

\newcommand{\mot}{\mathfrak{m}}
\newcommand{\CC}{\mathbb{C}}
\newcommand{\QQ}{\mathbb{Q}}
\newcommand{\ZZ}{\mathbb{Z}}

\newcommand {\sm} {\smallsetminus}

\DeclareMathOperator{\Aut}{Aut}
\DeclareMathOperator{\pr}{pr}

\newcommand{\Mod}[1]{\ (\mathrm{mod}\ \text{\normalfont #1})}

\newtheorem{Thm}{Theorem}
\newtheorem{Cor}[Thm]{Corollary}
\newtheorem{Prop}[Thm]{Proposition}
\newtheorem{Lem}[Thm]{Lemma}

\theoremstyle{definition}
\newtheorem{Def}[Thm]{Definition}
\newtheorem{Rem}[Thm]{Remark}

\allowdisplaybreaks

\title{On functional equations for Nielsen polylogarithms}
\date{\today}
\author[Charlton]{Steven Charlton}
\address{Max Planck Institute for Mathematics, Vivatsgasse 7,
    Bonn 53111, Germany}
\email{steven.charlton@gmail.com}

\author[Gangl]{Herbert Gangl}
\address{Department of Mathematical Sciences, Durham University, Durham DH1 3LE, United Kingdom}
\email{herbert.gangl@durham.ac.uk}

\author[Radchenko]{Danylo Radchenko}
\address{Max Planck Institute for Mathematics, Vivatsgasse 7,
    Bonn 53111, Germany}
\email{danradchenko@gmail.com}


\begin{document}
\begin{abstract}
	We derive new functional equations for Nielsen polylogarithms. We show that, when viewed modulo $\Li_5$ and products of lower weight functions, the weight $5$ Nielsen polylogarithm $S_{3,2}$ satisfies the dilogarithm five-term relation. We also give some functional equations and evaluations for Nielsen polylogarithms in weights up to $8$, and general families of identities in higher weight.
\end{abstract}
\date{13 August 2019}
\maketitle

\section{Introduction}
The classical $m$-th polylogarithm function $\Li_m$ is an 
analytic function defined by the Taylor series
	\[\Li_m(z) \= \sum_{n=1}^{\infty}\frac{z^n}{n^m}\,,\]
convergent for $\left|z\right| < 1$. For $m\ge 1$ the function $\Li_m$ 
extends to a multivalued analytic function 
on $\mathbb{P}^1(\CC)\sm\{0,1,\infty\}$, 
which can be seen, for example, from the recursive formula
	\[\Li_m(z) \= \int_{0}^{z}\Li_{m-1}(t)\frac{dt}{t}\,,\]
together with the initial condition $\Li_1(z)\=-\log(1-z)$.
Polylogarithms appear in several areas of mathematics:
for example, the Euler dilogarithm $\Li_2$ (or, more precisely, a 
single-valued version) can be used to compute volumes of hyperbolic 
$3$-folds, special values of Dedekind zeta functions at $s=2$, and it 
is intimately related to algebraic $K$-theory (more precisely to $K_3$ and $K_2$) 
of number fields, see~\cite{Za1},~\cite{Za3},~\cite{Za-Ga}.
One of the most curious features of polylogarithms is that they satisfy a 
plethora of identities and functional equations, the most famous of which is 
undoubtedly the five-term relation,
	\begin{equation} \label{eq:fiveterm}
    \Li_2(x)+\Li_2(y)
    -\Li_2\Big(\frac{x}{1-y}\Big)
    -\Li_2\Big(\frac{y}{1-x}\Big)
	+\Li_2\Big(\frac{xy}{(1-x)(1-y)}\Big) \= -\log(1-x)\log(1-y)\,,
    \end{equation}
(for \( |x| + |y| < 1 \)) in this or any of its equivalent forms 
(see Section 1.5 in \cite{Le}). Numerous other identities of this
kind are known, both for $\Li_2$ and for $\Li_m$ for $m>2$, but as soon as
$m$ becomes greater than $7$, the only relations that are known in general
are the inversion identity that relates $\Li_m(z)$ and $\Li_m(z^{-1})$, and 
the distribution relations 
$n^{1-m}\Li_m(z^n) \= \sum_{\lambda^n=1}\Li_m(\lambda z)$ for $n\geq 1$.
For an introduction to functional equations for polylogarithms we
refer the reader to~\cite{Za2}. Many examples of functional equations
for $\Li_m$ up to $m=5$ can already be found in Lewin's classical book \cite{Le},
while newer results are e.g.~given in \cite{Ki} for \( m = 2 \) and
in \cite{WoLi3}, \cite{Go1} for \( m \leq 3 \).
Inaugural results in weight 6 and 7 are treated in \cite{Ga-s},\cite{Ga-c}.  
For further examples, and background, we refer also to these theses \cite{Ch,Ga,R}.

In~\cite{Ni} Nielsen defined and studied the functions $S_{n,p}$
given by the following integral
    \begin{equation}\label{eqn:nielsendef}
    S_{n,p}(z) \ceq \frac{(-1)^{n+p-1}}{(n-1)!\,p!}
    \int_{0}^{1}\log^{n-1}(t)\log^{p}(1-zt)\frac{dt}{t}\,.
    \end{equation}
It is easy to show that $S_{m-1,1} \= \Li_{m}$, so that classical polylogarithms 
are a special case of Nielsen's generalised polylogarithms. 
On the other hand, Nielsen polylogarithms themselves are special cases
of multiple polylogarithms and iterated integrals:
    \begin{equation}\label{eq:nielsenasLi} 
    \begin{split}
    S_{n,p}(z) &\= \Li_{\{1\}^{p-1},n+1}(1,\dots,1,z)\\
    &\= (-1)^p I(0; \{1\}^p, \{0\}^n; z)\,,
    \end{split}
    \end{equation}
where $\{a\}^k$ denotes the string $a$ repeated $k$ times.  Here $\Li_{n_1,\dots,n_d}$ is the multiple polylogarithm function
    \[
    \Li_{n_1,\dots,n_d}(z_1,\dots,z_d) \ceq \sum_{0<k_1<\dots<k_d}
    \frac{z_1^{k_1}\cdots z_d^{k_d}}
    {k_1^{n_1}\cdots k_d^{n_d}}\,,
    \]
and $I$ is the iterated integral
    \[
    I(x_0; x_1, \ldots, x_N; x_{N+1}) \ceq
    \int_{x_0 < t_1 < \cdots < t_d < x_{N+1}} 
    \frac{dt_1}{t_1 - x_1} \wedge \frac{dt_2}{t_2 - x_2} \wedge \cdots \wedge
    \frac{dt_d}{t_d - x_N}\,.
    \]
(The integral implicitly depends on the choice of a path going 
from $x_0$ to $x_{N+1}$, where the integration variables $t_i$ are considered to be ordered on the path.)
Note that, as a multiple polylogarithm, $S_{n,p}$ has weight $n+p$ and 
depth $\leq p$. 
Moreover, the iterated integral identity extends the range of definition 
of $S_{n,p}$ to the case $ n = 0 $ and~$ p = 0 $.

While they are certainly not as well-studied as their classical counterparts, 
Nielsen polylogarithms do naturally appear in some calculations in 
quantum electrodynamics (for some references see~\cite{Ko2}),
and they also provide the simplest examples (aside from $\Li_m$) 
of harmonic polylogarithms, that appear, for example, 
in computations of planar scattering amplitudes~\cite{Di-Du-Pe}.
For the original paper on harmonic polylogarithms, see~\cite{Re-Ve}.
There is also some interest in computing special values of $S_{n,p}(z)$, at 
least when $z$ is a root of unity, since they arise 
in ~$\varepsilon$-expansions of some Feynman diagrams~\cite{Da-Ka}, 
and also in connection with Mahler measures and planar random 
walks~\cite{Bo-St1} (see also~\cite{Bo-St2}).  For approaches to the numerical computation of values of Nielsen polylogarithms and harmonic polylogarithms, see respectively \cite{Ko-Mi-Re} and \cite{Ge-Re}.

Despite this, there appear to be very few results concerning functional
equations for $S_{n,p}$. In fact, excluding the case of classical polylogarithms $\Li_m=S_{m-1,1}$, the most general functional equations that we have found in the literature are the relations that express $S_{n,p}(\gamma(z))$
in terms of $S_{n',p'}(z)$, where $n'+p'=n+p$ and $\gamma(z)$ is one 
of the M\"obius transformations $\{z, 1-z, \frac{1}{z}, \frac{z-1}{z}, 
\frac{z}{z-1},\frac{1}{1-z}\}$ (see~\cite{Ko2}, or more recent~\cite{shangNew}).
There is no known analogue of distribution relations for $S_{n,p}$ 
when $p\ge 2$.

Part of our initial motivation comes from the conjectures of Goncharov regarding the structure of the so-called motivic Lie coalgebra, which predicts the reduction in depth of certain linear combinations of iterated integrals.  One of these predictions is that $S_{3,2}$ of the five-term relation reduces to $\Li_5$, i.e.
    \begin{equation*} 
    S_{3,2}(x)+S_{3,2}(y)
    -S_{3,2}\Big(\frac{x}{1-y}\Big)
    -S_{3,2}\Big(\frac{y}{1-x}\Big)
    +S_{3,2}\Big(\frac{xy}{(1-x)(1-y)}\Big) \= 0 \Mod{$\Li_5$, products} \,.
    \end{equation*}
    Indeed we establish this in Theorem \ref{thm:s32fiveterm} below, together with the explicit form of the $\Li_5$ terms, thus corroborating part of these conjectures in weight 5.  We also establish other examples of depth reduction for Nielsen polylogarithms, by way of giving functional equations for various $ S_{n,p} $ up to and including weight 8.  Extending the results of Nielsen and of K\"olbig we also consider various evaluations and ladders (i.e.~relations among \( S_{n,p} \) at \( \pm \theta^k \) for some algebraic number \( \theta \) and \( k\in \ZZ \)) of Nielsen polylogarithms, most of them apparently new.
    \medskip

\paragraph{\textbf{Acknowledgements.}}  This work was initiated during our joint stay at the Kyushu University Multiple Zeta Value Research Center under the grant \emph{2017 Kyushu University World Premier International Researchers Invitation Program `Progress 100'}.  This work continued during the Trimester Program \emph{Periods in Number Theory, Algebraic Geometry and Physics} at the Hausdorff Research Institute for Mathematics in Bonn.  We are grateful to these institutions, as well as to the Max Planck Institute for Mathematics in Bonn, for their hospitality, support and excellent working conditions.

\section{Motivic framework, and symbols}\label{sec:symbmot}

We first briefly recall some of the motivic framework of multiple polylogarithms and iterated integrals from the works of Goncharov \cite{Go-galois} and \cite{Go-mpl}, in particular their Hopf algebra structure and the symbol of iterated integrals.

\subsection{The Hopf algebra of motivic iterated integrals}

In \cite{Go-galois},  Goncharov upgraded the iterated integrals \( I(x_0; x_1, \ldots, x_N; x_{N+1}) \), \( x_i \in \overline{\QQ} \) to framed mixed Tate motives to define  motivic iterated integrals \( I^{\mot}(x_0; x_1, \ldots, x_N; x_{N+1}) \), living in a graded  (by the weight \( N \)) connected Hopf algebra \( \mathcal{H}_\bullet= \mathcal{H}_\bullet(\overline{\QQ}) \), denoted  \( \mathcal{A}_\bullet (\overline{\QQ}) \) in   \cite{Go-galois}.  The coproduct \( \Delta \) on this Hopf algebra is computed via Theorem 1.2 in \cite{Go-galois} as
\begin{align*}
&  \Delta I^{\mot}(x_0; x_1,\ldots,x_N; x_{N+1}) \= {} \\
& \sum_{\substack{0 = i_0 < i_1 < \cdots \\ 
		< i_{k} < i_{k+1} = N+1}} I^{\mot}(x_0; x_{i_1}, \ldots, x_{i_k}; x_{N}) \otimes \prod_{p=0}^{k} I^\mot(x_{i_p}; x_{i_{p+1}}, \ldots, x_{i_{p+1}-1}; x_{i_{p+1}}).
\end{align*}
This is often stated mnemonically as a sum over all semicircular polygons, with the left hand factor corresponding to the main polygon, and the right hand factor corresponding to the product over all small cut-off polygons.  A typical term is given by the following picture:
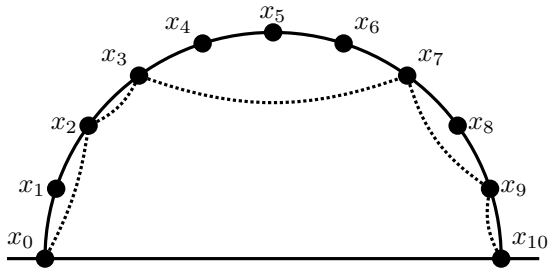
\begin{center}
	\begin{tikzpicture}[style = {very thick}]
	\draw (3,0) arc [start angle=0, end angle = 180, radius=3]
	node[pos=0.0, name=x10, inner sep=0pt] {}
	node[pos=0.1, name=x9, inner sep=0pt] {}
	node[pos=0.2, name=x8, inner sep=0pt] {}
	node[pos=0.3, name=x7, inner sep=0pt] {}
	node[pos=0.4, name=x6, inner sep=0pt] {}
	node[pos=0.5, name=x5, inner sep=0pt] {}
	node[pos=0.6, name=x4, inner sep=0pt] {}
	node[pos=0.7, name=x3, inner sep=0pt] {}
	node[pos=0.8, name=x2, inner sep=0pt] {}
	node[pos=0.9, name=x1, inner sep=0pt] {}
	node[pos=1.0, name=x0, inner sep=0pt] {};
	\draw (-3.5,0) -- (3.5,0);
	\filldraw
	(x0) circle [radius=0.1, fill=black, red]
	node[above left] {$x_0$}
	(x1) circle [radius=0.1, fill=black]
	node[left] {$x_1$}
	(x2) circle [radius=0.1, fill=black]
	node[left] {$x_2$}
	(x3) circle [radius=0.1, fill=black]
	node[above left] {$x_3$}
	(x4) circle [radius=0.1, fill=black]
	node[above left] {$x_4$}
	(x5) circle [radius=0.1, fill=black]
	node[above] {$x_5$}
	(x6) circle [radius=0.1, fill=black]
	node[above right] {$x_6$}
	(x7) circle [radius=0.1, fill=black]
	node[above right] {$x_7$}
	(x8) circle [radius=0.1, fill=black]
	node[right] {$x_8$}
	(x9) circle [radius=0.1, fill=black]
	node[right] {$x_9$}
	(x10) circle [radius=0.1, fill=black]
	node[above right] {$x_{10}$};
	\draw [bend right=10, densely dotted] (x0) edge (x2);
	\draw [bend right=20, densely dotted] (x2) edge (x3);
	\draw [bend right=20, densely dotted] (x3) edge (x7);
	\draw [bend right=20, densely dotted] (x7) edge (x9);
	\draw [bend right=20, densely dotted] (x9) edge (x10);
	\end{tikzpicture}
\end{center}
It is often convenient to invoke the reduced coproduct \( \Delta' = \Delta - 1\otimes \id{} - \id{} \otimes 1 \).

\subsection{The mod-products symbol}

Recall from \cite[Section 4.4]{Go-galois}, the `$\otimes^N$-invariant', or \emph{symbol}, of an iterated integral \( I \) of weight $ N $.  The symbol \( \Symb(I) \) is an algebraic invariant of \( I \), which respects functional equations among iterated integrals.  It can be obtained by maximally iterating the \( (N-1, 1) \) part of the coproduct \( \Delta \), giving \( \Symb = \Delta^{[N]} \) in weight \( N \).  Recall also the projectors \( \Pi_\bullet \) from \cite[Section 5.5]{Du-Ga-Rh} which annihilate the symbols of products.

The composition \( N \Pi_N \circ \Symb \;\eqqcolon\; \Symbsh \) is the so-called \emph{mod-products symbol}.  (We prefer \( N \Pi_N \) to $ \Pi_N $ in order to avoid unnecessary scaling factors, at the expense of that operator being no longer idempotent.)  By considering how the projector \( N \Pi_N \) acts on \( \Symb \) when written via the iterated coproduct in both ways (iterating the \( (N-1,1) \)-part and the \( (1,N-1) \)-part respectively), we derive the following recursion
\begin{align*}
\Symbsh I^{\mot}(x_0; \ldots; x_{N+1})
 \= 
 \mathclap{\phantom{\sum^{N}}} \smash{\sum_{j=1}^N} 
 \Symbsh I^{\mot}(x_0; x_1, \ldots, \widehat{x_j}, \ldots, x_N; x_{n+1}) &\otimes I^{\mot}(x_{j-1}; x_j; x_{j+1}) \\
 - \Symbsh I^{\mot}(x_1; x_2, \ldots, x_{N}; x_{N+1}) &\otimes I^{\mot}(x_0; \mathrlap{x_1}
 \phantom{x_N} 
 ; x_{N+1}) \\
 - \Symbsh I^{\mot}(x_0; x_1, \ldots, x_{N-1}; x_{N}) &\otimes I^{\mot}(x_0; x_N; x_{N+1}) \, .
\end{align*}
Here \( I^{\mot}(a;b;c) \) is regularised (cf.~(6) in \cite{Go-galois}) as
\[
	I^{\mot}(a;b;c) = \begin{cases}
		\log^{\mot}(1) & \text{if $ a = b $ and $ b = c $\,,} \\
		\log^{\mot}(\tfrac{1}{b - a})  & \text{if $ a \neq b $ and $ b = c $\,,} \\
		\log^{\mot}(b - c)  & \text{if $ a = b $ and $ b \neq c $\,,} \\
		\log^{\mot}(\tfrac{b - c}{b - a})  & \text{otherwise\,.}
	\end{cases}
\]
As usual with symbols, we will drop the \( \log^{\mot} \) from the notation and write tensors multiplicatively.  We will also omit \( {}^{\mot} \) from the iterated integral and Nielsen polylogarithm notation from now on.

Note that on the symbol level we can ignore signs in the tensor factors, since we work modulo 2-torsion, so we can identify  \( \otimes (-x) \) and  \( \otimes x \).  To emphasise that certain identities hold only on the level of the mod-products symbol, we shall write \( f \modsh g \) to mean \( \Symbsh f = \Symbsh g \).

\subsection{Lie coalgebra}

The coproduct induces a cobracket \( \delta = \Delta - \Delta^{\mathrm{op}} \), with \( \Delta^{\mathrm{op}} \) the opposite coproduct, on the Lie coalgebra of irreducibles
\[
\mathcal{L}_\bullet \ceq \mathcal{H}_{>0} / \mathcal{H}_{>0}^2 \, .
\]
We use the notation \( \{z\}_m \) for elements in the weight \( m \) pre-Bloch group \( B_m(F) \) (also called `polylogarithmic group' in the literature), where $F$ is any field. For a rigorous definition 
of \(B_m(F)\) see~\cite{Go1} \S1.9, but roughly one can think about it as the 
quotient space of formal linear combinations of elements of $F$ modulo the subspace 
given by specialisations of all functional equations for $\Li_m$. The 
conjectural structure of \( \mathcal{L}_\bullet \) implies that \( B_m(F) \) is 
isomorphic to a subgroup of \( \mathcal{L}_m(F) \), and hence one can think 
about \( \{z\}_m \) as the image of \( \Li_m(z) \) modulo products. 
In Section~\ref{sec:blochgroupids} we will also work with higher 
Bloch groups $\mathcal{B}_m(\QQ)$. These were originally defined (for
number fields) in~\cite{Za1}.
One can define $\mathcal{B}_m(F)$  as the kernel of $\delta$ restricted to the pre-Bloch group $B_m(F)$.

The 2-part of the cobracket in weight~\( N \), i.e. the projection to  \( \bigoplus_{k=2}^{N-2} \mathcal{L}_{k} \oplus \mathcal{L}_{N-k} \), is seen to annihilate all classical polylogarithms.  Conjecture 1.20 and Section 1.6 in \cite{Go-ams} on the structure of the motivic Lie coalgebra would imply that the kernel of the motivic cobracket should coincide with classical polylogarithms.

For example, up to weight 3 the 2-part vanishes identically for trivial reasons.  This corresponds to the fact that in weight 3 every iterated integral can be expressed in terms of the classical trilogarithm \( \Li_3 \), which is well known already from  Equation A.3.5 in (the appendix of) \cite{Le}, and given in a somewhat different form in \cite{hype5}.   The next case is weight 4, in which Goncharov predicted that
\[
  I_{3,1}(V(x,y), z) = 0 \Mod{$\Li_4$} \,.
\]
Here $I_{3,1}(x,y) = I(0; x, 0, 0, y; 1) $ and $ V(x,y) $ is any version of the five-term relation, such as in \eqref{eq:fiveterm} above.
This was established by the second author in \cite{gangl2016multiple}, and was subsequently also shown by Goncharov and Rudenko in \cite{goncharov2018motivic} where it played a key role in the proof of Zagier's polylogarithm conjecture for weight 4.

In weight 5, one of the predictions is that
\begin{equation}\label{eq:i41fiveterm}
 I_{4,1}^+(V(x,y), z) = 0 \Mod{$\Li_5$} \,,
\end{equation}
and this reduction is expected to play a similarly important role in any proof of Zagier's polylogarithm conjecture for weight 5. It follows from the special case of \eqref{eq:i41fiveterm} at $z=1$, that one also expects 
\[
 S_{3,2}(V(x,y)) = 0 \Mod{$\Li_5$} \,,
\]
which we prove in Theorem \ref{thm:s32fiveterm} below.

\section{General properties of Nielsen polylogarithms}
In this section we recall the basic two-term relations for Nielsen polylogarithms of arbitrary weight (Propositions \ref{prop:nielsenreflection} and \ref{prop:nielseninverse}), reduce \( S_{2,2} \) to \( \Li_4 \) (Proposition \ref{prop:s22reduction}) and determine a basis of the space of mod-products symbols for Nielsen polylogarithms of a given weight (Theorem \ref{thm:nielsendepth}).

\subsection{General relations}

We first recall the differential behaviour of Nielsen polylogarithms, which can be used to verify some of the identities we give later.  The behaviour follows by differentiating the integral \eqref{eqn:nielsendef} defining \( S_{n,p} \).

\begin{Prop}[Derivative, Equation 2.11 in \cite{Ko2}]
	For \( n,p \in \ZZ_{>0} \), the Nielsen polylogarithm \( S_{n,p}(z) \) satisfies the differential equation
	\[
	\Big( z \frac{d}{dz} \Big) S_{n,p}(z) \= S_{n-1,p}(z) \, .
	\]
	We use the convention \( S_{0,p}(z) = \frac{(-1)^p}{p!} \log^{p}(1-z) = 
	\frac{1}{p!} \Li_1^p(z) \), via the iterated integral definition 
	\eqref{eq:nielsenasLi}.
\end{Prop}

Nielsen already established a general inversion and reflection relation for the Nielsen polylogarithms.

\begin{Prop}[Reflection, Section 5.1 in \cite{Ko2}]\label{prop:nielsenreflection}
	For all \( z \in \CC \sm (-\infty, 0] \cup [1, \infty) \), and all \( n, p \in \ZZ_{>0} \), we have
	\begin{align*}
			S_{n,p}(1-z) \={} & \frac{(-1)^p}{n! \, p!} \log^n(1-z)\log^p(z) \\*
					& {} + \sum_{j=0}^{n-1} \frac{\log^j(1-z)}{j!} \Bigg( 
				S_{n-j,p}(1) - \sum_{k=0}^{p-1} \frac{(-1)^k \log^k(z)}{k!} S_{p-k,n-j}(z) \Bigg) \, .
	\end{align*}
	In particular, after neglecting products, one has an expression for \[
		S_{p,n}(z) =  -S_{n,p}(1-z) + S_{n,p}(1) \Mod{products}\, .
	\]
\end{Prop}

This also follows directly from the functoriality, shuffle product and path deconcatenation properties of iterated integrals \cite{Chen}.

\begin{Rem}
	From \eqref{eq:nielsenasLi}, we have \( S_{n,p}(1) = (-1)^p I(0; \{1\}^p, \{0\}^n; 1) =  \zeta(\{1\}^{p-1}, n+1) \), where \( \zeta(k_1,\ldots,k_r) \ceq \Li_{k_1,\ldots,k_r}(1,\ldots,1) \) is a multiple zeta value.  From \cite[Equation 10]{BBB}, the following generating function expansion
	\[
		\sum_{m,n\geq0} x^{m+1} y^{n+1} \zeta(\{1\}^n,m+2) \= 1 - \exp \bigg(\sum_{k\geq2} \frac{1}{k} \big(x^k +y^k - (x+y)^k\big) \zeta(k) \bigg) 
	\]
	shows that this class of MZV's are polynomials in the Riemann zeta values \( \zeta(q) \).
\end{Rem}

\begin{Prop}[Inversion, Section 5.3 in \cite{Ko2}]\label{prop:nielseninverse}
	For all \( z \in \mathbb{C} \sm [0, \infty) \), and all \( n, p \in \ZZ_{>0} \), we have
	\begin{align*}
		S_{n,p}\Big(\frac{1}{z}\Big) \= {} 
			& (-1)^n \sum_{k=0}^{n} (-1)^ k \sum_{m=0}^{k} \frac{\log^m(-z^{-1})}{m!} \binom{n+k-m-1}{k-m} S_{n+k-m,p-k}(z) \\
			& + (-1)^p \bigg( \frac{\log^{n+p}(-z^{-1})}{(n+p)!} + \sum_{j=0}^{n-1} \frac{\log^j(-z^{-1})}{j!} C_{n-j,p} \bigg) \, ,
	\end{align*}
	where \( C_{n,p} \) is some explicit homogeneous polynomial in \( \pi^2 \) and \( S_{a,b}(1) = \zeta(\{1\}^{a-1}, b+1) \), so is a polynomial in Riemann zeta values \( \zeta(q) \).
	
	In particular, 
	\[
	 S_{n,p}\Big(\frac{1}{z}\Big) - (-1)^n S_{n,p}(z) 
	\]
	reduces to lower depth and products.
\end{Prop}

\subsection{\texorpdfstring{Nielsen polylogarithms in weight $\le 4$}
{Nielsen polylogarithms in weight < 4}}
In weights up to $4$, Nielsen polylogarithms give the same class of functions
as the classical polylogarithms $\Li_m$.  More precisely, in weight 2, we only have \( S_{1,1} = \Li_2 \).  At weight 3, \( S_{2,1} = \Li_3 \) and already \( S_{1,2} \) is expressible in terms of \( \Li_3 \) by the reflection in Proposition \ref{prop:nielsenreflection}.

At weight 4, \( S_{1,3} = \Li_4 \) and \( S_{3,1} \) is also expressible in terms of  \( \Li_4 \) by reflection.  The function \( S_{2,2} \) is potentially new, but it too can be reduced to \( \Li_4 \)'s.  In \cite[Section 6]{Ko2}, K\"olbig notes that one can in principle find a `complicated' expression for \( S_{2,2}(z) \) in terms of polylogarithms, by studying the formulae for \( S_{n,p} \) under the 6 anharmonic transformations.  He references an equation in \cite[p. 204]{Le}, from which such a formula could also be derived.

K\"olbig perhaps overstates the complexity of this formula, and of the manner in which it should be derived.  Note that from Proposition \ref{prop:nielseninverse}, we have the following identity
\[
	S_{1,3}(z^{-1}) \= - S_{1,3}(z) +  S_{2,2}(z)-  S_{3,1}(z) \Mod{products}
\]
So immediately \( S_{2,2}(z) \) can be expressed in terms of the other weight 4 Nielsen polylogs and products.  Applying reflection to write \( S_{1,3}(z) = S_{3,1}(z) = \Li_4(z) \Mod{products} \), gives a reduction in depth to \( \Li_4 \).  Wojtkowiak already gives a version of this reduction in \cite[Equation 8.3.7]{WoMHS} for some single-valued analogue of \( S_{2,2} \).

\begin{Prop}[Reduction of \( S_{2,2} \)]\label{prop:s22reduction}
	The function \( S_{2,2}(z) \) can be reduced to the classical \( \Li_4 \), and products of lower weight classical polylogarithms, as follows.  For all \( z \in \CC \sm (-\infty, 0] \cup [1, \infty) \), we have
	\begin{align*}
	S_{2,2}(z) \={} & 
	{} -\Li_4(1-z)
	+\Li_4(z)
	+\Li_4\Big(\frac{z}{z-1}\Big) 
	-\Li_3(z) \log (1-z) \\
	& {} +\frac{1}{4!} \log ^4(1-z)
	 -\frac{1}{3!} \log (z) \log ^3(1-z) \\
	 & {} + \frac{1}{2!} \zeta(2) \log^2(1-z)
	 + \zeta (3) \log (1-z) 
	 +\zeta(4)\,.
	\end{align*}
	
	\begin{proof}
	This identity can be verified by differentiation, and checking the resulting weight 3 combination is identically 0.  Since \( S_{2,2}(0) = 0 \), the constant of integration is fixed to \( \zeta(4) \) to ensure the right hand side also vanishes at \( z = 0 \).
	\end{proof}
\end{Prop}

The same strategy also reduces \( S_{n,n}(z) \) to lower depth Nielsen polylogarithms.  Moreover, we can determine a spanning set for weight \( N \) Nielsen polylogarithms, as follows.

\subsection{\texorpdfstring{Generators for Nielsen polylogarithms}
	{Generators for Nielsen polylogarithms in weight N}}
\label{sec:spanningset}
	
We first state a lemma about the mod-products symbols of Nielsen polylogarithms.
	
\begin{Lem}\label{lem:snpmodsymb}
	The mod-products symbol of \( S_{n,p}(z) \), \( n, p > 0 \), is given by
	\[
		\Symbsh(S_{n,p}(z)) = - (1-z)\wedge z \otimes \Big( (1-z)^{\otimes p-1} 
		\shuffle  z^{\otimes n-1} \Big) \, ,
	\]
	where \( a \wedge b = a \otimes b - b \otimes a \), and \( \shuffle \) is the shuffle product of tensors, recursively defined on words via
	\[
		\big(a \otimes w_1\big) \shuffle \big(b \otimes w_2\big) = a \otimes \big( w_1 \shuffle \big(b \otimes w_2\big) \big) + b \otimes \big( \big( a \otimes w_1 \big) \shuffle w_2 \big) \, ,
	\]
	with the empty word \( 1 \) satisfying \( w \shuffle 1 = 1 \shuffle w = w \).
	
	\begin{proof}
		The cases where \( n = 0 \) or \( p = 0 \) are trivially 0, since the Nielsen polylog reduces to a product.  When \( n = 1 \), or \( p = 1 \), the result follows still by applying the recursive definition, taking into account that certain terms are trivial now. \medskip
		
		For \( S_{n,p}(z) \), only the terms \( \widehat{z_{n+p}} \) and \( \widehat{z_0} \) contribute in the recursion, so
		\begin{align*}
			\Symbsh S_{n,p}(z)
			 \=  & (-1)^p \Symbsh I(0; \{1\}^n, \{0\}^p; z) \\
			 \=  & \begin{aligned}[t]
				(-1)^p \Big(
				 & \Symbsh I(0; \{1\}^p, \{0\}^{n-1}; z) \otimes z \\ 
				 & {} - \Symbsh I(1; \{1\}^{p-1}, \{0\}^n; z) \otimes (1-z)
				\Big) \, .
			\end{aligned}
		\end{align*}
		Modulo products and constants we can deduce \( I(1; \{1\}^{p-1}, \{0\}^n; z) = I(0; \{1\}^{p-1}, \{0\}^n; z) \) by splitting the integration path at 0.  So
		\begin{align}\label{snprecursion}
			\Symbsh S_{n,p}(z)
			\= & \Symbsh S_{n-1,p}(z) \otimes z + \Symbsh S_{n,p-1}(z) \otimes (1-z) \\
			\= & \begin{aligned}[t]
				 -(1-z) \wedge z \otimes \Big( & ((1-z)^{\otimes p-1} \shuffle z^{\otimes n-2}) \otimes z \\
				& {} + ((1-z)^{\otimes p-2} \shuffle z^{\otimes n-1}) \otimes (1-z)  \Big) \, .
			\end{aligned}
		\end{align}
		By the recursive definition of the shuffle product \( \shuffle \), we obtain the result.
	\end{proof}
\end{Lem}

The above identities and mod-products symbol expressions show that \( \mathfrak{S}_3 \) acts on the set of Nielsen polylogarithms of anharmonic ratios.  The symbol of  the Nielsen polylogarithm \( S_{n,p}\) of  weight \( N=n+p \) is of the form \( \Symbsh(\Li_2(z)) \otimes ( (1-z)^{\otimes p-1} \shuffle z^{\otimes n-1} ) \), i.e.~is symmetric in the last \( N - 2 \) factors and antisymmetric in the first 2.  The representation of \( \mathfrak{S}_3 \) on Nielsen is then isomorphic to the representation of \( \mathfrak{S}_3 \) on 2-variable homogeneous polynomials of degree \( N - 2 \), under \( z \leftrightarrow X \) and \( 1-z \leftrightarrow Y \), and then tensored with the sign representation for the \( \Li_2(z) \) factor.  More precisely, under the following identification, relations among \( S_{n,p} \) and their associated polynomials correspond to each other in a bijective manner:
\begin{align*}
	 S_{n,p}(z) &{}\:\mapsto\: \phantom{+}\binom{N-2}{p-1} X^{n-1} Y^{p-1} \, , \\
	 S_{n,p}(1-z) &{}\:\mapsto\: -\binom{N-2}{p-1} Y^{n-1} X^{p-1} \, , \\
	 S_{n,p}\Big(1-\frac{1}{z}\Big) &{}\:\mapsto\: \binom{N-2}{p-1} (Y-X)^{n-1} (-X)^{p-1} \, .
\end{align*}

\begin{Thm}\label{thm:nielsendepth}
	Write \( d = \lfloor (N+1)/3 \rfloor \).  Then the following set forms a basis for the symbols of Nielsen polylogarithms of weight \( N \) modulo products, under the anharmonic ratios:
	\[
		\mathcal{B} \= \{ S_{N-i, i}(z), S_{N-i, i}(1-z), S_{N-i, i}(1-z^{-1}) \}_{i=1}^{d - 1} \:\cup\: \mathcal{X}^{d}_N \, ,
	\]
	where
	\[
		\mathcal{X}^i_N \= \begin{cases}
			\{ S_{N-i, i}(z) \} & \text{if \( N \equiv -1 \pmod{3} \)\,,} \\
			\{ S_{N-i, i}(z), S_{N-i, i}(1-z) \} & \text{if \( N \equiv \phantom{+}0 \pmod{3} \)\,,}  \\
            \{ S_{N-i, i}(z), S_{N-i, i}(1-z), S_{N-i, i}(1-z^{-1}) \} & \text{if \( N \equiv \phantom{+}1 \pmod{3} \)\,.}
		\end{cases}
	\]
	In particular, depth \( d = \lfloor (N+1)/3 \rfloor \) suffices to generate all the Nielsen polylogarithms of weight \( N \) modulo products.
\end{Thm}

\begin{Rem}\label{rem:nielsendepthbehaviour}
	This is the expected depth necessary.  Since the cobracket of depth \( p \) involves only terms of depth \( < p \), one can iterate the cobracket on the wedge factors in each term of \( \delta S_{n,p}(z) \) to determine a lower bound on the depth.  The cobracket of \( \delta S_{n,p}(z) \) involves both \( S_{n-2,p-1}(z) \wedge \{1\}_3 \) and \( S_{n-1,p-2}(z) \wedge \{1\}_3 \).  Note that the highest depth contributions at most come from \( S_{n-2k,p-1}(z) \wedge \{1\}_{2k+1} \) for \( 0<k<\frac{n}{2} \) , so we can very informally say that, in the cases \( (n,p)= (2m-\varepsilon, m) \) for \( \varepsilon \in \{0,1,2\} \), that modulo lower depth \( S_{n,p}(z) \) ``behaves like \( S_{n-2,p-1}(z) \)". (We will see instances of such a behaviour below, e.g.,~for \(S_{3,2} \) and \( S_{5,3} \), and, in a weaker form, for the cases \( (n,p)= (2m-\varepsilon, m) \)  as evidenced in \
Theorem \ref{thm:depthreduce} below.)
	So in weight \( 3M + k \), \( k = 2, 3, 4 \), with \( n = 2M+k-1,p = M+1, \) we can iterate down \( M \) times until we reach \( ((S_{n',p'}(z) \wedge \{ 1\}_3 ) \wedge \cdots ) \wedge \{1\}_3  \), \( n' + p' = k \).
	
	Since there are no (motivic) identities between the single term \( \Li_2(z) \), between the terms \( \Li_3(z) \) and \( S_{1,2}(z) \), or between the terms \( \Li_4(z), S_{2,2}(z), S_{1,3}(z) \), the left hand factor  \( S_{n',p'}(z) \), \( n' + p' =  2, 3, 4 \), cannot simplify to 0.  This shows that depth \( M+1 \) is necessary for weight \( 3M + k \).
	\end{Rem}

\begin{proof}[Proof of Theorem]
	For simplicity, we focus mainly on the case \( N \equiv 1 \pmod{3} \), say \( N = 3M + 1 \).  In this case, \( \mathcal{X}_N^M \) consists of 3 elements, and we claim the full basis is 
	\[
		\mathcal{B} \= \{S_{3M+1-i, i}(z), S_{3M+1-i, i}(1-z), S_{3M+1-i, i}(1-z^{-1}) \}_{i=1}^{M} \, .
	\]
	
	We need to check the image of \( \mathcal{B} \) has full rank, in terms of the basis \( \{X^i Y^{3M - 1 - i}\}_{i=0}^{3M - 1} \) of 2-variable homogeneous polynomials of degree \( 3M-1 \).  Up to scalars, \( S_{3M+1-i,i}(z) \mapsto X^{3M-i} Y^{i-1} \) and \( S_{3M+1-i,i}(1-z) \mapsto X^{i-1} Y^{3M-i} \).  
	So we can project the vector space of 2-variable degree \( 3M-1 \) homogeneous polynomials down to the quotient by the subspace
	\[
		\langle X^{3M-i} Y^{i-1}, X^{i-1} Y^{3M-i} \mid 1 \leq i \leq M \rangle \, .
	\]
	This leaves only basis monomials \( \{X^{i} Y^{3M-1-i}\}_{i=M}^{2M-1} \), 
	and we have to consider whether the projection of the image of 
	\[
		\{ S_{3M+1-i, i}(1-z^{-1}) \}_{i=1}^{M} \, ,
	\]
	has full rank in this quotient space.
	
	Up to scalars, we have
	\[
		S_{3M+1-i,i}(1-z^{-1}) \mapsto \begin{aligned}[t]
			& (Y-X)^{3M-i} X^{i-1} \\
			& \= \sum_{j=0}^{3M-i} (-1)^{3M - i - j} \binom{3M-i}{j} X^{3M-1-j} Y^j \, .
		\end{aligned}
	\]
	In the quotient space only terms \( j= M, \ldots, 2M-1 \) survive, so the matrix of the map is given by
	\[
		A_M = \begin{pmatrix} (-1)^{3M - i - j} \displaystyle\binom{3M-i}{j} \end{pmatrix}_{{1\leq{}i\leq{}M, \: M\leq{}j\leq{}2M-1}} \,.
	\]
	After factoring out scalars from each row and column, reversing the order of the columns, and reindexing, we obtain the matrix
	\[
		A'_M = \begin{pmatrix} \displaystyle\binom{3M-i}{2M - j} \end{pmatrix}_{i,j=1}^{M} \,.
	\]
	Standard evaluations show that
	\[
		\det(A'_M) = \prod_{i=0}^{M-1} \frac{i! \, (i + 2M)!}{(i+M)!^2} > 0 \, ,
	\]
	which proves that \( \mathcal{B} \) is a basis in weight \( 3M + 1 \). \medskip
	
	For the case of weight \( N \equiv 0 \pmod{3} \), say \( N = 3M \), the quotient matrix up to scalars is
	\[
		B'_M = \begin{pmatrix} \displaystyle\binom{3M-1-i}{2M-1 - j} \end{pmatrix}_{i,j=1}^{M-1} \, ,
	\] with
	\[
		\det(B'_M) = \frac{(2n-1)!}{(n-1)!} \prod_{i=0}^{M-1} \frac{i! \, (i + 2M - 1)!}{(M + i)!^2} > 0 \, .
	\]
	Finally, for the case of weight \( N \equiv -1 \pmod{3} \), say \( N = 3M+2 \), the quotient matrix up to scalars is
	\[
	C'_M = \begin{pmatrix} \displaystyle\binom{3M+1-i}{2M+1 - j} \end{pmatrix}_{i,j=1}^{M} \, ,
	\]
	with
	\[
	\det(C'_M) = \prod_{i=0}^{M} \frac{i! \, (i + 2M)!}{(M + i)!^2} > 0 \, .
	\]
	So the set \( \mathcal{B} \) always forms a basis, as claimed.
\end{proof}

\section{Clean single-valued Nielsen polylogarithms}

For the purposes of numerical experimentation with Nielsen and polylogarithm identities, we can apply the `clean single-value' procedure from \cite{Ch-Du-Ga}, to obtain functions which automatically lift mod-products symbol level identities to analytic identities, up to a constant of integration.

\subsection{Cleaning procedure}

To obtain clean single-valued functions we combine the single-valued map \( \sv \) defined in \cite{Br} with a cleaning map \( R_{\bullet} \). On a graded connected Hopf algebra $H=\bigoplus_N H_N$ with (reduced) coproduct \( \Delta' \) and multiplication \( \mu \), the cleaning map in grading $N$ is a linear map $R_N: H_N \to H_N$ defined by the recursion
\[
	R_N = N \id{} - \mu ({\id} \otimes R_\bullet ) \Delta' 
\,,
\]
as explained in \cite{Ch-Du-Ga}.  Functions obtained from the cleaning map complete the product terms in a universal way, and the single-valued map then ensures these combinations are single-valued.

The map \( R_\bullet \) kills products, so when applying it to \( \Delta' \) one only needs to keep the terms where the right-hand factor is not (trivially) a product.  For iterated integrals, such terms are encoded by an analogue of the  \emph{infinitesimal coproduct} \( D \) (cf. \cite[Definition 4.4]{Br-mzv}) given by 
\begin{align*}
	& D I(x_0; x_1, \ldots, x_N; x_{N+1}) \= \\*
	& \sum_{r = 1}^{N-1} \sum_{p = 0}^{N - r} I(x_0; x_1, \ldots, x_p, x_{p+r+1}, \ldots, x_N; x_{N+1}) \otimes I(x_{p}; x_{p+1}, \ldots x_{p+r}; x_{p+r+1}) \, .
\end{align*}
The terms in this can be mnemonically represented as the following type of segments cutting out a semicircular polygon
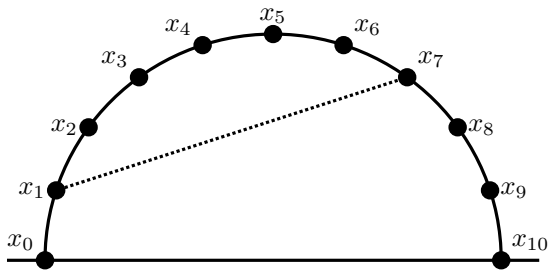
\begin{center}
	\begin{tikzpicture}[style = {very thick}]
	\draw (3,0) arc [start angle=0, end angle = 180, radius=3]
	node[pos=0.0, name=x10, inner sep=0pt] {}
	node[pos=0.1, name=x9, inner sep=0pt] {}
	node[pos=0.2, name=x8, inner sep=0pt] {}
	node[pos=0.3, name=x7, inner sep=0pt] {}
	node[pos=0.4, name=x6, inner sep=0pt] {}
	node[pos=0.5, name=x5, inner sep=0pt] {}
	node[pos=0.6, name=x4, inner sep=0pt] {}
	node[pos=0.7, name=x3, inner sep=0pt] {}
	node[pos=0.8, name=x2, inner sep=0pt] {}
	node[pos=0.9, name=x1, inner sep=0pt] {}
	node[pos=1.0, name=x0, inner sep=0pt] {};
	\draw (-3.5,0) -- (3.5,0);
	\filldraw
	(x0) circle [radius=0.1, fill=black] node[above left] {$x_0$}
	(x1) circle [radius=0.1, fill=black] node[left] {$x_1$}
	(x2) circle [radius=0.1, fill=black] node[left] {$x_2$}
	(x3) circle [radius=0.1, fill=black] node[above left] {$x_3$}
	(x4) circle [radius=0.1, fill=black] node[above left] {$x_4$}
	(x5) circle [radius=0.1, fill=black] node[above] {$x_5$}
	(x6) circle [radius=0.1, fill=black] node[above right] {$x_6$}
	(x7) circle [radius=0.1, fill=black] node[above right] {$x_7$}
	(x8) circle [radius=0.1, fill=black] node[right] {$x_8$}
	(x9) circle [radius=0.1, fill=black] node[right] {$x_9$}
	(x10) circle [radius=0.1, fill=black] node[above right] {$x_{10}$};
	\draw [densely dotted] (x1) edge (x7);
	\end{tikzpicture}
\end{center}
Here the main part containing the integration end points \( x_0 \) and \( x_{N+1} \) gives the left hand factor, and the cut-off segment gives the right hand factor in the tensor. \medskip

We recall a few facts about \( R_\bullet \):
\begin{alignat*}{2}
&R_\bullet \zeta(2k+1) &&\= (2k+1)\zeta(2k+1) \, , \\
&R_\bullet \log(z) &&\= \log(z)\, , \text{ and } \\
&R_\bullet \Li_n(z) &&\= n \Li_n(z) - \log(z) \Li_{n-1}(z) \, .
\end{alignat*}

The single-valued map \( \sv \) is an algebra homomorphism, so we can apply it to each factor in each term separately.  The clean single-valued version of a function \( f \) of weight \( N \) is then given by
\[
	\widehat{f} \ceq \frac{1}{N} ({\sv} \circ R_N)(f) \, ,
\] which will have \( \sv f \) as its main term.  The main result in \cite{Ch-Du-Ga} is that the clean single-valued functions \( \widehat{f_i} \) automatically lift a mod-products symbol identity \( \Symbsh\big( \sum_i \lambda_i f_i\big) = 0  \) to an analytic identity
 \[
 	\sum\nolimits_i \lambda_i \widehat{f_i} \= \mathrm{constant} \, .
 \]

\subsection{\texorpdfstring{Clean \( S_{n,2} \) Nielsen polylogarithms }{Clean S\textunderscore{}\{n,2\} Nielsen polylogarithms }}

One can easily derive a `clean' version of \( S_{n,p} \) for the symbol, for all \( n, p \in \ZZ_{>0} \), because constants go to 0 under the symbol map.  For a `clean' analytic version of \( S_{n,p} \) it is important to retain the constants, but this makes a general formula more difficult to obtain.  We focus only on the clean version of \( S_{n,2}(z) \) for the purposes of this paper. \medskip

We find that only the following terms contribute to the infinitesimal coproduct \( D S_{n,2}(z) \),
\begin{align*}
	& I(0; 1, \{0\}^{n-2j}; z) \otimes I(1; 1, \{0\}^{2j}; 0) \quad j \geq 1 \,, \\
	& I(0; 1, z) \otimes I(1; 1, \{0\}^n; z) \,, \\
	& I(0; 1, 1, \{0\}^{n-1}; z) \otimes I(0; 0; z) \, .
\end{align*}
Illustrated diagrammatically, they are the following segments (respectively the family in the upper left, connecting the first vertex `1' with any of the subsequent `0's,  the long segment from `1' to `$z$' at the bottom, and the short segment at the bottom right ending in `$z$')
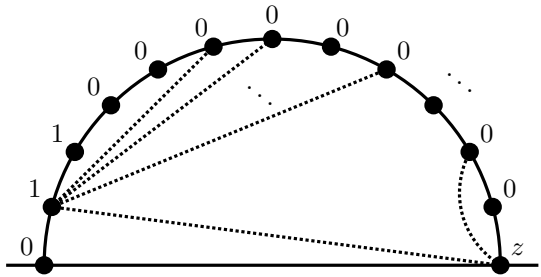
\begin{center}
	\begin{tikzpicture}[style = {very thick}]
	\draw (3,0) arc [start angle=0, end angle = 180, radius=3]
	node[pos=0.000, name=x12, inner sep=0pt] {}
	node[pos=0.083, name=x11, inner sep=0pt] {}
	node[pos=0.167, name=x10, inner sep=0pt] {}
	node[pos=0.250, name=x9, inner sep=0pt] {}
	node[pos=0.333, name=x8, inner sep=0pt] {}
	node[pos=0.417, name=x7, inner sep=0pt] {}
	node[pos=0.500, name=x6, inner sep=0pt] {}
	node[pos=0.583, name=x5, inner sep=0pt] {}
	node[pos=0.667, name=x4, inner sep=0pt] {}
	node[pos=0.750, name=x3, inner sep=0pt] {}
	node[pos=0.833, name=x2, inner sep=0pt] {}
	node[pos=0.917, name=x1, inner sep=0pt] {}
	node[pos=1.0, name=x0, inner sep=0pt] {};
	\draw (-3.5,0) -- (3.5,0);
	\filldraw
	(x0) circle [radius=0.1, fill=black, red] node[above left] {$0$}
	(x1) circle [radius=0.1, fill=black] node[above left] {$1$}
	(x2) circle [radius=0.1, fill=black] node[above left] {$1$}
	(x3) circle [radius=0.1, fill=black] node[above left] {$0$}
	(x4) circle [radius=0.1, fill=black] node[above left] {$0$}
	(x5) circle [radius=0.1, fill=black] node[above left, xshift=2] {$0$}
	(x6) circle [radius=0.1, fill=black] node[above, yshift=2] {$0$}
	(x7) circle [radius=0.1, fill=black] node[above right, xshift=-2] {$0$}
	(x8) circle [radius=0.1, fill=black] node[above right] {$0$}
	(x9) circle [radius=0.1, fill=black] node[above right] {$\ddots$}
	(x10) circle [radius=0.1, fill=black] node[above right] {$0$}
	(x11) circle [radius=0.1, fill=black] node[above right] {$0$}
	(x12) circle [radius=0.1, fill=black] node[above right] {$z$};
	\draw [densely dotted] (x1) edge (x5);
	\draw [densely dotted] (x1) edge (x6);
	\draw [densely dotted] (x1) edge (x8);
	\draw [bend right=40, densely dotted] (x10) edge (x12);
	\draw [densely dotted] (x1) edge (x12);
	\path (x1) -- (x7) node [pos=0.75] {$\ddots$};
	\end{tikzpicture}
\end{center}Moreover, after rewriting 
\begin{alignat*}{3}
	 I(1; 1, \{0\}^n; z) &\= \underbrace{ I(0; 1, \{0\}^n; z)}_{{}= {}-\Li_{n+1}(z)} + \underbrace{I(1; 1, \{0\}^n; 0)}_{{}={} \zeta(n + 1)} \:&&\Mod{products} \,, 
\end{alignat*}
we obtain
\[
	D S_{n,2}(z) \= S_{n-1,2}(z) \otimes \log(z) - \log(1-z) \otimes \Li_{n+1}(z) - \sum_{j = 1}^{\left\lfloor n/2 \right\rfloor} \Li_{n + 1 - 2j}(z) \otimes \zeta(2j + 1)\,.
\]

So the clean version of \( S_{n,2}(z) \) is given by
\begin{equation}\label{clean:sn2}
	\begin{aligned}
	S_{n,2}^\shuffle(z) \ceq & \Big( {\id} - \frac{1}{n+2} \mu  ({\id} \otimes R_\bullet) D \Big) S_{n,2}(z)  \\
	\= & S_{n,2}(z) - \frac{1}{n+2} S_{n-1,2}(z) \log(z) + \frac{n+1}{n+2} \log(1-z) \Li_{n+1}(z) \\
		& - \frac{1}{n+2} \log(1-z) \log(z) \Li_{n}(z)  + \sum_{j=1}^{\lfloor n/2 \rfloor}\frac{2j+1}{n+2} \zeta(2j+1) \Li_{n+1-2j}(z) \, .
\end{aligned}
\end{equation}

In \cite{Ch-Du-Ga}, it was noted already that the clean version of \( \Li_n \) is given by
\[
	\Li_n^\shuffle(z) \ceq \Li_n(z) - \frac{1}{n} \log(z) \Li_{n-1}(z) \, .
\]

\subsection{\texorpdfstring{Single-valued \( S_{n,2} \) Nielsen polylogarithms }{Single-valued S\textunderscore{}\{n,2\} Nielsen polylogarithms }}

Applying Brown's single-valued map to \( S_{n,2} \) produces the following function
\begin{align*}
		\sv S_{n,2}(z) \= &  \Big( S_{n,2}(z)+ (-1)^{n+1} S_{n,2}(\bar{z}) \Big)
	-\log(1-\bar{z}) \big(\Li_{n+1}(z) + (-1)^n\Li_{n+1}(\bar{z}) \big) \\
	& {} -\sum _{j=0}^{n-1} \frac{(-1)^j}{(n-j)!} \log ^{n-j}\big(\abs{z}^2\big) \Big( S_{j,2}(\bar{z}) + \log (1-\bar{z}) \Li_{j+1}(\bar{z}) \Big) \\
	& {} +\sum_{\substack{k=1 \\ \text{\( k \) odd}}}^{n-1} \sum_{j=1}^{n-k} \frac{2
		(-1)^j  \zeta (k+2)}{(n-j-k)!} \Li_j(\bar{z}) \log ^{n-j-k}\big(\abs{z}^2\big) \,.
\end{align*}
Here we again use the convention \( S_{0,p}(z) = \frac{(-1)^p}{p!}\log^p(1-z) \), via \eqref{eq:nielsenasLi}.

Computed already in \cite{Br} is the following single-valued version of \( \Li_n \), obtained from the single-valued map \( \sv \):
\[
	\sv \Li_n(z) \= \Big( \Li_{n}(z) - (-1)^n \Li_{n}(\bar{z}) \Big) - 
	\sum_{j=1}^{n-1} \frac{(-1)^j}{(n - j)!} \Li_j(\bar{z}) \log^{n - 
	j}\big(\abs{z}^2\big) \, ,
\]
although this does not yet satisfy clean functional equations.  The single-valued version of \( \Li_n^\shuffle(z) \), namely
	\begin{align*}
\LiCS_n(z) \ceq & \Big( \Li_n(z) - (-1)^n \Li_n(\bar{z}) \Big) - \frac{1}{n} \Li_{n-1}(z) \log\big(\abs{z}^2\big) \\
		& {} - \sum_{j=1}^{n-1} \frac{j (-1)^j}{n (n-j)!} \Li_{j}(\bar{z}) \log^{n-j}\big(\abs{z}^2\big) \, ,
\end{align*}
 does have this property.

\begin{Rem}
	This single-valued polylogarithm is not simply Zagier's single-valued version (denoted $P_n(z)$ in \cite{Za1}) \[
	 \LiZS_n(z) \ceq \Re_n \Bigg( \sum_{j=0}^{n-1} \frac{2^j B_j}{j!} \log^j\abs{z} \, \Li_{n-j}(z) \Bigg) \, ,
	 \] 
	where \( \Re_n = \Re \) for \( n \) odd, \( \Re_n = \Im \) for \( n \) even and \( B_j \) is the \( j \)-th Bernoulli number.  It is shown in \cite{Ch-Du-Ga} how Zagier's single-valued version \( \LiZS \) and the clean single-valued version \( \LiCS \) are related.
\end{Rem}

Applying the single-valued map to the expression for the clean Nielsen polylogarithm \( S_{n,2}^\shuffle(z) \) in \eqref{clean:sn2}  gives the following clean single-valued Nielsen polylogarithm
\begin{align*}
\SCS_{n,2}(z) \ceq & \Big(S_{n,2}(z)-(-1)^{n+2} S_{n,2}(\bar{z}) \Big)
 -\frac{1}{n+2}\log \big(\abs{z}^2\big) \Big( S_{n-1,2}(z) + \log (1-z) \Li_n(z) \Big) \\[1ex]
& {} + \frac{1}{n+2} \Big((n+1) \log(1-z) - \log(1-\bar{z})\Big) \Big( \Li_{n+1}(z) - (-1)^{n+1} \Li_{n+1}(\bar{z})  \Big) \\[1ex]
 & {} +\sum_{j=1}^{n} \frac{(-1)^j }{n+2} \Big\{(j+1) S_{j-1,2}(\bar{z}) - \Big(j \log (1-z)- \log(1-\bar{z})\Big) \Li_j(\bar{z}) \Big\} \frac{\log ^{n-j+1}\big(\abs{z}^2\big)}{(n-j+1)!}  \\*[1ex]
& {} + \sum_{\substack{k=1 \\ \text{$k$ odd}}}^{n-1}\frac{2 \zeta (k+2)}{n+2} \bigg\{ (k+2) \Li_{n-k}(z) + \sum _{j=1}^{n-k} j (-1)^j \frac{\log ^{n-j-k}\big(\abs{z}^2\big)}{ (n-j-k)!} \Li_j(\bar{z}) \bigg\} \,.
\end{align*}
In particular, the main term is \( \Re_{n+2} S_{n,2}(z) \), just as \( \Re_{n} \Li_{n}(z) \) is the main term for \( \LiCS_n \) and \(  \LiZS_n \). \medskip

Combined with Lemma \ref{lem:snpmodsymb}, we see that \( \SCS_{n,p} \) satisfies mod-products symbol level identities up to an integration constant, i.e. if
\[
	\Symbsh \Big( \sum\nolimits_i \alpha_i S_{n_i,p_i}(x_i) \Big) = 0 \, ,
\] then
\[
	 \sum\nolimits_i \alpha_i \SCS_{n_i,p_i}(x_i) = \mathrm{constant} \, .
\]
In particular, we obtain the following clean-single-valued versions of the inversion and reflection results in Propositions \ref{prop:nielsenreflection} and \ref{prop:nielseninverse}
\begin{align*}
	\SCS_{n,p}(z) &\= -\SCS_{p,n}(1-z) + \frac{1}{p+n} \binom{p+n}{n} \zeta^{\sv}(n + p) \,, \\
	\SCS_{n,p}\bigg(\frac{1}{z}\bigg) &\= (-1)^n \sum_{k=0}^{p-1} (-1)^k \binom{n + k - 1}{k} \SCS_{n+k,p-k}(z) \\
	& \quad\quad + \frac{ \zeta^{\sv}(n+p)}{(n+p)^2} \bigg( (n+p) + p (-1)^p \binom{n+p}{n} \bigg) \, ,
\end{align*}
where \( \zeta^{\sv}(n) \) is the single-valued MZV given by
\[
	\zeta^{\sv}(n) \= \begin{cases}
		2\zeta(n) & \text{\( n \) odd\,,} \\
		0 & \text{\( n \) even\,.}
	\end{cases}
\]

\section{\texorpdfstring{The algebraic \( \Li_2 \), \( \Li_3 \) and \( \Li_4 \) functional equations}{The algebraic Li\textunderscore{}2, Li\textunderscore{}3 and Li\textunderscore{}4 functional equations}}\label{sec:algfe}

We recall the following infinite family of functional equations given in \cite{Ga}, for \( \Li_2 \), \( \Li_3 \) and \( \Li_4 \).  We will use them in later sections, particularly Sections \ref{sec:s32alg}, \ref{sec:s42alg}, \ref{sec:s52alg} and \ref{sec:s53alg}, to provide some additional evidence for the behaviour we expect of Nielsen polylogarithms modulo the classical polylog \( \Li_n \). \medskip

Let \( a,b,c\in \ZZ\sm \{0\}\) be such that $a+b+c=0$, and let \( \{ p_i(t) \}_{i=1}^{r} \) be the roots (counted with multiplicity) of \( x^a(1-x)^b = t \). Furthermore, assume $a>0$
for convenience. Then with the earlier notation that \( \{z\}_n \) means the image of the motivic \( \Li_n(z) \) modulo products, we have
\begin{align*}
&\sum_{i=1}^{r} \{p_i(t)\}_2 \= 0 \, , \\
&\sum_{i = 1}^{r} -\frac{1}{a} \{1-p_i(t)\}_3 + \frac{1}{b} \{p_i(t)\}_3 \= 0 \, , \\
&\sum_{i = 1}^{r} -\frac{1}{a} \{ 1-p_i(t) \}_4 + \frac{1}{b} \{p_i(t)\}_4 + \frac{1}{c} \{ 1 - p_i(t)^{-1} \}_4 \= 0 \, .
\end{align*}

As in \cite{Ga}, we observe the following facts about \( p_i \):
\begin{align*}
 	& \prod_{i=1}^{r} p_i(t) \= \begin{cases} 
 		\pm t & \text{if $ a + b > 0 $\,,} \\
 		\pm 1 & \text{if $ a + b < 0 $\,,}
 	\end{cases} \\
 	&1 - p_i(t) \= \frac{t^{1/b}}{p_i^{a/b}(t)} \,, \text{ up to a \( b \)-th root of unity\,.}
\end{align*}

Note that the case \( (a,b,c) = (1,2,-3) \), or any permutation thereof, can be rationally parametrised over \( \QQ \).  Namely the solutions to 
\[
	x(1-x)^2 \= \frac{(1-t)^2 t^2}{(1- t + t^2)^3}
\] are given by
\[
	p_1(t) \= \frac{1}{1-t+t^2}, \quad p_2(t) \= \frac{t^2}{1-t + t^2}, \quad p_3(t)  \= \frac{(1-t)^2}{1 - t + t^2} \, .
\]
The case \( (a,b,c) = (1,3,-4) \), or any permutation thereof, can also be rationally parametrised but this time only over \( \QQ(i) \).  In fact, for some variable $t$ let
\begin{alignat*}{2}
	U_1 &\= -1 -(1-2 i) t + i t^2 \,,  \quad & U_2 &\= 1+t+\phantom{1}t^2 \,, \\
	U_3 &\= -i -(1+2 i) t - {}\phantom{1}t^2 \,,  & U_4 &\= i + t -i t^2 \,, \\[1ex]
	V &\= \mathrlap{-(U_1 + U_2)(U_1 + U_3)(U_2 + U_3)\,.}
\end{alignat*}
Then the roots of
\[
x(1-x)^3 \=  \prod_{j=1}^4 (U_j^3/V) \, ,
\]
are given by
\[
	p_j(t) \= U_j^3/V \, ,
\]
for \( j = 1, \ldots, 4 \).

\section{Nielsen polylogarithms in weight 5}

In this section we prove one of our main results, stating that ``$S_{3,2}$ evaluated on functional equations of $\Li_2$ is expressible in terms of $\Li_5$". We first corroborate this for the simpler two term relations in Propositions \ref{prop:s32inv} and \ref{prop:s32twoterm} as well as for a family of algebraic functional equations (which are not known to be consequences of the five-term relation) in Proposition \ref{prop:s32alg}, all of which have been already proved in \cite{Ch}, before turning to the basic five-term relation itself (Theorem \ref{thm:s32fiveterm}) and subsequent specialisations like distribution relations as well as ladders and special values. As a further corollary we recover a functional equation for $\Li_5$ recently obtained in \cite{R}.

\smallskip
\paragraph{\em Preconsideration:} Following Section \ref{sec:symbmot}, the 2-part of the motivic cobracket of \( S_{3,2}(z) \) is computed to be
\[
	\delta S_{3,2}(z) = \{z\}_2 \wedge \{1\}_3 \, . 
\]
Since \( \{1\}_3 \neq 0 \), this does not vanish in general, and we cannot reduce \( S_{3,2} \) to \( \Li_5 \) on the motivic level, hence
we should not expect this on a function level, either.

Combinations \( \sum_i \alpha_i [x_i] \) such that \( \sum_i \alpha_i \{x_i\}_2 = 0 \), i.e. functional equations for \( \Li_2 \), will automatically kill \( \delta \sum_i \alpha_i S_{3,2}(x_i) \).  Hence, we expect the Nielsen polylogarithm \( S_{3,2} \) behaves like \( \Li_2 \), modulo \( \Li_5 \)'s and products.

\subsection{Two-term identities}

We can give relatively simple analytic identities for \( S_{3,2} \) under the basic two term identities \( \{z\}_2 + \{z^{-1}\}_2 = 0 \) and \( \{z\}_2 + \{1-z\}_2 = 0 \) for \( \Li_2 \).  These identities are already contained within the reflection and inversion results, and so can be shown without the need to invoke the clean single-valued functions.

\begin{Prop}\label{prop:s32inv}
	For all \( z \in \mathbb{C} \sm [0, \infty) \), the following identity holds
	\begin{align*}
		S_{3,2}(z) + S_{3,2}\Big(\frac{1}{z}\Big) \= & 3 \Li_5(z)-\Li_4(z) \log (-z) -\frac{1}{5!} \log ^5(-z) \\
		& {} +\frac{1}{2!} \zeta (3) \log ^2(-z) + \frac{7}{4} \zeta (4) \log (-z) + \Big( \zeta (5) + \zeta (2) \zeta (3) \Big) \,.
	\end{align*}
	
	\begin{proof}
		This is just the case \( S_{3,2} \) of Proposition \ref{prop:nielseninverse}.
	\end{proof}
\end{Prop}

\begin{Prop}\label{prop:s32twoterm}
	For all \( z \in \CC \sm (-\infty, 0] \cup [1, \infty) \), the following identity holds
	\begin{align*}
	 S_{3,2}(1-z)+S_{3,2}(z) \= & \Li_5(1-z)+\Li_5(1 - z^{-1} )+\Li_5(z) -\Li_4(1-z) \log (z)-\Li_4(z) \log (1-z) \\
	  & -\frac{1}{5!} \log ^5(z)+\frac{1}{4!} \log ^4(z) \log (1-z)-\frac{1}{3!\,2!}  \log ^3(z) \log ^2(1-z) \\
	  & - \frac{1}{3!} \zeta (2) \log ^3(z) +\frac{1}{2!} \zeta (2) \log ^2(z)  \log (1-z) + \zeta (3)  \log (z) \log (1-z) \\ 
	  & + \zeta (4) \log (1-z) - \frac{3}{4} \zeta(4) \log (z) + \Big( \zeta (5) - \zeta (2) \zeta (3) \Big) \,.
	\end{align*}
	
	\begin{proof}
		Use differentiation to reduce to a weight 4 identity, which can also be verified.  The constant of integration is fixed by evaluating as \( z \to 1 \), to obtain \( S_{3,2}(1) - \Li_5(1) = -\zeta(2)\zeta(3) + \zeta(5)\).\medskip
		
		Because of the argument in Section \ref{sec:spanningset}, one can verify that such an identity exists, and that it follows from the reflection and inversion just by computing the mod-products symbol.  From the recursion \eqref{snprecursion} for \( \Symbsh S_{n,p}(z) \), and the reduction of weight 4 Nielsen polylogs to \( \Li_4 \), we have
		\begin{align*}
		\Symbsh S_{3,2}(z) \= & 
		\Symbsh S_{2,2}(z) \otimes z \ 
		 + \  \Symbsh S_{3,1}(z) \otimes (1-z) \\
		\= & \Symbsh\Big( -\Li_4(1-z) + \Li_4(z) - \Li_4(1 - z^{-1}) \Big) \otimes z \\
		 & + \Symbsh \Li_4(z) \otimes (1-z) \, .
		\end{align*}
		We also have \( \Symbsh \Li_5(z) = \Symbsh \Li_4(z) \otimes z \).  Recall too that \( \Symbsh \Li_4(z) = -\Symbsh \Li_4(z^{-1}) \), and that our tensor symbols are written multiplicatively in each slot.  We will drop the notation \( \Symbsh \) from the tensors, for simplicity.  We compute directly that the mod-products symbol of the left hand side is
		\begin{align*}
			& \Symbsh\Big( S_{3,2}(1-z) + S_{3,2}(z) \Big)  \\
			& \= \begin{aligned}[t] & \Big( -\Li_4(1-z) + \Li_4(z) - \Li_4(1-z^{-1}) \Big) \otimes z + \Li_4(z) \otimes (1-z) \\
			& {} + \Big( -\Li_4(z) + \Li_4(1-z) - \Li_4\Big(\frac{z}{z-1}\Big) \Big) \otimes (1-z) + \Li_4(1-z) \otimes z \end{aligned} \\[1ex]
			 & \= \Big( \Li_4(z) - \Li_4(1-z^{-1}) \Big) \otimes z + \Big( \Li_4(1-z) - \Li_4\Big(\frac{z}{z-1}\Big) \Big) \otimes (1-z) \\[1ex]
			& \= \Li_4(1-z) \otimes (1-z) + \Li_4\Big(\frac{z-1}{z}\Big) \otimes \frac{z-1}{z} +  \Li_4(z) \otimes z \, .
		\end{align*}
		This is already the mod-products symbol of the right hand side, i.e. it equals:
		\[
			\Symbsh \Big( \Li_5(1-z)+\Li_5(1 - z^{-1} )+\Li_5(z) \Big) \ .
		\] 
		The remaining terms on the right hand side do not contribute, as they are already non-trivial products.\medskip
		
		Alternatively, the symbol calculus above translates to the following straightforward to check equality of polynomial invariants:
		\begin{align*}
			S_{3,2}(z) + S_{3,2}(1-z) &\:\mapsto\: 3YX^2 - 3XY^2 \,, \\
			\Li_5(z) + \Li_5(1-z) + \Li_5(1-z^{-1}) &\:\mapsto\: X^3 - Y^3 + (Y-X)^3 \,. \qedhere
		\end{align*}
	\end{proof}
\end{Prop}

	By setting \( z = -1 \) in the first identity, and \( z = \frac{1}{2} \) in the second, we recover the following evaluations, contained in Table 2 and Equation 9.9 \cite{Ko2}.
	\begin{align}
\label{eqn:s32m1} \begin{split} S_{3,2}(-1) \= & 
	- \frac{29}{32} \zeta(5) 
	+ \frac{1}{2} \zeta(2) \zeta(3) \,,
 \end{split} \\
\label{eqn:s32half}	\begin{split} S_{3,2}\Big(\frac{1}{2}\Big) \= & 
\Li_5\Big(\frac{1}{2}\Big)
+\Li_4\Big(\frac{1}{2}\Big) \log (2)
+\frac{1}{2} \Big( \frac{1}{16}\zeta (5) -\zeta(2) \zeta (3) \Big)
-\frac{1}{8} \zeta(4) \log (2) 
\\& 
+\frac{1}{2!} \zeta (3) \log ^2(2) 
-\frac{1}{3!} \zeta(2) \log ^3(2) 
+\frac{3}{5!} \log^5(2) \,.
\end{split}
	\end{align}
	
	The existence of these reductions corresponds to the fact that \( \{\frac{1}{2}\}_2 = \{-1\}_2 = 0 \), so that the 2-part of the motivic cobrackets of \( S_{3,2}(\frac{1}{2}) \) and of \( S_{3,2}(-1) \) vanish.

	\subsection{\texorpdfstring{Algebraic \( \Li_2 \) functional equation}{Algerbaic Li\textunderscore{}2 functional equation}} \label{sec:s32alg}
	
	Before dealing with the full five-term identity, we instead consider the simplest case of the algebraic \( \Li_2 \) functional equation from Section \ref{sec:algfe}.  In more general cases, we have greater success reducing these algebraic functional equations, and so this is a good place to introduce them.  This identity was already observed in \cite{Ch}, where it was used to obtain a new functional equation for \( \Li_5 \). Note that the special case \( a=b=1 \) is essentially Proposition \ref{prop:s32twoterm}.
	
	\begin{Prop}[Proposition 7.4.19 in \cite{Ch}]\label{prop:s32alg}
		Let \(a,b,c\in \ZZ\sm\{0\}\), with \( a+b+c=0 \), and let \( \{p_i(t)\}_{i=1}^r \)  be the roots of \( x^a(1-x)^b = t \).  For convenience take \( a > 0 \).  Then the following reduction holds on the level of the mod-products symbol
		\[
			\sum_{i = 1}^{r} S_{3,2}(p_i(t)) \modsh \sum_{i=1}^{r} \bigg\{ \frac{b-a}{b} \Li_5(p_i(t)) + \frac{b}{a} \Li_5(1 - p_i(t)) + \frac{b}{a+b} \Li_5(1 - p_i(t)^{-1}) \bigg\}\,.
		\]
	\end{Prop}

	\begin{Cor}
		We have the clean single-valued identity
		\begin{align*}
		\sum_{i = 1}^{r} \SCS_{3,2}(p_i(t))  - \sum_{i=1}^{r} & \bigg\{ \frac{b-a}{b} \LiCS_5(p_i(t)) + \frac{b}{a} \LiCS_5(1 - p_i(t)) + \frac{b}{a+b} \LiCS_5(1 - p_i(t)^{-1}) \bigg\} \\[1ex]
		& \quad \= \begin{cases}
			2a \zeta(5) & \text{if $b > 0$\,,} \\
			-2b \zeta(5) & \text{if $-a < b < 0$\,,} \\
			-2(a+b) \zeta(5) & \text{if $ b < -a$\,.} \\
		\end{cases} 
		\end{align*}
		
		\begin{proof}
			Consider the limit \( t \to 0 \) and use \( \LiCS_5(0) = \LiCS_5(\infty) = 0 \), \( \LiCS_5(1) = 2 \zeta(5) \) and \( \SCS_{3,2}(0) = 0 \), \( \SCS_{3,2}(1) = \SCS_{3,2}(\infty) = 2\zeta(5) \).
			
			If \( b > 0 \), we obtain roots \( p_i = 0 \) with multiplicity \( a \) and \( p_i = 1 \) with multiplicity \( b \), giving the constant \( 2a\zeta(5) \).  If \( -a < b < 0 \), we obtain roots \( p_i = 0 \) with multiplicity \( a \), giving the constant \( -2b\zeta(5) \).  Finally if \( b < -a \), we obtain roots \( p_i = 0 \) with multiplicity \( a \) and roots \( p_i = \infty \) with multiplicity \( -b-a \), giving the constant \( -2(a+b)\zeta(5) \).
		\end{proof}
	\end{Cor}

	\begin{proof}[Proof of Proposition]
		For simplicity, we again drop the notation \( \Symbsh \) from tensors for these calculations.  We also write \( p_i  = p_i(t) \) for simplicity.  Then \( 1-p_i = t^{1/b} p_i^{-a/b} \) from our observation in Section \ref{sec:algfe}.  So the mod-products symbol of the left hand side is
		\[
			\sum_{i=1}^{r} \Big\{ \Big( - \Li_4(1-p_i) + \Li_4(p_i) - \Li_4(1-p_i^{-1}) \Big) \otimes p_i + \frac{1}{b} \Li_4(p_i) \otimes t - \frac{a}{b} \Li_4(p_i) \otimes p_i \Big\}\,.
		\]
		The mod-products symbol of the right hand side is
		\begin{align*}
			\sum_{i=1}^{r} \bigg\{ 
				& \frac{b-a}{b} \Li_4(p_i) \otimes p_i + \frac{b}{a} \bigg( \frac{1}{b} \Li_4(1-p_i) \otimes t - \frac{a}{b} \Li_4(1-p_i) \otimes p_i \bigg) \\
				& + \frac{b}{a+b} \bigg( \frac{1}{b} \Li_4(1-p_i^{-1}) \otimes t - \frac{a}{b} \Li_4(1-p_i^{-1}) \otimes p_i - \Li_4(1-p_i^{-1}) \otimes p_i \bigg)
				\bigg\}\,.
		\end{align*}
		
		In the difference of the left hand side and right hand side, all terms ending \( \otimes p_i \) cancel.  We are left with
		\[
			\sum_{i=1}^{r} \bigg\{ \frac{1}{b} \Li_4(p_i) - \frac{1}{a} \Li_4(1-p_i) - \frac{1}{a+b} \Li_4(1-p_i^{-1}) \bigg\} \otimes t = 0 \, ,
		\]
		since the expression in brackets is already the algebraic \( \Li_4 \) functional equation.		
	\end{proof}

	\subsection{Five-term identity}
	Our main result is that \( S_{3,2} \), evaluated on the five-term relation, can be reduced to explicit \( \Li_5 \) terms.  On account of the known two-term inversion and reflection identities for \( S_{3,2} \) in Propositions \ref{prop:s32inv} and \ref{prop:s32twoterm} above, we can without loss of generality fully antisymmetrise the five-term relation over \( \mathfrak{S}_5 \). Here and below we use the notation
    \[f\big(\sum\nolimits_{j}\nu_j[x_j]\big) \coloneqq \sum\nolimits_{j}\nu_jf(x_j)\,,\]
    i.e. we extend functions to formal linear 
    combinations $\sum_{j}\nu_j[x_j]$ by linearity.
	
\begin{Thm}[\( S_{3,2} \) of the five-term relation] \label{thm:s32fiveterm} 
	For indeterminates \( x_1,\ldots,x_5 \), we have the following 
    identity between the mod-products symbols of 
	$S_{3,2}$ and $\Li_5$ in \( \bigotimes_{i=1}^5 \QQ(x_1,\ldots,x_5)^{\times} \)
	\begin{equation} \label{eq:fivetermsym}
	\Alt_{5} \bigg(11 S_{3,2}(\CR(x_1, x_2, x_3, x_4))
	+\Li_5\Big(15[\R_1(x_1,\ldots,x_5)] 
	- 9[\R_2(x_1,\ldots,x_5)] 
	+ [\R_3(x_1,\ldots,x_5)]\Big) \bigg) \modsh 0 \, .
 	\end{equation}
	Here 
	\begin{align*}
		\CR(x_1,x_2,x_3,x_4) &\ceq \frac{(x_1-x_3)(x_2-x_4)}{(x_1-x_4)(x_2-x_3)}
	\intertext{is the classical cross-ratio, and \( \R_1, \R_2, \R_3 \) are the following `higher' ratios}
	\R_1(x_1,\ldots,x_5) &\ceq  
	-\frac{(x_{1}-x_{2}) (x_{1}-x_{4}) (x_{3}-x_{5})}
	{(x_{1}-x_{3}) (x_{1}-x_{5}) (x_{2}-x_{4})} 
	\,, \\
	\R_2(x_1,\ldots,x_5) &\ceq  
	-\frac{(x_{1}-x_{2})^2(x_{3}-x_{4}) (x_{3}-x_{5})}
	{(x_{1}-x_{3})  (x_{1}-x_{4}) (x_{2}-x_{3}) (x_{2}-x_{5})} 
	\,, \\
	\R_3(x_1,\ldots,x_5) &\ceq  
	-\frac{(x_{1}-x_{2})^3 (x_{1}-x_{5}) (x_{3}-x_{4})^2 (x_{3}-x_{5})}
	{(x_{1}-x_{3})^3 (x_{1}-x_{4}) (x_{2}-x_{4}) (x_{2}-x_{5})^2} 
	\,.
	\end{align*}
\end{Thm}

\begin{Cor}	\label{cor:s32fiveterm}
	For \( x_1,\ldots,x_5 \in \mathbb{P}^1(\CC) \) we have the following
    identity for the clean single-valued functions
	\begin{equation*} 
	\Alt_{5} \bigg(11 \SCS_{3,2}(\CR(x_1, x_2, x_3, x_4))
	+\LiCS_5\Big(15[\R_1(x_1,\ldots,x_5)] 
	- 9[\R_2(x_1,\ldots,x_5)] 
	+ [\R_3(x_1,\ldots,x_5)]\Big) \bigg) = 0 \, .
	\end{equation*}
	
	\begin{proof}
		By the antisymmetry, the constant in the clean single-valued identity must be 0.
	\end{proof}
\end{Cor}

\begin{Rem}
	With some work, one can in fact give an analytic version of Corollary \ref{cor:s32fiveterm}, which holds for the complex analytic \( S_{3,2} \) and \( \Li_5 \) functions, not just their single-valued versions.
\end{Rem}

\begin{proof}[Proof of Theorem]
	Let us define the polynomials $\pi_1$ and $\pi_2$ by
	$\pi_j = \operatorname{numerator}(1-\R_j)$. One can easily check
	that $G_j\subset\mathfrak{S}_5$ fixes $\pi_j$ (up to sign), where
	\[G_1 = \langle (23), (2435)\rangle\,,
	\quad\quad\quad
	G_2 = \langle(123)\rangle\,.\]
	We have $|G_1|=8$, and $|G_2|=3$.
	It is also easy to check that $\operatorname{numerator}(1-\R_3) = \pi_1\pi_2$.
	We claim that
	\begin{align} \label{eq:pi1pi2vanish}
	\begin{split}
	\Alt_{G_1} 
	\Big(
	\phantom{{}-9} \mathllap{15} 
	 \,\R_1^{\otimes 4} + \R_3^{\otimes 4}\Big) = 0\,,\\
	\Alt_{G_2} 
	\Big(-9\,\R_2^{\otimes 4} + \R_3^{\otimes 4}\Big) = 0 \,.
	\end{split}
	\end{align}
	To see this, let us denote $\Ss_1=\sigma_{(45)}(\R_1)$,
	$\Ss_2=\sigma_{(123)}(\R_2)$, 
	where $\sigma_{g}$ denotes the action of $g\in\mathfrak{S}_5$
	on $\QQ(x_1,\dots,x_5)$.
	Note the identities
	\[\R_3 = \R_1 \cdot\, \Ss_1^2 = \R_2^2 \cdot\, \Ss_2\,.\]
	The group $G_1$ acts in the following way on 
	$\R_1$ and $\Ss_1$:
	\begin{alignat*}{4}
	\sigma_{e}(\R_1) &{}= \R_1\,,\quad
	&\sigma_{(23)(45)}(\R_1) &{}= \R_1^{-1}\,,\quad
	&\sigma_{(24)(35)}(\R_1) &{}= \R_1\,,\quad
	&\sigma_{(25)(34)}(\R_1) &{}= \R_1^{-1}\,,\\
	\sigma_{e}(\Ss_1) &{}= \Ss_1\,,\quad
	&\sigma_{(23)(45)}(\Ss_1) &{}= \Ss_1^{-1}\,,\quad
	&\sigma_{(24)(35)}(\Ss_1) &{}= \Ss_1^{-1}\,,\quad
	&\sigma_{(25)(34)}(\Ss_1) &{}= \Ss_1\,,\\
	\sigma_{(23)}(\R_1) &{}= \Ss_1^{-1}\,,\quad
	&\sigma_{(45)}(\R_1) &{}= \Ss_1\,,\quad
	&\sigma_{(2435)}(\R_1) &{}= \Ss_1^{-1}\,,\quad
	&\sigma_{(5342)}(\R_1) &{}= \Ss_1\,,\\
	\sigma_{(23)}(\Ss_1) &{}= \R_1^{-1}\,,\quad
	&\sigma_{(45)}(\Ss_1) &{}= \R_1\,,\quad
	&\sigma_{(2435)}(\Ss_1) &{}= \R_1\,,\quad
	&\sigma_{(5342)}(\Ss_1) &{}= \R_1^{-1}\,.
	\end{alignat*}
	(That is, the representation of $G_1$ on the multiplicative
	group generated by $\R_1,\Ss_1$ is isomorphic to 
	the standard $2$-dimensional representation of the dihedral group $D_4$.)
	Thus the first identity in~\eqref{eq:pi1pi2vanish} holds by
	\begin{align*}
	\Alt_{G_1} 
	\Big(15 \R_1^{\otimes 4} 
	{}+ \R_3^{\otimes 4}\Big) 
	\= 60 \R_1^{\otimes 4} - 60 \Ss_1^{\otimes 4}
	{}+2(\R_1\Ss_1^2)^{\otimes 4}
	{}+2\Big(\frac{\R_1}{\Ss_1^{2}}\Big)^{\otimes 4}
	{}-2(\Ss_1\R_1^2)^{\otimes 4}
	{}-2\Big(\frac{\Ss_1}{\R_1^{2}}\Big)^{\otimes 4} \= 0
	\,,
	\end{align*}
	where the vanishing is equivalent to 
	the following easily checked polynomial identity
	\[2(X+2Y)^4+2(X-2Y)^4-2(2X+Y)^4-2(2X-Y)^4 \= 60(Y^4-X^4)\,.\]
	Similarly, for $G_2$ we have
	\begin{alignat*}{3}
	\sigma_{e}(\R_2) &{}= \R_2\,,\quad
	&\sigma_{(123)}(\R_2) &{}= \Ss_2\,,\quad
	&\sigma_{(321)}(\R_2) &{}= (\R_2\Ss_2)^{-1} \,,\\
	\sigma_{e}(\Ss_2) &{}= \Ss_2\,,\quad
	&\sigma_{(123)}(\Ss_2) &{}= (\R_2\Ss_2)^{-1} \,,\quad
	&\sigma_{(321)}(\Ss_2) &{}= \R_2\,.
	\end{alignat*}
	Thus
	\begin{align*}
	\Alt_{G_2} 
	\Big(-9 \R_2^{\otimes 4} 
	{}+ \R_3^{\otimes 4}\Big)
	\= -9 \R_2^{\otimes 4}
	{}-9\Ss_2^{\otimes 4}
	{}-9(\R_2^{-1}\Ss_2^{-1})^{\otimes 4}
	{}+(\R_2\Ss_2^2)^{\otimes 4}
	{}+(\R_2^{2}\Ss_2)^{\otimes 4}
	{}+\Big(\frac{\R_2}{\Ss_2}\Big)^{\otimes 4} \= 0
	\,,
	\end{align*}
	since
	\[(X-Y)^4+(2X+Y)^4+(X+2Y)^4 \= 9(X^4+Y^4+(X+Y)^4)\,.\]
	Since $G_j \subset \Aut(\pi_j)$,
	$G_1\subset \Aut(x_{12}x_{13}x_{14}x_{15}x_{23}x_{45})$,
	and $G_2\subset \Aut(x_{12}x_{13}x_{23}x_{45})$, 
	where $x_{ij} \ceq x_i-x_j$, we obtain that
	\begin{align*}
	\Alt_{G_1} \bigg(
	\frac{\pi_1^2}{x_{12}x_{13}x_{14}x_{15}x_{23}x_{45}}\otimes \Big(
	\phantom{{}-9} \mathllap{15} 
	\,\R_1^{\otimes 4} 
	+ \R_3^{\otimes 4}\Big)\bigg) \= 0\,,\\
	\Alt_{G_2} \bigg(
	\frac{\pi_2}{x_{12}x_{13}x_{23}x_{45}}\otimes \Big(-9\,\R_2^{\otimes 4}
	+\R_3^{\otimes 4}\Big)\bigg) \= 0 \,.
	\end{align*}
	Since  $\Symbsh\big( \Li_5(z) \big) \= -(1-z)\wedge z\otimes z^{\otimes 3}$, 
	where again $a\wedge b \= a\otimes b - b\otimes a$, we see 
	from these two identities that the mod-products symbol of the $\Li_5$ 
	part of~\eqref{eq:fivetermsym} is equal to 
	\begin{gather*}
	\Alt_5\Big(-\frac{15}{2}\,\frac{x_{12}x_{13}x_{14}x_{15}x_{23}x_{45}}
	{x_{12}^2x_{14}^2x_{35}^2}\wedge \R_1
	\otimes \R_1^{\otimes 3}
	\;{}+9\,\frac{x_{12}x_{13}x_{23}x_{45}}{x_{13}x_{14}x_{23}x_{25}}\wedge \R_2
	\otimes \R_2^{\otimes 3} \\[1ex]
	{}-\frac{1}{2}\,\frac{x_{12}^3x_{13}^3x_{23}^3x_{45}^3x_{14}x_{15}}{x_{13}^6x_{14}^{2}x_{24}^{2}x_{25}^{4}}
	\wedge \R_3 \otimes \R_3^{\otimes 3}
	\Big)\,.
	\end{gather*}
	Since we are working modulo $2$-torsion, and 
	since $\sigma_{(35)}(\R_1) = -\R_1$, we get that 
	$\Alt_5\big( a\wedge \R_1\otimes \R_1^{\otimes 3} \big) = 0$ whenever
	$\sigma_{(35)}(a) = \pm a$, and hence
	\[\Alt_5\Big( \frac{x_{12}x_{13}x_{14}x_{15}x_{23}x_{45}}
	{x_{12}^2x_{14}^2x_{35}^2}\wedge \R_1
	\otimes \R_1^{\otimes 3}\Big) \= 
	\Alt_5\Big(\frac{x_{23}x_{45}}{x_{24}x_{35}}
	\wedge \R_1 \otimes \R_1^{\otimes 3} \Big)
	\= \Alt_5 \Big( [2534] \wedge \R_1 \otimes \R_1^{\otimes 3} \Big) \,,\]
	where we denote by $[ijkl]$ the cross-ratio 
	$\frac{x_{ik}x_{jl}}{x_{il}x_{jk}}$ (note that 
    after canceling all $\sigma_{(35)}$-invariant terms we
    have added the denominator $x_{24}x_{35}$, which
    is also $\sigma_{(35)}$-invariant, up to sign).
	For the $\R_3$ term we compute
	\[\frac{x_{12}^3x_{23}^3x_{45}^3x_{15}}
	{x_{13}^3x_{14}x_{24}^{2}x_{25}^{4}}
	\wedge \R_3 
	\= \frac{x_{23}^3x_{45}^3}{x_{24}x_{35}x_{25}^{2}x_{34}^{2}}
	\wedge \R_3 
	\= [2435]^2[2534] \wedge \R_3\,.\]
	From these observations we see that the mod-products symbol
	of the $\Li_5$ part of~\eqref{eq:fivetermsym} is equal to
	\begin{align} \label{eq:li5part}
	\Alt_5\Big(
	-\frac{15}{2}\,[2534]\wedge \R_1\otimes \R_1^{\otimes 3}
	\;{}+9\,[1524]\wedge \R_2 \otimes \R_2^{\otimes 3}
	{}-\frac{1}{2}\,[2435]^2[2534] 
	\wedge \R_3 \otimes \R_3^{\otimes 3}
	\Big)\,.
	\end{align}
	
	Next, we introduce the variables $u_j \ceq [j,j+1,j+2,j+3]$, 
	where as before $[ijkl] \ceq \CR(x_i,x_j,x_k,x_l)$ and all indices 
	are written	modulo~$5$.
	The action of $\mathfrak{S}_5$ on $u_j$ gives rise to an
	irreducible $5$-dimensional representation~$V$ (written multiplicatively),
	in which $\sigma_{(12345)}(u_j)=u_{j+1}$ 
	and $\sigma_{(12)}$ acts by
	\[(u_1,u_2,u_3,u_4,u_5)\;\mapsto\; \Big(u_1^{-1},u_2u_4,u_1u_3,u_4^{-1},-\frac{u_5}{u_1u_4}\Big)\,.\]
	 Since $\Symbsh\big(S_{3,2}(z)\big)= -\big((1-z)\wedge z\big)\otimes \big((1-z)\shuffle (z\otimes z)\big)$, we have under the mod-products symbol
	\begin{equation} \label{eq:s32part}
	\Alt_{5} S_{3,2}([1234])
	\= 
	\Alt_{5}\Big( -(u_1\wedge u_2u_5)\otimes (\frac{u_1}{u_2u_5}\shuffle (u_1\otimes u_1)) \Big) \quad {} \in \bigwedge\nolimits^2V\otimes \Sym^3(V)\, .
	\end{equation}
	Therefore,~\eqref{eq:fivetermsym} is an identity in
	the skew-symmetric part of an $\mathfrak{S}_5$-module 
	$\bigwedge\nolimits^2V\otimes \Sym^3(V)$. The 
	$10$-dimensional representation $\bigwedge\nolimits^2V$ decomposes into a direct 
	sum $V_4\oplus V_6$ of a $4$-dimensional and a $6$-dimensional 
	irreducible representation. We can take the basis for $V_4$ to be
	$u_j\wedge u_{j-1}u_{j+1}$, $j=1,\dots,4$,
	and the basis for $V_6$ to be given by $w$ and $\sigma_{(j,j+1)}(w)$, 
	$j=1,\dots,5$, where
	\[w = \frac{1}{5}\Big(u_1\wedge u_2+u_2\wedge u_3+u_3\wedge u_4+u_4\wedge u_5+u_5\wedge u_1\Big)\,.\]
	First, we want to show that~\eqref{eq:li5part} projects trivially onto
	$V_6\otimes \Sym^3(V)$. We compute
	\begin{align*}
	\pr_{V_6}([2534]\wedge \R_1) &= \sigma_{(12)}(w)+\sigma_{(51)}(w)\,,\\
	\pr_{V_6}([1524]\wedge \R_2) &= 2w+\sigma_{(23)}(w)+\sigma_{(34)}(w)+2\sigma_{(51)}(w)\,,\\
	\pr_{V_6}([2435]^2[2534] \wedge \R_3) &= 4w-3\sigma_{(12)}(w)+6\sigma_{(23)}(w)-2\sigma_{(45)}(w)+9\sigma_{(51)}(w)\,.
	\end{align*}
	From this we see that the projection of~\eqref{eq:li5part} onto
	$V_6\otimes \Sym^3(V)$ is equal to
	\begin{align*} 
	\Alt_5 \bigg( w\otimes\Big(
	&\frac{15}{2}\sigma_{(12)}(\R_1)^{\otimes 3}
	{}+\frac{15}{2}\sigma_{(51)}(\R_1)^{\otimes 3}
	{}+18\R_2^{\otimes 3}
	{}-9\sigma_{(23)}(\R_2)^{\otimes 3}
	{}-9\sigma_{(34)}(\R_2)^{\otimes 3}\\
	&{}-18\sigma_{(51)}(\R_2)^{\otimes 3}
	{}-2\R_3^{\otimes 3}
	{}-\frac{3}{2}\sigma_{(12)}(\R_3)^{\otimes 3}
	{}+3\sigma_{(23)}(\R_3)^{\otimes 3}
	{}-\sigma_{(45)}(\R_3)^{\otimes 3}
	{}+\frac{9}{2}\sigma_{(51)}(\R_3)^{\otimes 3}
	\Big)\bigg)\,.
	\end{align*}
	Factorising in terms of $u_j$ and switching to additive 
	notation with indeterminate $U_j$ corresponding to~$u_j$, 
	we can rewrite the last expression as
	\begin{align*} 
	\Alt_5 \bigg( w\otimes\Big(
	&\frac{15}{2}(U_1-U_5)^{3}
	{}+\frac{15}{2}(U_4-U_3)^{3}
	{}+18(U_1-U_2-2U_5)^{3}
	{}-9(U_1+U_5)^{3}
	{}-9(-U_2-U_3-2U_5)^{3}\\
	&{}-18(U_1-U_2+U_3+U_5)^{3}
	{}-2(-2U_2-U_4-3U_5)^{3}
	{}-\frac{3}{2}(3U_1-2U_2+2U_4-3U_5)^{3}\\
	&{}+3(U_2-U_4+3U_5)^{3}
	{}-(-U_2+U_4-3U_5)^{3}
	{}+\frac{9}{2}(-2U_2+U_3-U_4+2U_5)^{3}
	\Big)\bigg)\,.
	\end{align*}
	Note that for any dihedral permutation $g\in D_5=\langle (12345),(12)(35)\rangle\subset \mathfrak{S}_5$
	we have $\sigma_g(w) = \chi(g) w$, where $\chi\colon D_5\to\{\pm 1\}$
	takes value $1$ on rotations and $-1$ on reflections.
	From this we see that \linebreak $\Alt_5 \big({w\otimes v}\big) = \frac{1}{10}\Alt_5\big( w\otimes (\Alt_{D_5}v)\big)$, where $\Alt_{D_5}(v) = \sum_{g\in D_5}\chi(g)\sigma_g(v)$.
	The dihedral group $D_5$ acts on $U_j$ as on the vertices of a regular 
	pentagon, and it is not hard to see 
	that $\Alt_{D_5}U_j^3=\Alt_{D_5}U_iU_jU_k=0$,
	hence the image of $\Alt_{D_5}$ on cubic polynomials is two-dimensional
	and it is spanned by $\Alt_{D_5}U_1U_{j}^2$ for $j=2,3$.
	From this we get that the projection of~\eqref{eq:li5part} 
	onto $V_6\otimes \Sym^3(V)$ is equal to
	\begin{align*} 
	192\Alt_5\Big( w\otimes\Big(-3U_1U_2^2+U_1U_3^2\Big)\Big)\,.
	\end{align*}
	On the other hand, one can easily check that
	\[w+\sigma_{(23)}(w)-\sigma_{(24)}(w)-\sigma_{(243)}(w) = 0\,,\]
	and therefore 
	\begin{align*} 
	0 &\= \Alt_5 \Big( (-w-\sigma_{(23)}(w)+\sigma_{(24)}(w)
	+\sigma_{(243)}(w))\otimes (U_1-U_2)^3 \Big) \\
	&\= \Alt_5 \Big( w \otimes \Big(-(U_1-U_2)^3+(U_1-U_5)^3
	-(-U_1+U_2-U_3+U_5)^3+(-U_1-U_3-U_5)^3\Big) \Big) \\
	&\= 6\Alt_5\Big( w\otimes \big(-3U_1U_2^2+U_1U_3^2\big)\Big)\,.
	\end{align*}
	
	This shows that~\eqref{eq:li5part} is an element of $V_4\otimes \Sym^3(V)$.
	Next, we compute the projections onto $V_4$:
	\begin{align*}
	\pr_{V_4}([2534]\wedge \R_1) &\= 
	\frac{2}{5}u_1\wedge u_5u_2
	+\frac{1}{5}u_2\wedge u_1u_3
	+\frac{2}{5}u_4\wedge u_3u_5\,,\\
	\pr_{V_4}([1524]\wedge \R_2) &\= 
	\frac{2}{5}u_2\wedge u_1u_3
	+\frac{2}{5}u_3\wedge u_2u_4
	+\frac{1}{5}u_4\wedge u_3u_5\,,\\
	\pr_{V_4}([2435]^2[2534] \wedge \R_3) &\= 
	\frac{26}{5}u_1\wedge u_5u_2
	+\frac{13}{5}u_2\wedge u_1u_3
	+\frac{16}{5}u_3\wedge u_2u_4
	+2 u_4\wedge u_3u_5\,.
	\end{align*}
	Then~\eqref{eq:li5part} is equal to
	\begin{align*}
	\frac{1}{5}\Alt_5 \bigg( (u_1\wedge u_2u_5)\otimes\Big(
	&{}-15(U_4-U_5)^3 -\frac{15}{2}(U_3-U_4)^3 - 15(U_1-U_2)^3
	{}+18(U_5-U_1-2U_4)^3 \\
	&{}+18(U_4-U_5-2U_3)^3 + 9(U_3-U_4-2U_2)^3
	{}-13(-2U_2-U_4-3U_5)^3 \\
	&{}-\frac{13}{2}(-2U_1-U_3-3U_4)^3
	{}-8(-2U_5-U_2-3U_3)^3-5(-2U_4-U_1-3U_2)^3     
	\Big) \bigg) \,.
	\end{align*}
	Let us denote the parenthesised polynomial by $P(U_1,\dots,U_5)$.
	Then, combining this identity with~\eqref{eq:s32part} we get
	that the mod-products symbol of the left-hand-side of~\eqref{eq:fivetermsym} 
	is equal to
	\begin{equation} \label{eq:5termfinalstep}
	\Alt_5 \bigg( (u_1\wedge u_2u_5)\otimes\Big(
	\frac{1}{5}P(U_1,U_2,U_3,U_4,U_5)-33(U_1-U_2-U_5)U_1^2
	\Big) \bigg) \,.
	\end{equation}
	The term $(u_1\wedge u_2u_5)$ is skew-symmetric under 
	the subgroup $\mathfrak{S}_4\subset \mathfrak{S}_5$ 
	that permutes $x_1,\dots,x_4$, therefore 
	\[
	\Alt_5 \Big( (u_1\wedge u_2u_5)\otimes v \Big) \= 
	\Alt_5 \Big( (u_1\wedge u_2u_5)\otimes 
	\frac{1}{24}(\Sym_{\mathfrak{S}_4}v) \Big) \,.
	\]
	In view of this we compute
	\begin{align*}
	& \frac{1}{24}\Sym_{\mathfrak{S}_4}
	\Big(\frac{1}{5}P(U_1,U_2,U_3,U_4,U_5)
	-33(U_1-U_2-U_5)U_1^2\Big) \\[1ex]
	& \= 12\sum_{j\Mod{5}} \big(-U_j^3+U_jU_{j+1}(U_j+U_{j+1}+U_{j+2})
	-U_jU_{j+2}(U_j+U_{j+2})\big)\,.
	\end{align*}
	Finally, combining this with
	\[
	\sum_{j\Mod{5}}(u_j\wedge u_{j-1}u_{j+1}) \= 
	\sum_{j\Mod{5}}(u_j\wedge u_{j+1} - u_{j-1}\wedge u_{j}) \= 0\,
	\]
	we get that~\eqref{eq:5termfinalstep} is equal to
	\begin{align*} 
	& 12 \Alt_5 \Big( (u_1\wedge u_2u_5)\otimes
	\sum_{j\Mod{5}} \big( -U_j^3+U_jU_{j+1}(U_j+U_{j+1}+U_{j+2})-U_jU_{j+2}(U_j+U_{j+2}) \big)
	\Big) \\
	& \= 12 \Alt_5 \Big( \sum_{j\Mod{5}}(u_j\wedge u_{j-1}u_{j+1})
	\otimes
	\big(-U_1^3+U_1U_{2}(U_1+U_{2}+U_{3})-U_1U_{3}(U_1+U_{3})\big) \Big)
	=0\,,
	\end{align*}
	concluding the proof of~\eqref{eq:fivetermsym}.
\end{proof}

\begin{Rem}
	If we utilise the results in \cite{Br-rep}, one can potentially obtain a simpler proof of the \( S_{3,2} \) of five-term reduction in \eqref{eq:fivetermsym}.  From \eqref{eq:pi1pi2vanish}, we see that the mod-products symbol of
	\[
	\Alt_{5} \Big(\Li_5\big(15[\R_1(x_1,\ldots,x_5)] 
	- 9[\R_2(x_1,\ldots,x_5)] 
	+ [\R_3(x_1,\ldots,x_5)]\big) \Big)
	\]
	lands in the space of (integrable) tensors of iterated integrals on \( \mathfrak{M}_{0,5} \), since the contributions to each of the irreducibles \( \pi_1 \) and \( \pi_2 \) cancel.
	
	Since the 2-part of the deconcatenation cobracket (or functional cobracket, rather than motivic cobracket) also vanishes (it is a combination of depth 1 polylogarithms), Theorem 56 in \cite{Br-rep} implies that it must be expressible in terms of Nielsen polylogarithms of weight 5, with cross-ratio arguments.  The \( \Alt_{5} \) symmetry, and the fact that \( S_{3,2} \) and \( \Li_5 \) suffice by Theorem \ref{thm:nielsendepth}, means that it must equal
	\[
	\Alt_5 \Big( c_1 S_{3,2}(\CR(x_1,\ldots,x_4)) + c_2 \Li_5(\CR(x_1,\ldots,x_4)) \Big) \, , 
	\]
	for some \( c_1, c_2 \) to be determined.  Moreover, \( \Alt_5 \Li_5(\CR(x_1,\ldots,x_4)) \= 0 \) already, by the \( \Li_5 \) inversion.  To fix the coefficient \( c_1 \), we can compare the coefficients of \( (x_1 - x_2) \wedge (x_2 - x_3) \otimes (x_1 - x_2)^{\otimes 3} \).
\end{Rem}

If we alternate \eqref{eq:fivetermsym} over a sixth point, the Nielsen term vanishes
as it depends on only four points. As a corollary, 
 we obtain the following non-trivial functional 
equation for $\Li_5$ that was previously found by the third 
author as a result of an extensive computer search.
\begin{Cor}[{\cite[Theorem 5.13]{R}}]\label{cor:l5fe}
	For any $x_1,\dots,x_6\in\mathbb{P}^1(\CC)$ we have
	\begin{align*}
	\Alt_{6} \Big( 
	L_5\big(15[\R_1(x_1,\ldots,x_5)]
	-9[\R_2(x_1,\ldots,x_5)]
	+[\R_3(x_1,\ldots,x_5)] \big) \Big) \= 0 \,.
	\end{align*}
	Here we can choose either \( L_5 = \LiCS_5 \), the clean single-valued polylogarithm above, or \( L_5 = \LiZS_5 \), Zagier's single-valued polylogarithm.
\end{Cor}

Since every rational functional equation for \( \Li_2 \), i.e. any relation \( \sum_i n_i \LiZS_2(F_i(x_1,\ldots,x_k)) \= 0 \) with \( F_i \in \QQ(x_1,\ldots,x_k) \),  follows from the five-term relation, we see that \( S_{3,2} \) satisfies dilogarithm functional equations modulo \( \Li_5 \) terms.  

\begin{Cor}[Distribution relations for \( S_{3,2} \)]
	The Nielsen polylogarithm satisfies the distribution relations
	\[
		\frac{1}{n} S_{3,2}(z^n) - \sum_{\lambda^n = 1} S_{3,2}(\lambda z) \= 0 \Mod{$\Li_5$,  products} \, ,
	\]
	with algorithmically determinable \( \Li_5 \) terms.
	
	\begin{proof}
		Wojtkowiak \cite{Wo} gives an algorithm which reduces any  functional equation in a single variable~\( z \), with arguments in \( \CC(z) \), to a combination of five-term relations (for a condensed version of the proof see also \cite{Za3}, Proposition 4).  From this, we can write the \( \Li_2 \) distribution relation as a sum of five-term relations, and obtain the corresponding statement for \( S_{3,2} \).
	\end{proof}
\end{Cor}

\begin{Cor}\label{cor:s32eval}
	Any \( \Li_2 \) evaluation which is accessible via the five-term relation (i.e. following explicitly from the five-term relation, see \cite{Ki}) can be upgraded to an \( S_{3,2} \) evaluation, with explicit \( \Li_5 \) terms.  In particular, the Nielsen 
	polylogarithm $S_{3,2}(z)$ can be evaluated in terms of $\Li_5$ whenever
	$\Li_2(z)$ can be evaluated in terms of products of logs. 
\end{Cor}

\subsection{Special values and ladders}

Corollary \ref{cor:s32eval} gives us the previous formulae for $S_{3,2}(1) = \zeta(1,4) = -\zeta(2)\zeta(3) + 2 \zeta(5) $, and for $S_{3,2}(-1)$, $S_{3,2}(\tfrac{1}{2})$ given in \eqref{eqn:s32m1}, and \eqref{eqn:s32half} above.  It also gives the following new identities involving the golden ratio, and ladders involving \( \frac{1}{3} \) or \( \sqrt{2} - 1 \). \medskip

\paragraph{\bf Values involving the golden ratio:} 
	Recall the following evaluation  involving the golden ratio \( \phi = \frac{1}{2} (1 + \sqrt{5}) \) for \( \Li_2 \) (see \cite[Equations 1.20 and 1.21]{Le}, or \cite[Section 1.1]{Za3}):
\begin{alignat*}{2}
    \Li_2(-\phi ) &\= -\frac{3}{5} \zeta(2)-\log ^2(\phi) \,.
\end{alignat*}
We have the following evaluation for the clean single-valued Nielsen polylogarithm \( \SCS_{3,2} \):
    \begin{alignat*}{7}
	\SCS_{3,2}(-\phi )\=\frac{1}{33}\LiCS_5\big(&&{}-8 [\phi ^{-3}]{}-{}&&243 [\phi ^{-1}]{}-{}&&219 [-\phi ]{}+{}&&8 [-\phi ^3] \big) {}+{}&&\zeta (5) \,.
    \end{alignat*}
     For the complex analytic Nielsen polylogarithm \( S_{3,2} \) we have:
    \begin{align*}
	    S_{3,2}(-\phi )\={}&
	    \frac{1}{33}\Li_5\big({}-8 [\phi ^{-3}]{}-{}243 [\phi ^{-1}]{}-{}219 [-\phi ]{}+{}8 [-\phi ^3] \big)
	    -2 \Li_4(-\phi )  \log (\phi )\\ &
	    +\frac{1}{2}\zeta (5)
	    -\frac{325}{22} \zeta (4) \log (\phi)
	     -\zeta(3)\Li_2(-\phi)
	    -\frac{16}{11} \zeta (2) \log ^3(\phi )
	    +\frac{8 }{15}\log ^5(\phi ) \,.
    \end{align*}

Note that the coefficient of \( \zeta(5) \) in the analytic identity is \( \frac{1}{2} \) of the coefficient in the single-valued identity.  The \( \zeta(5) \) appearing in the single-valued identity is really \( \frac{1}{2} \zeta^{\sv}(5) = \frac{1}{2} \LiCS_5(1) \), and this coefficient becomes manifest when passing to the analytic identity. 

\smallskip
There are three related evaluations for $\phi^{-2}$, $\phi^{-1}$ and $-\phi^{-1}$ which we reproduce for the sake of completeness in Appendix \ref{app:reductiongoldenratio}.

\medskip

\paragraph{\bf Ladder with \( \frac{1}{3} \):} Corresponding to the evaluation 
\[
	\Li_2\bigg(\bigg[\frac{1}{9} \bigg]-6 \bigg[\frac{1}{3}\bigg] \bigg) \= -2 \zeta(2) + \log^2(3) \,,
\]
we have a clean single-valued identity for \( \SCS_{3,2} \) which states
\begin{align*}
	\SCS_{3,2}\bigg(\bigg[\frac{1}{9} \bigg]-6 \bigg[\frac{1}{3}\bigg] \bigg) \= \LiCS_5\bigg( & \frac{1}{16}  \bigg[\frac{1}{9}\bigg]+\frac{21}{2} \bigg[\frac{1}{4}\bigg]+36 \bigg[\frac{1}{3}\bigg]
	-100 \bigg[\frac{1}{2}\bigg] \\ {}& -60 \bigg[\frac{2}{3}\bigg]+\frac{69}{2} \bigg[\frac{3}{4}\bigg]-2 \bigg[\frac{8}{9}\bigg] \bigg) + \frac{1855 }{12}\zeta (5) \,.
\end{align*}
We also have the corresponding analytic identity for $S_{3,2}\big(\big[\frac{1}{9} \big]-6 \big[\frac{1}{3}\big] \big) $ which we reproduce in Appendix \ref{app:reductions32ladder}.

\medskip
\paragraph{\bf Lewin's ladder with \( \alpha = \sqrt{2}-1 \):} Corresponding to one of Lewin's ladders (Equation 94(a) in \cite{Lewindilog})
\[
	\Li_2(\alpha^2)-4 \Li_2(\alpha) \= \log
	^2(\alpha) - \frac{3}{2} \zeta(2) \,,
\]
we have the clean single-valued identity
\begin{alignat*}{5}
	\SCS_{3,2}(\alpha ^2)-4 \SCS_{3,2}(\alpha ) \= 
	 \frac{1}{117} \LiCS_5\bigg(
	&&14 \bigg[{-}\frac{\beta ^5}{\alpha}\bigg]
	&&{}+{}28 \big[\alpha  \beta ^4\big]
	&&{}+{}62 \big[\alpha  \beta ^3\big]
	&&{}-{}252 \big[-\beta ^3\big] &
	\\&&
	{}+{}44 \bigg[\frac{\beta ^3}{\alpha}\bigg]
	&&{}-{}574 \big[\alpha  \beta ^2\big]
	&&{}-{}252 \bigg[\frac{\beta ^2}{\alpha }\bigg]
	&&{}-{}22 \bigg[{-}\frac{ \beta ^2}{\alpha}\bigg] &
	\\&&
	{}+354 \big[\alpha  \beta \big]
	&&{}-{}252 \big[-\alpha  \beta \big]
	&&{}-{}2488 \big[\beta \big]
	&&{}-{}2896 \big[-\beta \big] &
	\\&&
	{}+70 \bigg[\frac{\beta}{\alpha }\bigg]
	&&{}+{}28 \bigg[{-}\frac{ \beta}{\alpha ^{2}} \bigg]
	&&{}+{}1260 \big[\alpha \big]
	&&{}+{}1824 \big[-\alpha \big] &\bigg)
	-\frac{659}{117} \zeta (5)
	 \,,
\end{alignat*}
where we write \( \beta = \sqrt{2} \) for convenience. From this, an analytic identity can again be derived.

\def \om {\omega}

\subsection{\texorpdfstring{Evaluation of \( S_{3,2} ([\om^2] + 2[\om]) \) for \( \om \) a root of the polynomial \( u^3 + u^2 -1 \)}{Evaluation of S\textunderscore{}{3,2}([w\textasciicircum{}2] + 2[w]) for w a root of the polynomial u\textasciicircum{}3 + u\textasciicircum{}2 - 1}}\label{s32ladderomega}
By combining different functional equations of \( S_{3,2} \) we can give another ladder evaluation. Let \( \om \) be a root of the polynomial \( u^3 + u^2 -1 \). Then we use the depth reduction of \( S_{3,2} \) applied to the following  algebraic \( \Li_2 \) functional equation \( \big[t(1-t) \big]  + \big[-\frac {t}{(1-t)^2} \big] 
+ \big[-\frac {1-t}{t^2} \big] \) from the three roots of \( x^2(1-x)^{-3} = \frac{t^2(1-t)^2}{(1-t+t^2)^3} \) (case \( a=2, b=-3 \) in Proposition \ref{prop:s32alg} above) and specialise to \( t=-\om \). The three arguments turn actually out to be equal to \( -\om^{-1}, \om^5 \) and  \( -\om^{-4} \), respectively. Now using  further algebraic relations for \( \om \) like \( 1+\om^4 = \om^{-1} \) and  \( 1-\om^5 = \om \) together with inversion and reflection relations as well as the duplication relation, we can rewrite the given combination as \( -\frac{1}{2} \SCS_{3,2}([\omega^2] + 2[\omega]) \) modulo explicit \( \LiCS_5 \) terms. Moreover, if we take the {\em real} embedding of \( \om \) the same ladder holds even for \( S_{3,2} \) modulo \( \Li_5 \).

\section{Identities in weight 6}
In this section we first show that the depth~3 integral  \( S_{3,3} \) can be reduced to \( S_{4,2}\) and \( S_{5,1}=\Li_6 \) (Proposition \ref{prop:s33reduce}).
Moreover, in analogy to the situation for \( S_{3,2} \) and functional equations of \( \Li_2 \) above, we expect that  \( S_{4,2} \),  evaluated on any functional equation of \( \Li_3 \), can itself be depth reduced  to  \( \Li_6 \), at least modulo products.
As evidence we show the corresponding statement for the three term relation (Proposition \ref{prop:s42three}) and for an algebraic family of functional equations (Proposition \ref{prop:s42alg}). As a consequence, we evaluate  \( S_{3,3} \) at certain roots of unity (Corollary \ref{s33rootsof1}), and we match the coproduct for \( S_{3,3}(-1) \) as well as for \( S_{4,2} \) evaluated at \( 1 \), \( \frac{1}{2}\)  and \( -\phi^{-2} \), where \( \phi \)  denotes again the golden ratio.

\smallskip
\paragraph{\em Preconsideration:} The 2-part of the  motivic coboundary of \( S_{4,2}(z) \) and \( S_{3,3}(z) \) is computed to be
\begin{align*}
	\delta S_{4,2}(z) &\= -\{z\}_3 \wedge \{1\}_3 \,, \\
	\delta S_{3,3}(z) &\=  - \{z\}_3 \wedge \{1\}_3  + \{1-z\}_3 \wedge \{1\}_3 \,.
\end{align*}
This suggests that \( S_{4,2}(z) \) should behave like \( \Li_3 \) modulo \( \Li_6 \), and gives a candidate for reducing \( S_{3,3} \) to \( S_{4,2} \) by matching their cobrackets.

\subsection{\texorpdfstring{Depth reduction of \( S_{3,3} \)}{Depth reduction of S\textunderscore{}\{3,3\}}}

We know that \( S_{3,3}(z) \) can be reduced to \( S_{4,2} \) and \( \Li_6 \), but from the motivic cobracket evaluated above we expect the  combination
\[
S_{3,3}(z) + (S_{4,2}(1-z) - S_{4,2}(z))  
\]
 in particular to reduce modulo products to \( \Li_6 \)'s. Indeed, we find
 
\begin{Prop}\label{prop:s33reduce}
	The following identity holds for all \( z \in \CC \sm (-\infty, 0] \cup [1, \infty) \), and reduces the Nielsen polylogarithm \( S_{3,3} \) to depth \( \leq 2 \)
	\begin{align*}
	S_{3,3}(z) \={}& 
	 S_{4,2}(z)
	-S_{4,2}(1-z)
	 +\Li_6(1-z)
	 -\Li_6(z)
	 -\Li_6\Big(\frac{z}{z-1}\Big) 
	 -S_{3,2}(z) \log (1-z)
	 \\&
	  {}+\Li_5(z) \log (1-z)
	  -\Li_5(1-z) \log(z)
	   -\frac{1}{2!}\Li_4(z) \log ^2(1-z) 
	   \\& 
	    {}-\frac{1}{6!} \log ^6(1-z)
	     +\frac{1}{5!} \log (z) \log ^5(1-z)
	   -\frac{1}{2! \, 4!}\log ^2(z) \log ^4(1-z) 
	    \\& 
	    {}-\frac{1}{4!} \zeta (2) \log ^4(1-z)
	    +\frac{1}{3!}\zeta (2) \log (z) \log ^3(1-z)
	    +\frac{1}{2!} \zeta (3) \log (z) \log ^2(1-z)
	     \\&
	     {}+\zeta (4) \log (z) \log(1-z)
	    -\frac{3}{4 \cdot 2!} \zeta (4) \log ^2(1-z)
	      +\zeta (5) \log (z)
	     \\&
	      {}+ \Big(\zeta (5) -\zeta (2) \zeta (3) \Big) \log (1-z)
	      - \Big( \frac{1}{4}\zeta (6) + \frac{1}{2}\zeta (3)^2 \Big) \,.
	\end{align*}
	
	\begin{proof}
		Differentiate, and use weight 5 identities to see the result is constant.  By taking \( z \to 0 \)  we can fix the constant as \( \ -S_{4,2}(1) + \Li_6(1) = \frac{1}{4} \zeta(6) + \frac{1}{2} \zeta(3)^2 \). \medskip
		
		Also, one can check the polynomial invariant from Section \ref{sec:spanningset}.  Whenever an identity among  these polynomial invariants holds,  the corresponding mod-products symbol identity is also true, and the result is necessarily derivable from inversion (Proposition \ref{prop:nielsenreflection}) and reflection (Proposition \ref{prop:nielseninverse}).  We have
		\begin{align*}
			& S_{3,3}(z) - \Big( S_{4,2}(z)-S_{4,2}(1-z) +\Li_6(1-z)-\Li_6(z)-\Li_6\Big(\frac{z}{z-1}\Big) \Big) \\
			 \mapsto\:{} & 6 X^2 Y^2 - \big(4 X^3 Y + 4 X Y^3 - Y^4 -X^4 - (X-Y)^4 \big) = 0 \, . \qedhere
		\end{align*}
	\end{proof}
\end{Prop}

By specialising the proposition to \( z = \tfrac{1}{2} \), where the \( S_{4,2} \) terms cancel, and to \( z =-1 \), respectively, we get the following.

\begin{Cor}\label{s33spec}
\begin{enumerate}
\item One has the reduction

\begin{align*}
	S_{3,3}\Big(\frac{1}{2}\Big) \= &
	\Li_5\Big(\frac{1}{2}\Big) \log (2)
	+\frac{1}{2} \Li_4\Big(\frac{1}{2}\Big) \log^2(2)
	+\frac{23}{32} \zeta(6)
	-\frac{1}{2}\zeta (3)^2
	-\frac{63}{32} \zeta (5) \log (2) 
	\\& 
	+\frac{1}{2} \zeta(2) \zeta(3) \log (2) 
	+\frac{1}{2!} \zeta(4) \log ^2(2)
	-\frac{1}{4!} \zeta(2) \log ^4(2)
	+\frac{8}{6!} \log^6(2) \, .
\end{align*}
\item
We can reduce \( S_{3,3}(-1) \) to \( S_{4,2}(-1) \)  and \( S_{4,2}(\tfrac{1}{2}) \) modulo products as
\begin{equation}
\label{eqn:s33m1}
\begin{aligned} 
	S_{3,3}(-1) \= & 
	S_{4,2}(-1)
	-S_{4,2}\Big(\frac{1}{2}\Big)
	 + 2 \Li_6\Big(\frac{1}{2}\Big)
	 +\Li_5\Big(\frac{1}{2}\Big) \log(2)
	  -\frac{41}{32} \zeta(6)
	 -\frac{1}{2}\zeta (3)^2
	 \\& 
	{}-\frac{1}{2} \Big( \frac{1}{16} \zeta (5) - \zeta(2)\zeta (3) \Big) \log (2)
	 +\frac{1}{8 \cdot 2!} \zeta(4) \log ^2(2) 
	 \\&
	 {}-\frac{1}{3!} \zeta (3) \log^3(2) 
	  +\frac{1}{4!}\zeta(2) \log ^4(2)
	  -\frac{6}{2 \cdot 6!} \log ^6(2)
	  \, .
\end{aligned}
\end{equation}

\end{enumerate}
\end{Cor}

Note that an evaluation of \( S_{3,3}(\tfrac{1}{2}) \)  is already known, but the general result in \cite[Theorem 4]{Ko2}  would only express it in terms of \( S_{2,4}(-1) \) (equivalently of \( S_{4,2}(\tfrac{1}{2}) \), by reflection and inversion) and \( S_{3,3}(-1) \).  The above reduction corresponds to the fact that \( \delta S_{3,3}(\tfrac{1}{2}) = 0 \).

We also stress that this reduction still contains weight 6 Nielsen polylogs. However, we expect that both \( S_{4,2}(-1) \) and \( S_{4,2}(\tfrac{1}{2}) \) reduce further, since  their motivic 2-coboundaries vanish.  From the 3-term and duplication relation for \( \Li_3 \), we obtain that both \( \{\tfrac{1}{2}\}_3 \) and \( \{-1\}_3 \) are rational multiples of \( \{1\}_3 \), so each coboundary reduces to \( 0 \) via the antisymmetry of the wedge product \( \{1\}_3 \wedge \{1\}_3 = 0 \).  These reductions would imply also that \( S_{3,3}(-1) \) reduces to depth~1, as opposed to depth~2 above.  We return to these questions  in Section \ref{sec:s42m1reduction} below.

\subsection{\texorpdfstring{Functional equations for \( S_{4,2} \)}{Functional equations for S\textunderscore\{4,2\}}}\label{sec:s42alg}
As mentioned above, in analogy with the case of $S_{3,2}$ we expect that \( S_{4,2} \) of any  \( \Li_3 \) functional equation can be reduced to \( \Li_6 \) terms.  As evidence for this, we show this for the three term relation and the algebraic family of functional equations from Section \ref{sec:algfe}.
 \medskip
 
Corresponding to the three-term relation for \( \Li_3 \), namely
\begin{equation}
	\Li_3(1-z)+\Li_3(z)+\Li_3\Big(\frac{z}{z-1}\Big) = \zeta(3) \Mod{products} \, , \label{eq:li3threeterm}
\end{equation}
we have the following functional equation for \( S_{4,2} \).

\begin{Prop}[Three-term relation]\label{prop:s42three}
	For all \( z \in \CC \sm (-\infty, 0] \cup [1, \infty) \), the following 3-term identity for \( S_{4,2} \) holds
	\begin{align*}
	&S_{4,2}(1-z)+S_{4,2}(z)+S_{4,2}\Big(\frac{z}{z-1}\Big) \= \\[1ex] 
		& 2 \Li_6(1-z)
		+2 \Li_6(z)
		+2\Li_6\Big(\frac{z}{z-1}\Big)
		 - \Big(\Li_5(z) 
		-\Li_5\Big(\frac{z}{z-1}\Big) \Big) \log (1-z)
		-\Li_5(1-z) \log (z)
		\\& 
		{}-\frac{3}{6!} \log ^6(1-z)
		+\frac{2}{5!} \log (z) \log ^5(1-z)
		-\frac{1}{2! \, 4!} \log ^2(z) \log ^4(1-z)
		-\frac{2}{4!} \zeta (2) \log ^4(1-z)
		\\&
		{}+\frac{1}{3!} \zeta (2) \log(z) \log ^3(1-z)
		-\frac{1}{3!} \zeta (3) \log ^3(1-z)
		+\frac{1}{2!} \zeta (3) \log (z) \log^2(1-z)
		-\frac{7}{4 \cdot 2!} \zeta (4) \log ^2(1-z)
		\\&
		{}+\zeta (4) \log (z) \log(1-z)
		+\zeta (5) \log (z)
		-\zeta (2) \zeta (3) \log (1-z)
		-\Big( \frac{1}{2}\zeta (3)^2+\frac{5}{4} \zeta (6) \Big) \,.
	\end{align*}
	
	\begin{proof}
		Differentiate, and take \( z \to 0 \) to fix the constant as \( S_{4,2}(1) - 2 \Li_6(1) = -\frac{1}{2} \zeta(3)^2 - \frac{5}{4} \zeta(6) \). \medskip
		
		Alternatively, we can also verify that this follows from reflection and inversion, by checking the polynomial invariant from Section \ref{sec:spanningset}:
		\begin{alignat*}{2}
			&& S_{4,2}(1-z)+S_{4,2}(z)+S_{4,2}\Big(\frac{z}{z-1}\Big) & - 2 \Big( \Li_6(1-z)+ \Li_6(z)+ \Li_6\Big(\frac{z}{z-1}\Big) \Big) \\
			\mapsto{}\: && -4 X Y^3 \:+\: 4 X^3 Y \:+\: 4 Y (X-Y)^3 & - 2\big( -Y^4 + X^4 -(X-Y)^4\big) = 0 \,.
			\tag*{\hspace{-1em}\qedhere} 
			 \end{alignat*}
	\end{proof}
\end{Prop}

We note the following reductions  of \( S_{3,3} \) at roots of unity (the latter of which is likely to be known, and the former of which follows from the MZV Data Mine \cite{mzvDM} as the Nielsen polylogarithms at $-1$ are alternating MZV's).

\medskip
\begin{Cor}\label{s33rootsof1}
We have the following specialisations.\medskip
\begin{enumerate}
\item
		\abovedisplayskip=0pt\abovedisplayshortskip=0pt~\vspace*{-1.3\baselineskip}
\begin{align} \label{eqn:s33m1simpler}
	\hspace{-12.5em } S_{3,3}(-1) \= \frac{3}{2} S_{4,2}(-1) + \frac{5}{16} \zeta(6) - \frac{1}{4} \zeta(3)^2\,.
\end{align}
\medskip
\item
		\abovedisplayskip=0pt\abovedisplayshortskip=0pt~\vspace*{-1.3\baselineskip}
\begin{align*}
	S_{3,3}( e^{2\pi i / 6} ) \= &
	3 \Li_6(e^{2\pi i / 6})
	-\frac{1}{2}\zeta(3)^2 
	-\frac{1829}{1944}\zeta (6)
	+\frac{1}{3} i \pi \Big( S_{3,2}(e^{2\pi i / 6})-2 \zeta (5) \Big)
	\\ & 
	+ \frac{1}{3} \zeta (2) \Li_4(e^{2\pi i / 6})
	+\frac{1}{324} (2 i \pi) ^3  \zeta (3) 
	\,.
\end{align*}
\end{enumerate}
\end{Cor}

\begin{proof}
\begin{enumerate}
		\item
Setting \( z = \frac{1}{2} \) in Proposition \ref{prop:s42three} leads to the following two-term identity
\begin{equation} \label{eqn:s42half}
	\begin{aligned} 
	S_{4,2}\bigg(\big[-1\big] + 2\bigg[\frac{1}{2}\bigg]\bigg) \= &
	4\Li_6\Big(\frac{1}{2}\Big)
	+2\Li_5\Big(\frac{1}{2}\Big) \log (2)
	 -\frac{51}{16} \zeta(6)
	 	-\frac{1}{2}\zeta (3)^2
	 \\&
	 -\Big( \frac{1}{16} \zeta (5)
	 -\zeta(2) \zeta (3) \Big) \log (2)
	 +\frac{1}{4 \cdot 2!} \zeta(4) \log ^2(2)
	 \\&
	 -\frac{2}{3!} \zeta (3) \log^3(2)
	  +\frac{2}{4!} \zeta(2) \log ^4(2) 
	  	 -\frac{6}{6!} \log ^6(2)
	  \,.
	\end{aligned}
\end{equation}
Now note that \eqref{eqn:s33m1} and \eqref{eqn:s42half} together imply \eqref{eqn:s33m1simpler}.
\item
 This follows from the \( S_{3,3} \) to \( S_{4,2} 
\) reduction (Proposition \ref{prop:s33reduce}) and the \( S_{4,2} \) three term identity (Proposition \ref{prop:s42three}). \qedhere
\end{enumerate}
\end{proof}

For the algebraic \( \Li_3 \) functional equation from Section \ref{sec:algfe}, we can reduce \( S_{4,2} \) to \( \Li_6 \), as expected.  This was also used in \cite{Ch} to obtain pure \( \Li_6 \) functional equations, from certain depth reductions of the depth 2 integral \( I_{5,1} \) under trilogarithm functional equations.

	\begin{Prop}[Proposition 7.6.12 in \cite{Ch}]\label{prop:s42alg}
	Let \(a,b,c\in \ZZ\sm\{0\}\), with \( a+b+c=0 \), and let \( \{p_i(t)\}_{i=1}^r \)  be the roots of \( x^a(1-x)^b = t \).  For convenience take \( a > 0 \).  Then the following reduction holds on the level of the mod-products symbol
	\begin{align*}
	& \sum_{i = 1}^{r} -\frac{1}{a} S_{4,2}(1-p_i(t)) + \frac{1}{b} S_{4,2}(p_i(t)) \modsh
	\\& \sum_{i=1}^{r} \frac{b-a}{a^2} \Li_6(1-p_i(t)) -  \frac{a - b}{b^2} \Li_6(p_i(t)) - \frac{1}{a+b} \Li_6(1 - p_i(t)^{-1}) \,.
	\end{align*}
\end{Prop}
	\begin{Cor}
	We have the clean single-valued identity
		\begin{align*}
	& \sum_{i = 1}^{r} -\frac{1}{a} \SCS_{4,2}(1-p_i(t)) + \frac{1}{b} \SCS_{4,2}(p_i(t)) \= 
	\\& \sum_{i=1}^{r} \frac{b-a}{a^2} \LiCS_6(1-p_i(t)) -  \frac{a - b}{b^2} \LiCS_6(p_i(t)) - \frac{1}{a+b} \LiCS_6(1 - p_i(t)^{-1}) \, .
	\end{align*}
\end{Cor}

\begin{proof}
	Consider the limit \( t \to 0 \) and use \( \LiCS_6(0) = \LiCS_6(1) = \LiCS_6(\infty) = 0 \) and \( \SCS_{4,2}(0) = \SCS_{4,2}(1) = \SCS_{4,2}(\infty) = 0 \).
	
	If \( b > 0 \), we obtain roots \( p_i = 0 \) with multiplicity \( a \) and \( p_i = 1 \) with multiplicity \( b \), giving constant \( 0 \).  If \( -a < b < 0 \), we obtain roots \( p_i = 0 \) with multiplicity \( a \), and the constant is also 0.  Otherwise \( b < -a \), and we obtain roots \( p_i = 0 \) with multiplicity \( a \) and roots \( p_i = \infty \) with multiplicity \( -b-a \), and the constant is still 0.
\end{proof}

\begin{proof}[Proof of proposition]
	Expand out as in the proof of Proposition \ref{prop:s32alg}, using the recursive definition of the mod-product symbol of \( S_{n,p}(z) \), and replace \( 1-p_i \) in the last tensor factor by \( t^{1/b} p_i^{-a/b} \).  The difference becomes
	\begin{align*}
		\sum_{i=1}^{r} \bigg( & \frac{1}{b} \big\{ S_{3,2}(1-p_i) + S_{3,2}(p_i) - \Li_5(1-p_i) - \Li_5(1-p_i^{-1}) - \Li_5(p_i) \big\} \otimes p_i
		 \\ 
		& -\frac{1}{a b} \Big\{ S_{3,2}(1-p_i) - \frac{a-b}{a} \Li_5(1-p_i) - \frac{a}{a+b} \Li_5 (1 - p_i^{-1}) - \frac{a}{b}	\Li_5(p_i) \Big\} \otimes t \bigg) \,.
	\end{align*}
	The first bracket cancels using 
	the two term \( S_{3,2}(z) + S_{3,2}(1-z) \) identity from Proposition \ref{prop:s32twoterm}.  The second factor cancels using this, and the reduction for \( S_{3,2} \) of the algebraic \( \Li_2 \) equation from Proposition \ref{prop:s32alg}.
\end{proof}

In particular, we expect a reduction of \( S_{4,2} \), applied to Goncharov's 840-term relation \cite{Go-ams} for \( \Li_3 \), to \( \Li_6 \). 
By analogy with the weight 5 case (Corollary \ref{cor:l5fe}), we anticipate an interesting many variable functional equation for $\Li_6$ to arise from applying the obvious 8-fold antisymmetrisation to such a reduction.

\subsection{\texorpdfstring{Depth reductions of \( S_{4,2}(-1) \), \( S_{4,2}(\frac12) \) and \( S_{4,2}(\phi^{-2})\) on the level of the coproduct}{Depth reductions of S\textunderscore{}\{4,2\}(-1) and S\textunderscore{}\{4,2\}(phi\textasciicircum{}(-2))}}\label{sec:s42m1reduction}

The motivic coproduct yoga  suggests that one can reduce \( S_{4,2}(-1) \) alone to classical polylogarithms and lower weight products, and that one can similarly reduce \( S_{4,2}(\phi^{-2}) \), where \( \phi = \frac{1}{2} (1 + \sqrt{5}) \).  These claims are expected because the following trilogarithm identities
\[
 \{-1\}_3 \= -\frac{3}{4} \{1\}_3 \,, \quad 
 \{\phi^{-2}\}_3 \= \frac{4}{5} \{1\}_3 \,,
\]
lead to the following trivial cobrackets for \( S_{4,2} \) using the antisymmetry of the wedge:
\begin{alignat*}{4}
\delta S_{4,2}(-1) \= && \{-1\}_3 \wedge \{1\}_3 \= && -\frac{3}{4} \{1\}_3 \wedge \{1\}_3 \= && 0 \, \\
\delta S_{4,2}(\phi^{-2}) \= && \{\phi^{-2}\}_3 \wedge \{1\}_3 \= && \frac{4}{5} \{1\}_3 \wedge \{1\}_3 \= && 0 \,.
\end{alignat*}

The first trilogarithm identity is just the case \( z = -1 \) of the trilogarithm duplication relation
\[
	\Li_3(z) + \Li_3(-z) \= \frac{1}{4} \Li_3(z^2) \,.
\]
The second identity, known to Landen, follows from duplication and the three-term relation \eqref{eq:li3threeterm}.  For the analytic functions we have
\begin{align*}
	\Li_3(-1) & \= -\frac{3}{4} \zeta(3) \,, \\
	\Li_3(\phi ^{-2}) & \= \phantom{{}+{}} \frac{4}{5} \zeta (3) -\frac{4}{5} \zeta (2) \log (\phi )+\frac{2 }{3}\log ^3(\phi
	) \,.
\end{align*}

The corresponding depth reductions for \( S_{4,2}(-1) \) and \( S_{4,2}(\phi^{-2}) \) would immediately follow from the conjectured reduction of \( S_{4,2} \) of the trilogarithm duplication relation \( \Li_3(z) + \Li_3(-z) - \frac{1}{2^2} \Li_3(z^2) \= 0 \).  Unfortunately, we do not have such an expression.  We can still investigate these reductions numerically, and via other functional identities. \medskip

We are able to find a certain mod-products symbol level identity
\[
	I_{5,1}(z, -1) \modsh \sum\nolimits_i \alpha_i \Li_6(f_i) + \sum\nolimits_j \beta_j S_{4,2}(g_j) \, ,
\]
where the \( S_{4,2} \)-arguments evaluate to \( \pm 1 \), at \( z = 1 \).  
Knowing how \( I_{5,1}(1,-1) \) evaluates in terms of \( S_{4,2}(-1) \), 
one is able to obtain an identity expressing \( S_{4,2}(-1) \) in terms 
of \( \Li_6 \) modulo products.
With some work one can also derive an analytic identity. Using a well known lattice reduction algorithm (`LLL') and some post hoc simplifications we have 
searched for the simplest such identity, and so far we have found 
the following relation which holds to high precision.  We have verified it to 10,000 decimal places in PARI/GP \cite{PARI2}.  Here \( \overset{?}{=} \) means the result has been numerically checked to high precision, but it has not yet been formally proven.
\begin{equation}\label{eqn:s42m1reduction}
	\begin{aligned}
	S_{4,2}(-1) \overset{?}{\=} & \frac{1}{13} \bigg(\frac{1}{3}\Li_6\Big(-\frac{1}{8}\Big)-162
	\Li_6\Big(-\frac{1}{2}\Big)-126
	\Li_6\Big(\frac{1}{2}\Big)\bigg)
	-\frac{1787 }{624}\zeta (6)
	+\frac{3}{8} \zeta (3)^2
	\\& 
	{}+\frac{31}{16} \zeta (5) \log(2)
	-\frac{15}{26} \zeta (4) \log ^2(2)
	+\frac{3}{104} \zeta (2) \log^4(2)
	-\frac{1}{208} \log ^6(2) \,.
\end{aligned}
\end{equation}
Note that the coefficient of \( \Li_6(-\frac{1}{8}) \) is written deliberately as \( \frac{1}{3} \) inside the parentheses, for structural reasons as follows. \medskip

\paragraph{\em Strategy for finding \( S_{4,2}(-1) \) evaluation by matching the coproduct}
	Using the coproduct, we can better understand the nature and structure of this reduction, and attempt to generalise it to higher cases.  This equality on the motivic level means that the coproducts (reduced coproducts for simplicity) of both sides must agree.  We can compute that
	\[
		\Delta' S_{4,2}(-1) \=  \frac{3}{4} \zeta(3) \otimes \zeta(3) + 
		\frac{31}{16} \log(2) \otimes \zeta(5) \,.
	\]
	Firstly, it is straightforward to see \(
		\Delta' \zeta(3)^2 \= 2 \zeta(3) \otimes \zeta(3) \),
	So the \( \zeta(3) \otimes \zeta(3) \) component of \( \Delta' S_{4,2}(-1) \) is matched by
	\[
		\frac{3}{8} \zeta(3)^2 \, ,
	\]
	exactly as appears in the reduction \eqref{eqn:s42m1reduction}.  How is the rest of the coproduct matched by this reduction?  Recall that
	\[
		\Delta' \Li_n(x) \= \sum_{k=1}^{n-1} \frac{1}{(n-k)!} \Li_k(x) \otimes 
		\log^{n-k}(x) \, ,
	\]
	and the right hand factor is only defined ``modulo \( i\pi \)'', because of branch cut ambiguities.  So  the following weight \( (5,1) \)-part of the coproduct becomes
	\begin{align*}
		& \Delta^{(5,1)} \Bigg( \frac{1}{3}\Li_6\Big(-\frac{1}{8}\Big)-162
		\Li_6\Big(-\frac{1}{2}\Big)-126
		\Li_6\Big(\frac{1}{2}\Big) \Bigg) \= \\
		& -\bigg(\Li_5\Big(-\frac{1}{8}\Big)-162
		\Li_5\Big(-\frac{1}{2}\Big)-126	\Li_5\Big(\frac{1}{2}\Big) \bigg) \otimes \log(2) \, ,
	\end{align*}
	the factor \( \frac{1}{3} \) annihilating with the \( \log(-\frac{1}{8}) = -3 \log(2) + i \pi \).  This is essentially an avatar of Lewin's pseudo-integration process \cite[Section 1.4]{lewin1991structural}.
	
	Recall now the identity \cite[p. 419]{Za1} for the single-valued polylog
	\[
	\LiZS_5\Big({-}\frac{1}{8}\Big) - 162 \LiZS_5\Big({-}\frac{1}{2}\Big) - 126 \LiZS_5\Big(\frac{1}{2}\Big) \= \frac{403}{16} \zeta(5) \, ;
	\]
	this is already given in \cite[Equation 7.100]{Le} for the analytic function \( \Li_5 \) with the following explicit lower order terms (after correcting the missing coefficient $\frac{1}{4}$ for  $\pi^2 \log(2)^3$)
	\begin{align*}
		& \Li_5\Big({-}\frac{1}{8}\Big) - 162 \Li_5\Big({-}\frac{1}{2}\Big) - 126 \Li_5\Big(\frac{1}{2}\Big) \= \\
		& \frac{403}{16} \zeta(5) - \frac{3}{8} \log(2)^5 + \frac{6}{4} \zeta(2) \log(2)^3 - 15 \zeta(4) \log(2) \, .
	\end{align*}
	In particular, we obtain the following term in the coproduct
	\[
		-\frac{403}{16} \zeta(5) \otimes \log(2) \, .
	\]
	So to match the actual term \( \frac{31}{16} \log(2) \otimes \zeta(5) \) 
	appearing in \( \Delta' S_{4,2}(-1) \) we can take the combination
	\[
		\frac{1}{13} \Bigg( \frac{1}{3}\Li_6\Big(-\frac{1}{8}\Big)-162
		\Li_6\Big(-\frac{1}{2}\Big)-126
		\Li_6\Big(\frac{1}{2}\Big) \Bigg) + \frac{31}{16} \zeta(5) \log(2) \, ,
	\]
	as is manifest in the reduction in \eqref{eqn:s42m1reduction}.  This 
	explains the main term of the reduction.  We have, for simplicity, ignored 
	much of the coproduct, not just the lower order product terms in the weight 
	\( (5,1) \)-part, but also the weight \( (k, 6-k) \)-parts, for \( k = 1, 
	\ldots, 4 \).  This is not a cause for concern, since these parts are 
	strictly simpler and so easier to deal with; they involve only products in 
	the left hand factor, or higher powers of \( \log(2) \) in the right hand 
	factor. \medskip
	
	In fact, by using analytic identities among \( \Li_k(-\frac{1}{8}), 
	\Li_k(-\frac{1}{2}), \Li_k(\frac{1}{2}) \), for \( k = 2, \ldots, 5 \), one 
	can derive the following coproduct expression purely involving zeta values, 
	and \( \log(2) \)
	\begin{align*}
		& \Delta' \bigg( \frac{1}{3}\Li_6\Big(-\frac{1}{8}\Big)-162
		\Li_6\Big(-\frac{1}{2}\Big)-126
		\Li_6\Big(\frac{1}{2}\Big) \bigg) \= \\
		&
		-\frac{403}{16} \zeta (5)\otimes \log (2)
		\quad + \quad\frac{15}{2} \bigg( \zeta(4) \otimes \log ^2(2)
		+ 2 \left(\zeta(4) \log(2)\right)\otimes \log (2) \bigg)
		\\
		&
		-\frac{3}{8} \bigg( 
		 \zeta(2) \otimes \log^4(2)
		+ 4 \left(\zeta(2) \log (2)\right)\otimes \log ^3(2)
		+ 6 \left(\zeta(2) \log^2(2)\right)\otimes \log ^2(2)
		+ 4 \left(\zeta(2) \log ^3(2)\right)\otimes \log(2)
		\bigg)
		\\& + \frac{1}{16} \bigg( 
		 6 \log ^5(2)\otimes \log (2)		
		+ 15 \log ^4(2)\otimes \log ^2(2)
		+ 20 \log^3(2)\otimes \log ^3(2)
		\\& 
		\phantom{ + \frac{16}{16} \bigg( } 
		 +  15 \log ^2(2)\otimes \log^4(2)
		+ 6 \log (2)\otimes \log^5(2) \Big) \,.
	\end{align*}
	The bracketed terms can be recognised as reduced coproducts of simple product expressions, so that the resulting combination 
	\begin{align*}
		\frac{1}{13} \Bigg( & \frac{1}{3}\Li_6\Big(-\frac{1}{8}\Big)-162
		\Li_6\Big(-\frac{1}{2}\Big)-126
		\Li_6\Big(\frac{1}{2}\Big) \\
		& - \frac{15}{2}  \zeta(4) \log^2(2) + \frac{3}{8} \zeta(2) \log^4(2) - 
		\frac{1}{16} \log^6(2) + \frac{403}{16} \zeta(5) \log(2)  \Bigg) 
	\end{align*}
	has reduced coproduct exactly \( \frac{31}{16} \log(2) \otimes \zeta(5) \).

	This explains all of the terms and  fixes all of the coefficients in the reduction, except for the \( \zeta(6) \) coefficient.  Like all Riemann zeta values \( \zeta(n) \), it has trivial reduced coproduct, and so this coefficient can only be fixed by numerically evaluating. \medskip

Assuming~\eqref{eqn:s42m1reduction}, we can use \eqref{eqn:s42half} and~\eqref{eqn:s33m1simpler} to obtain the following.
\begin{Rem}\label{rem:s42evaluations}
	One has the following reductions of \( S_{4,2}(\frac{1}{2}) \) and \( 
	S_{3,3}(-1) \) to polynomials in classical polylogarithms
	\begin{align*}
		S_{4,2}\Big(\frac{1}{2}\Big) \overset{?}{\=} & 
		-\frac{1}{26} \Big(
		\frac{1}{3}\Li_6\Big(-\frac{1}{8}\Big)
		-162\Li_6\Big(-\frac{1}{2}\Big)
		-178\Li_6\Big(\frac{1}{2}\Big)\bigg)
		-\frac{7}{16}\zeta (3)^2
		-\frac{101}{624} \zeta(6)
		\\& 
		+\Li_5\Big(\frac{1}{2}\Big) \log (2)
		-\Big(\zeta (5)	-\frac{1}{2} \zeta (2) \zeta (3) \Big) \log (2)
		+\frac{73}{208} \zeta (4) \log ^2(2)
		-\frac{1}{6} \zeta (3) \log	^3(2)
		\\& 
		+\frac{17}{624} \zeta (2) \log ^4(2)
		-\frac{11 }{6240}\log ^6(2) \, ,
		\\[2ex]
		S_{3,3}(-1) \overset{?}{\=} & \frac{3}{26} \bigg(
			\frac{1}{3} \Li_6\Big(-\frac{1}{8}\Big)
			-162\Li_6\Big(-\frac{1}{2}\Big)
			 - 126 \Li_6\Big(\frac{1}{2}\Big)\bigg)
		+\frac{5}{16} \zeta (3)^2
		-\frac{1657}{416} \zeta(6)
		\\&
		+\frac{93}{32} \zeta(5) \log (2)
		-\frac{45}{52} \zeta (4) \log ^2(2)
		+\frac{9}{208} \zeta (2) \log ^4(2)
		-\frac{3}{416} \log ^6(2) \, .
	\end{align*}
\end{Rem}
We emphasise that the only uncertainty in these equations lies in the respective coefficient of $\zeta(6)$, as the coproduct expressions of both sides agree in each case. \medskip

\paragraph {\em Aside: connection with alternating MZV's} The reduction from \eqref{eqn:s42m1reduction} allows us to give an apparently new evaluation for some weight 6 alternating (Euler) MZV's, and thence reduce all weight $\leq6$ alternating MZV's to polynomials in classical polylogarithms.

More explicitly, one has the equalities
\begin{align*}
	S_{4,2}(-1) &= \zeta(1, \overline{5}) \,, \\
	\zeta(1,1,1,\overline{3}) - \frac{1}{2} \zeta(1,\overline{5}) &= \begin{aligned}[t]
	& 2 \Li_6\Big(\frac{1}{2}\Big)+2
	\Li_5\Big(\frac{1}{2}\Big) \log (2)+\Li_4\Big(\frac{1}{2}\Big) \log ^2(2)-\frac{1}{4}\zeta (3)^2 \\& +\frac{7}{24}
	\zeta (3) \log ^3(2)-\frac{53}{32} \zeta(6)+\frac{1}{36}\log ^6(2)-\frac{1}{8} \zeta(2) \log ^4(2)\,,
	\end{aligned}
\end{align*}
where
\begin{alignat*}{5}
	&\zeta(1,\overline{5}) && \ceq && \:\:\:\:\:\:\: \sum_{0 < n_1 < n_2} \:\:\:\:\:\:\: && \:\: \frac{(-1)^{n_2}}{n_1 n_2^5} \:\: &&\= \Li_{1,5}(1, -1) \,, \\
	&\zeta(1, 1, 1, \overline{3}) &&\ceq && \sum_{0 < n_1 < n_2 < n_3 < n_4} && \frac{(-1)^{n_4}}{n_1 n_2 n_3 n_4^3} && \= \Li_{1,1,1,3}(1, 1, 1, -1)
\end{alignat*}
are alternating MZV's of weight $6$.

Using the MZV Data Mine \cite{mzvDM}, a set of algebra generators of alternating MZV's is given up to weight 6 by
\[
	\big\{ \log(2), \zeta(2), \zeta(3), \zeta(5), \zeta(1,\overline{3}), \zeta(1, 1, \overline{3}), \zeta(1, 1, 1, \overline{3}), \zeta(1,\overline{5})  \big\} \, .
\]
The strictly alternating MZV's \( \zeta(1,\overline{3}) \) and \( \zeta(1,1,\overline{3}) \) are already known to be polynomials in classical polylogarithms, namely
\begin{align*}
	\zeta(1, \overline{3}) \= 
	& 2 \Li_4\Big(\frac{1}{2}\Big)
	-\frac{15}{8} \zeta(4)
	+\frac{7}{4} \zeta (3) \log (2)
	-\frac{1}{2!} \zeta(2) \log ^2(2)
	+\frac{2}{4!}\log^4(2)
	 ,\, \\[1ex]
	\zeta(1, 1, \overline{3}) \= 
	& -2 \Li_5\Big(\frac{1}{2}\Big)
	-2 \Li_4\Big(\frac{1}{2}\Big) \log (2)
	+\frac{33}{32}\zeta (5)
	 +\frac{1}{2} \zeta(2) \zeta(3)
	\\&
	-\frac{7}{8} \zeta (3) \log ^2(2)
	+\frac{1}{3} \zeta(2) \log ^3(2) 
	-\frac{1}{15}\log ^5(2)
	\, .
\end{align*}  Together with the above reduction for \( \zeta(1,\overline{5}) \), and consequently \( \zeta(1,1,1,\overline{3}) \), one obtains a reduction, albeit complicated, of all alternating MZV's of weight \( \leq 6 \) to polynomials in classical polylogarithm values.  \medskip

\paragraph {\em Reduction of \( S_{4,2}(\phi^{-2}) \) obtained using the coproduct.} As explained in the paragraph on the strategy for finding \( S_{4,2}(-1) \) after \eqref{eqn:s42m1reduction}, a great deal of structure in the \( S_{4,2}(-1) \) reduction above becomes manifest in the coproduct.  By combining this understanding with the \( S_{3,2}(\phi^{-2}) \) reduction found earlier, we can produce a very short list of potentially relevant polylog arguments for a candidate \( S_{4,2}(\phi^{-2}) \) reduction.  We quickly find the following with LLL to high precision, which was then subsequently verified to 10,000 decimal places in PARI/GP \cite{PARI2}.  A complete analysis of the coproduct, similar to the case \( S_{4,2}(-1) \) above, explains all of the coefficients and terms, except for the \( \zeta(6) \) coefficient which must be numerically fixed.
\begin{align*}
S_{4,2}(\phi ^{-2}) \overset{?}{\=} &
\frac{1}{396} \Li_6\Big( 2 \big[ \phi ^{-6} \big] -128 \big[ \phi ^{-3}\big]+801
\big[\phi ^{-2}\big]-576 \big[\phi^{-1} \big] \Big)
+\frac{35 }{99}\zeta (6)
+\frac{2}{5}  \zeta (3)^2
\\&
+ \Li_5 \left(\phi ^{-2}\right) \log (\phi)
-\zeta (5)\log (\phi )
+\frac{2}{11} \zeta (4) \log ^2(\phi )
-\zeta (3) \Li_3\left(\phi ^{-2}\right)
\\&
+\frac{10}{33} \zeta (2) \log ^4(\phi )
-\frac{79}{990} \log ^6(\phi ) \,.
\end{align*}

\smallskip

\section{Identities in weight 7}
In this section we `depth reduce' \( S_{4,3} \), and give evaluations of it at $-1$ and $\frac{1}{2}$. Furthermore, in order to guarantee the vanishing of the coproduct terms and hence to have a chance to depth reduce \( S_{5,2} \) we need to invoke functional equations which hold  simultaneously for \( \Li_2 \) and  \( \Li_4 \). 
This is the smallest weight where such a requirement is needed, and in general we would need to understand simultaneous functional equations for different \( \Li_a \). An approach for finding equations of that type, at least with algebraic arguments, is outlined in Section \ref{sec:limlin}.
Finally, in Section \ref{sec:blochgroupids}, we also corroborate our expectations on linear combinations which simultaneously represent an element of both (higher) Bloch groups \( \mathcal{B}_2(\QQ) \) and   \( \mathcal{B}_4(\QQ) \).

\smallskip
\paragraph{\em Preconsiderations.} The 2-part of the motivic coproduct of \( S_{5,2}(z) \) and \( S_{4,3}(z) \) are computed to be
\begin{align*}
	\delta S_{5,2}(z) & \= \{z\}_2 \wedge \{1\}_5 + \{z\}_4 \wedge \{1\}_3 \, , \\
	\delta S_{4,3}(z) & \= -2 \{z\}_2 \wedge \{1\}_5 - \big( \{z\}_4 + S_{2,2}(z) \big)\wedge \{1\}_3 \\*
	& \= -2 \{z\}_2 \wedge \{1\}_5 - \Big( 2 \{z\}_4 - \{1-z\}_4 + \Big\{\frac{z}{z-1}\Big\}_4 \Big)\wedge \{1\}_3 \,.
\end{align*}

Hence we expect a reduction of \( S_{5,2} \) to \( \Li_7 \) only when \( \sum \alpha_i [x_i] \) simultaneously satisfies a \( \Li_2 \) and a \( \Li_4 \) identity.  On the other hand, we showed in Theorem \ref{thm:nielsendepth} that \( S_{4,3} \) can be reduced to lower depth.  Since its coproduct is matched by
\begin{align*}
	-S_{5,2}(1-z)+2 S_{5,2}(z)+S_{5,2}\Big(\frac{z}{z-1}\Big) \, ,
\end{align*}
the difference should be expressible in terms of \( \Li_7 \)'s.

\begin{Prop} The following identity follows from inversion and reflection, and it reduces \( S_{4,3}(z) \) to lower depth
\begin{align*}
	S_{4,3}(z) \= & -S_{5,2}(1-z)+2
	S_{5,2}(z)+S_{5,2}\Big(\frac{z}{z-1}\Big) \\ & + 2 \Li_7(1-z)-3 \Li_7(z)-3 \Li_7\Big(\frac{z}{z-1}\Big)  \Mod{products} \,.
\end{align*}

\begin{proof}
	The polynomial invariant of the difference of the left hand side and the right hand side is
	\[
		10 X^3 Y^2 - \Big( 5 X Y^4 + 10 X^4 Y +  5 (X - Y)^4 Y -2 Y^5  -3 X^5 + 3 (X - Y)^5 \Big) \= 0 \, . \qedhere
	\]	
\end{proof}
\end{Prop}

At the value \( z = -1 \), it follows from the inversion identity of \( S_{n,2} \) in Proposition \ref{prop:nielseninverse} that \( S_{5,2}(-1) \) is reducible, and
\begin{align*}
S_{5,2}(-1) \= 
 -\frac{251}{128} \zeta (7) 
 +\frac{1}{2} \zeta(2)\zeta(5)
 +\frac{7}{8}\zeta(3) \zeta (4)
 \,.
\end{align*}

However, since \( \delta S_{5,2}(\tfrac{1}{2}) = \{\tfrac{1}{2}\}_4 \wedge \{1\}_3 \neq 0 \), we do not expect a reduction of this to lower depth.  Similarly \( S_{4,3}(-1) \) and \( S_{4,3}(\tfrac{1}{2}) \) both have non-vanishing cobracket involving \( \{\tfrac{1}{2}\}_4 \wedge \{1\}_3 \).  But there is the following reduction of each to \( S_{5,2}(\tfrac{1}{2}) \) and simpler objects.
\begin{align*}
S_{4,3}(-1) \= &
	2 S_{5,2}\Big(\frac{1}{2}\Big)
	-6 \Li_7\Big(\frac{1}{2}\Big)
	+2 S_{4,2}\Big(\frac{1}{2}\Big) \log (2)
	-6\Li_6\Big(\frac{1}{2}\Big) \log (2)
	-2 \Li_5\Big(\frac{1}{2}\Big) \log ^2(2)
	\\& 
	-\frac{31}{32} \zeta (7)
	+2 \zeta (2) \zeta (5)
	+\frac{11}{4} \zeta (4) \zeta(3)
	+\frac{1}{2!} \Big( \frac{1}{16} \zeta (5) - \zeta (2) \zeta (3) \Big) \log^2(2)
	\\&
	 -\frac{1}{2 \cdot 3!} \zeta(4) \log ^3(2)	
	+\frac{4}{4!} \zeta (3) \log ^4(2)
	-\frac{8}{5!} \zeta (2) \log^5(2)
	+\frac{36}{7!}\log ^7(2)
	 \,, \\[1em]
	S_{4,3}\Big(\frac{1}{2}\Big) \= & 
	S_{5,2}\Big(\frac{1}{2}\Big)
	-\Li_7\Big(\frac{1}{2}\Big)
	-S_{4,2}\bigg(\big[-1\big] + \bigg[\frac{1}{2}\bigg] \bigg)\log(2)
	+3\Li_6\Big(\frac{1}{2}\Big) \log(2)
	+\frac{3}{2} \Li_5\Big(\frac{1}{2}\Big) \log ^2(2)
	\\& 
	+\frac{255}{128} \zeta (7)
	-\frac{1}{2} \zeta (2) \zeta (5)
	-\frac{1}{8} \zeta (4) \zeta (3)
	-\frac{125}{32} \zeta(6) \log (2)
	+\frac{1}{2!} \bigg( \frac{15}{8} \zeta (5) + \zeta (2) \zeta (3) \bigg)\log ^2(2)
	\\& 
	-\frac{1}{2 \cdot 3!} \zeta (4) \log ^3(2)
	-\frac{5}{4!} \zeta (3) \log ^4(2)
	+\frac{7}{5!} \zeta (2) \log ^5(2)
	-\frac{48}{7!}\log ^7(2)
	\,.
\end{align*}

Moreover \( S_{5,2}(\tfrac{1}{2}) \) should be the only new irreducible object needed, by combining with the earlier reductions of \( S_{4,2}(-1) \) and \( S_{4,2}(\tfrac{1}{2}) \) in \eqref{eqn:s42m1reduction} and Remark \ref{rem:s42evaluations}.

\subsection{\texorpdfstring{\( \Li_2 + \Li_4 \) functional equations}{Li\textunderscore{}2 + Li\textunderscore{}4 functional equations}}  \label{sec:s52alg} Recall from Section \ref{sec:algfe} that we have the following \( \Li_4 \) functional equation
\[
	\sum_{i = 1}^{r} \frac{1}{a} \Big\{\frac{1}{1-p_i(t)}\Big\}_4 + \frac{1}{b} \{p_i(t)\}_4 + \frac{1}{c} \{ 1 - p_i(t)^{-1} \}_4 \= 0 \, ,
\]
where \( \{p_i(t)\}_{i=1}^r \) are the roots of \( x^a (1 - x)^b = t \), for fixed \( a, b, c \in \ZZ \sm \{0\}  \) with \( a + b + c = 0 \).

One can notice that the individual orbits are already \( \Li_2 \) functional equations, since under the six-fold symmetry each reduces to a multiple of
\[
	\sum_{i=1}^{r} \{p_i(t)\}_2 \= 0 \, .
\]
Hence \( S_{5,2} \) of the same combination should be expressible in terms of \( \Li_7 \).  As was noted in Section \ref{sec:algfe}, for the case \( (a,b,c) = (1,2,-3) \), the roots of the equation can be rationally parametrised over \( \QQ \), giving a functional equation even with \emph{rational} arguments.

\begin{Prop}\label{prop:s52alg}
	Let \(a,b,c\in \ZZ\sm\{0\}\), with \( a+b+c=0 \), and let \( \{p_i(t)\}_{i=1}^r \)  be the roots of \( x^a(1-x)^b = t \).  For convenience take \( a > 0 \).
	Then the following reduction holds on the mod-products symbol
	\begin{align*}
		& \sum_{i = 1}^{r} \frac{1}{a} S_{5,2}\Big(\frac{1}{1-p_i(t)}\Big) - \frac{1}{b} S_{5,2}(p_i(t)^{-1}) + \frac{1}{c} S_{5,2}(1 - p_i(t)^{-1}) \modsh \\& \sum_{i=1}^{r} \frac{3a + b}{a^2} \Li_7\Big(\frac{1}{1-p_i(t)}\Big) - \frac{3 b + a}{b^2} \Li_7(p_i(t)^{-1}) - \frac{3c + a}{c^2} \Li_7(1 - p_i(t)^{-1}) \,.
	\end{align*}
\end{Prop}

\begin{Cor}
	We have the clean single-valued identity
	\begin{align*}
	& \sum_{i = 1}^{r} \frac{1}{a} \SCS_{5,2}\Big(\frac{1}{1-p_i(t)}\Big) - \frac{1}{b} \SCS_{5,2}(p_i(t)^{-1}) + \frac{1}{c} \SCS_{5,2}(1 - p_i(t)^{-1}) \\ - {}& \sum_{i=1}^{r} \frac{3a + b}{a^2} \LiCS_7\Big(\frac{1}{1-p_i(t)}\Big) - \frac{3 b + a}{b^2} \LiCS_7(p_i(t)^{-1}) - \frac{3c + a}{c^2} \LiCS_7(1 - p_i(t)^{-1}) \\
	& \quad \= \begin{cases}
		\displaystyle\frac{2 a}{c} \zeta(7) & \text{if $ b > 0 $\,,} \\[1ex]
		\displaystyle\frac{2(a^2b - a^2c - b^2c)}{abc} \zeta(7) & \text{if $ -a < b < 0 $\,,} \\[1ex]
		-\displaystyle\frac{2(a^2 + b^2)}{ab} \zeta(7) & \text{if $ b < -a  $\,.}
	\end{cases}
	\end{align*}
\end{Cor}

\begin{proof}
		Consider the limit \( t \to 0 \), and use \( \LiCS_7(0) = 0 \), \( \LiCS_7(1) = 2\zeta(7) \), \( \LiCS_7(\infty) = 0 \) and \( \SCS_{5,2}(0) = 0 \), \( \SCS_{5,2}(1) = 6\zeta(7) \), \( \SCS_{5,2}(\infty) = 2\zeta(7) \).
		
		If \( b > 0 \), we obtain roots \( p_i = 0 \) with multiplicity \( a \) and \( p_i = 1 \) with multiplicity \( b \), giving constant \( \frac{2 a}{c} \zeta(7) \).  If \( -a < b < 0 \), we obtain roots \( p_i = 0 \) with multiplicity \( a \), and the constant is \( \frac{2(a^2b - a^2c - b^2c)}{abc} \zeta(7) \).  Otherwise \( b < -a \) and we obtain roots \( p_i = 0 \) with multiplicity \( a \) and \( p_i = \infty \) with multiplicity \( -b-a \), giving the constant \( -\frac{2(a^2 + b^2)}{ab} \zeta(7) \).
\end{proof}

\begin{proof}[Proof of proposition]
	The strategy is exactly the same as in Propositions \ref{prop:s32alg} and \ref{prop:s42alg}.  Expand out using the recursive definition of the mod-products symbol, and we reduce to the algebraic functional equations in lower weight, plus the three-term for \( S_{4,2} \).
\end{proof}

\subsection{\texorpdfstring{\( \Li_{a_1} + \cdots + \Li_{a_n} \) functional equations}{Li\textunderscore{}a1 + ... + Li\textunderscore{}an}}
 \label{sec:limlin}
It is possible to somewhat artificially construct simultaneous 
functional equations for $\Li_2$ and $\Li_4$, and more generally 
simultaneous functional equations for $\Li_{a_1},\dots,\Li_{a_n}$
in the following manner.

Consider the function \(f(z) = \mu_1\Li_{a_1}(z) +\cdots+ \mu_n\Li_{a_n}(z)\), 
where $a_1<\dots<a_n$ are positive integers and $\mu_1,\dots,\mu_n$
are arbitrary non-zero numbers.
Then by the distribution relations of order $N$ we have
    \[
 	f^{(N)}(z) \ceq \sum_{y^N = z} f(y) \= \frac{\mu_1}{N^{a_1-1}} \Li_{a_1}(z) + \cdots + \frac{\mu_n}{N^{a_n-1}} \Li_{a_n}(z) \, .
    \]
Let $0<N_1<\cdots<N_n$ be positive integers, and denote 
$\underline{f}(z) = (f^{(N_1)}(z),\dots,f^{(N_n)}(z))^T$.
Then collecting the various distribution relations we get the equation
    \[
	\underline{f}(z) \= V_{a,N}\,\underline{\Li_{a}}(z)\,,
    \]
where $\underline{\Li_{a}}(z) \= (\mu_1\Li_{a_1}(z),\dots,\mu_n\Li_{a_n}(z))^T$,
and $V_{a,N}\=\begin{pmatrix}N_j^{1-a_i}\end{pmatrix}_{i,j=1}^n$
is a generalised Vandermonde matrix. Since $a_i$ and $N_j$ are distinct,
$\det(V_{a,N})$ is a non-zero multiple of an appropriate Schur 
polynomial $s_{\lambda}(N_1^{-1},\dots,N_n^{-1})$, which is positive 
since $N_j>0$ and $s_{\lambda}$ is a sum of monomials with 
positive coefficients. Therefore, $V_{a,N}$ is an invertible matrix and we have
    \[
	\underline{\Li_{a}}(z) \= V_{a,N}^{-1} \underline{f}(z)  \, .
    \]
The resulting combination \( \mu_j\Li_{a_j}(z) = \sum_k \alpha_k f(g_k(z)) \), 
where $\alpha_k$ and $g_k$ do not depend on $\mu_j$, then vanishes identically 
under any \( \Li_{a_j} \) functional equation \( \Lambda = \sum_\ell 
\gamma_\ell [x_\ell] \), so that
    \(
    \sum_{k,\ell} \alpha_k \gamma_\ell [g_k(x_\ell)]
    \)
is a functional equation for \( \Li_{a_1}, \ldots, \Li_{a_n} \) 
simultaneously. For the special case when $\Lambda$ is the distribution relation
    \[\Lambda \= \sum_{\ell=1}^k [z^M \zeta_k^\ell] - k^{1-a_j} [z^{Mk}] \,,\]
where $M = \operatorname{lcm}(N_1,\ldots, N_n)$, one obtains a rational 
\( \Li_{a_1} + \cdots + \Li_{a_n}  \) functional equation, but in general 
the functional equations constructed in this way will involve algebraic 
arguments.
  
\subsection{Bloch group identities.} \label{sec:blochgroupids}
Despite the scarcity of (rational) functional equations for $S_{5,2}$, we 
can still investigate experimentally, along the same lines as was done for 
the classical polylogarithms in~\cite{Za1}, whether 
combinations $S_{5,2}(\sum_j\nu_j[x_j])$ reduce to $\Li_7$ 
whenever $\sum_j\nu_j\{x_j\}_k=0$ for $k=2,4$.

Taking the algebraic identity $x^a(1-x)^b=t$ for \( a = 1, b = 2 \), and \( t = \frac{4}{27} \) leads to three roots \( p_i = \frac{1}{3}, \frac{1}{3}, \frac{4}{3} \).  Proposition \ref{prop:s52alg} then gives the following identity
\begin{equation}  \label{eq:s52of23units}
    \begin{split}
	&
	\SCS_{5,2}\bigg( \frac{2}{3}\bigg[{-}\frac{1}{2}\bigg]
	-\bigg[{-}\frac{1}{3}\bigg]
	-\frac{1}{3}\bigg[\frac{1}{4}\bigg]
	+\bigg[\frac{1}{3}\bigg]
	-2\bigg[\frac{2}{3}\bigg]
	-\frac{1}{2}\bigg[\frac{3}{4}\bigg] \bigg)
	\= 
	\\& \quad
	\LiCS_7\bigg({-}\frac{14}{9}\bigg[{-}\frac{1}{2}\bigg]
	+\frac{8 }{9}\bigg[\frac{1}{4}\bigg]
	-\frac{3}{2}\bigg[\frac{1}{3}\bigg]
	+\frac{7}{4}\bigg[\frac{3}{4}\bigg] \bigg)
	+\frac{10}{3}\zeta (7) \, .
    \end{split}
\end{equation}
Note that the arguments of $S_{5,2}$ are exceptional $\{2,3\}$-units,
i.e. they are numbers $z$ such that both $z$ and $1-z$ are of the 
from $\pm 2^k3^l$, $k,l\in\ZZ$. 
Looking at all possible combinations of non-trivial exceptional 
\( \{2,3\} \)-units in \( [-1, 1) \) that define elements lying both 
in $\mathcal{B}_2(\QQ)$ and $\mathcal{B}_4(\QQ)$ (in this case it is equivalent 
to their vanishing in the pre-Bloch groups $B_2(\QQ)$ and $B_4(\QQ)$), we 
find that they form a $5$-dimensional space, generated by
\begin{align*}
& \alpha_1 \= [-1] \, , \\&
\alpha_2 \= 
10 \bigg[{-}\frac{1}{2}\bigg]
+\bigg[{-}\frac{1}{8}\bigg]
-8 \bigg[\frac{1}{4}\bigg]
+22\bigg[\frac{1}{2}\bigg]
 \, , \\&
\alpha_3 \= 
6 \bigg[{-}\frac{1}{3}\bigg]
-4 \bigg[{-}\frac{1}{2}\bigg]
+2 \bigg[\frac{1}{4}\bigg]
-6 \bigg[\frac{1}{3}\bigg]
+12\bigg[\frac{2}{3}\bigg]
+3 \bigg[\frac{3}{4}\bigg]
  \, , \\&
\alpha_4 \= 
\bigg[{-}\frac{1}{8}\bigg]
-8\bigg[{-}\frac{1}{3}\bigg]
-14 \bigg[{-}\frac{1}{2}\bigg]
+\bigg[\frac{1}{9}\bigg]
-5\bigg[\frac{1}{4}\bigg]
-8 \bigg[\frac{1}{3}\bigg]
-2 \bigg[\frac{1}{2}\bigg]
 \, , \\&
\alpha_5 \= \:
\bigg[{-}\frac{1}{8}\bigg]
-9\bigg[{-}\frac{1}{3}\bigg]
\:+\:4\bigg[{-}\frac{1}{2}\bigg]
+\bigg[\frac{1}{9}\bigg]
-5\bigg[\frac{1}{4}\bigg]
-4 \bigg[\frac{1}{2}\bigg]
+9\bigg[\frac{3}{4}\bigg]
+\bigg[\frac{8}{9}\bigg]\, .
\end{align*}
In each of these cases we expect $S_{5,2}(\alpha_j)$ to reduce to $\Li_7$.
For \( \alpha_1 = [-1] \) we already gave the corresponding reduction for 
the analytic functions, the single-valued version of which is
    \[\SCS_{5,2}(-1) = -\frac{251}{64} \zeta(7) \, ,\]
while the combination given as the argument of $S_{5,2}$ in~\eqref{eq:s52of23units} corresponds to a 
multiple of $\alpha_3$. The remaining elements~$\alpha_2$,~$\alpha_4$, 
and~$\alpha_5$ appear to be a lot more difficult to reduce rigorously. However, 
in each case we can find a \emph{candidate} combination which works 
numerically to high precision (we have verified them for the single-valued
functions to 10,000 decimal places using PARI/GP~\cite{PARI2}).
For instance, for \( \alpha_2 \) we have

\begin{alignat*}
{7}
 \mathrlap{\SCS_{5,2}\big(
 10 \big[{-}\tfrac{1}{2}\big]
+\big[{-}\tfrac{1}{8}\big]
-8 \big[\tfrac{1}{4}\big]
+22 \big[\tfrac{1}{2}\big]
\big)
 \overset{?}{=}}
 \\ &&\LiCS_7\big(&&
\tfrac{1}{1105}\big[{-}\tfrac{2048}{2187}\big]&&
-\tfrac{77443}{195}\big[{-}\tfrac{3}{4}\big]&&
+\tfrac{23501}{663}\big[{-}\tfrac{2}{3}\big]&&
-\tfrac{32842}{9945}\big[{-}\tfrac{9}{16}\big]&&
-\tfrac{1049696}{255}\big[{-}\tfrac{1}{2}\big]
\\[-0.5ex]&&&&
+\tfrac{217}{34}\big[{-}\tfrac{4}{9}\big]&&
+\tfrac{217}{765}\big[{-}\tfrac{27}{64}\big]&&
-\tfrac{26449}{2210}\big[{-}\tfrac{3}{8}\big]&&
+\tfrac{16321}{9945}\big[{-}\tfrac{1}{3}\big]&&
-\tfrac{2420}{1989}\big[{-}\tfrac{8}{27}\big]
\\[-0.5ex]&& &&
-\tfrac{51647}{884}\big[{-}\tfrac{1}{4}\big]&&
+\tfrac{2648}{221}\big[{-}\tfrac{2}{9}\big]&&
-\tfrac{3140}{663}\big[{-}\tfrac{1}{6}\big]&&
-\tfrac{18}{1105}\big[{-}\tfrac{32}{243}\big]&&
+\tfrac{3932}{1105}\big[{-}\tfrac{1}{8}\big]
\\[-0.5ex]&&&&
-\tfrac{21139}{9945}\big[{-}\tfrac{1}{9}\big]&&
-\tfrac{307}{1530}\big[{-}\tfrac{3}{32}\big]&&
-\tfrac{217}{51}\big[{-}\tfrac{1}{12}\big]&&
+\tfrac{83}{6630}\big[{-}\tfrac{27}{512}\big]&&
-\tfrac{3167}{3978}\big[{-}\tfrac{1}{24}\big]
\\[-0.5ex]&&&&
+\tfrac{9359}{9945}\big[{-}\tfrac{1}{27}\big]&&
-\tfrac{88}{3315}\big[{-}\tfrac{1}{32}\big]&&
+\tfrac{77}{3978}\big[{-}\tfrac{1}{48}\big]&&
+\tfrac{328}{663}\big[{-}\tfrac{1}{54}\big]&&
+\tfrac{217}{3060}\big[{-}\tfrac{1}{64}\big]
\\[-0.5ex]&&&&
-\tfrac{61}{6630}\big[{-}\tfrac{2}{243}\big]&&
+\tfrac{31}{1020}\big[{-}\tfrac{1}{324}\big]&&
+\tfrac{12}{1105}\big[{-}\tfrac{1}{384}\big]&&
-\tfrac{7}{2210}\big[{-}\tfrac{1}{4374}\big]&&
-\tfrac{29}{1105}\big[\tfrac{1}{243}\big]
\\[-0.5ex]&&&&
+\tfrac{23}{2210}\big[\tfrac{3}{128}\big]&&
-\tfrac{217}{612}\big[\tfrac{1}{27}\big]&&
+\tfrac{294}{221}\big[\tfrac{2}{27}\big]&&
-\tfrac{5268}{1105}\big[\tfrac{1}{12}\big]&&
-\tfrac{84341}{19890}\big[\tfrac{1}{8}\big]
\\[-0.5ex]&&&&
+\tfrac{48827}{1989}\big[\tfrac{1}{6}\big]&&
-\tfrac{217}{102}\big[\tfrac{3}{16}\big]&&
+\tfrac{4895}{1989}\big[\tfrac{2}{9}\big]&&
-\tfrac{985027}{39780}\big[\tfrac{1}{3}\big]&&
+\tfrac{109586}{9945}\big[\tfrac{3}{8}\big]
\\[-0.5ex]&&&&
+\tfrac{1253}{13260}\big[\tfrac{32}{81}\big]&&
-\tfrac{1049557}{255}\big[\tfrac{1}{2}\big]&&
-\tfrac{1174}{3315}\big[\tfrac{16}{27}\big]&&
+\tfrac{67273}{663}\big[\tfrac{2}{3}\big]&&
-\tfrac{4447459}{9945}\big[\tfrac{3}{4}\big]
\\[-0.5ex]&&&&
+\tfrac{7859}{6630}\big[\tfrac{27}{32}\big]&&
+\tfrac{643}{306}\big[\tfrac{8}{9}\big]&&
-\tfrac{31}{1020}\big[\tfrac{243}{256}\big] \mathrlap{\big)}&&
+\tfrac{4241}{1105}\zeta (7) \, .
\end{alignat*}

\begin{Rem}
	Notice that although each $x_j\in\QQ$ that appears in this combination is 
    a $\{2,3\}$-unit, we also have primes $5$, $7$, $11$, and $13$ appearing 
    in factorisations of $1-x_j$.
\end{Rem}
 
\section{Identities in weight 8} 
In this section, we depth reduce \( S_{4,4}(z) \)  (Proposition \ref{prop:s44reduce}) and we reduce  \( S_{5,3} \) evaluated on the same family of algebraic \( \Li_2 \) functional equations (Proposition \ref{prop:s53alg}) as for \( S_{3,2} \). A special case thereof allows to reduce \( S_{5,3}(-1) \) to \( S_{6,2}(-1) \) and \( S_{6,2}(\frac{1}{2}) \), modulo polylogs and products, and subsequently to match the coproduct for \( S_{6,2}(-1) \) and even arrive at a tentative evaluation (Proposition \ref{prop:s62minusone} and Appendix \ref{app:reduction}). 

\smallskip
\paragraph{\em Preconsiderations:}
Since \( \lfloor (8 + 1)/3 \rfloor = 3 \), Theorem \ref{thm:nielsendepth} shows that we can  at best reduce to depth~3, meaning \( S_{5,3}(z) \) is a new more complicated function in weight 8.  On computing the 2-part of the motivic cobrackets, we find 

\begin{align*}
	\delta S_{6,2}(z) \={} & \{1\}_3 \wedge \{z\}_5 - \{z\}_3 \wedge \{1\}_5 \,, \\[1ex]
	\delta S_{5,3}(z) \={} & \{1\}_3 \wedge \big( \{z\}_5 + S_{3,2}(z) \big) + \big( \{1-z\}_3 - 2 \{z\}_3 \big) \wedge \{1\}_5 \,, \\[1ex]
	\delta S_{4,4}(z) \={} & \{1\}_3 \wedge \left( - \{z\}_5 - \{1-z\}_5  - \{1-z^{-1} \}_5 + 2 S_{3,2}(z) \right)  \\
		& + 2\left( \{1-z\}_3 - \{z\}_3 \right) \wedge \{1\}_5 \,.
\end{align*}

We observe that \( S_{5,3}(z) \) cannot reduce to \( S_{6,2} \) motivically, even with more complicated arguments, since it contains a single Nielsen polylogarithm in its coproduct, which can never be matched by \( S_{6,2} \) alone.  Instead, we expect  \( S_{5,3}(z) \) to behave like \( \Li_2 \) modulo \( S_{6,2} \) and \( \Li_8 \), as explained in Remark \ref{rem:nielsendepthbehaviour}.

\subsection{\texorpdfstring{Depth reduction of \( S_{4,3} \)}{Depth reduction of \textunderscore{}{4,3}}}

We know that \( S_{4,4}(z) \) reduces to \( S_{5,3} \), so we can attempt to do this by explicitly by killing the \( S_{3,2} \) factor in the motivic cobracket.

\begin{Prop}\label{prop:s44reduce}
	The following reduction expresses \( S_{4,4} \) in terms of lower depth Nielsen polylogarithms
	\begin{align*}
	S_{4,4}(z) \= {} &  2
	S_{5,3}(z)-S_{6,2}(1-z)-3 S_{6,2}(z)-S_{6,2}\Big(\frac{z}{z-1}\Big) \\
	& + 2 \Li_8(1-z)+4 \Li_8(z)+4 \Li_8\Big(\frac{z}{z-1}\Big) \Mod{products} \,.
	\end{align*}
	
	\begin{proof}
		The polynomial invariant of the difference of the left hand side and the right hand side is
		\begin{align*}
		20 X^3 Y^3 - \Big( 30 X^4 Y^2 + & 6 X Y^5  -18 X^5 Y -6 (X - Y)^5 Y \\
		& -2 Y^6 + 4 X^6  -4 (X - Y)^6 \Big) \= 0 \, . \qedhere
		\end{align*}
	\end{proof}
\end{Prop}

\subsection{\texorpdfstring{On the special values of \( S_{6,2}(z) \) and  \( S_{5,3}(z) \) at \( z=-1 \) and \(  z=\frac{1}{2} \)}{On the special values of S\textunderscore{}{6,2}(z) and   S\textunderscore{}{5,3}(z) at z = -1 and z = 1/2}}
At \( z = \tfrac{1}{2} \) or \( z = -1 \) we compute the coproduct as
\begin{alignat*}{4}
	\delta S_{6,2}(-1) &= -\frac{3}{16} \{1\}_3 \wedge \{1\}_5 \,, \quad\quad &
	\delta S_{6,2}\Big(\frac{1}{2}\Big) &= -\frac{15}{8} \{1\}_3 \wedge \{1\}_5 + \{1\}_3 \wedge \Big\{\frac{1}{2}\Big\}_5 \,, \\
	\delta S_{5,3}(-1) &= -\frac{15}{32} \{1\}_3 \wedge \{1\}_5 \,, &
	\delta S_{5,3}\Big(\frac{1}{2}\Big) &= -\frac{7}{8} \{1\}_3 \wedge \{1\}_5 + 2 \{1\}_3 \wedge \Big\{\frac{1}{2}\Big\}_5 \,.
\end{alignat*}

In order to match coproduct terms, we are thus led to investigating the following linear combination (on the left) and 
we find that it reduces to Riemann zeta values.

\begin{Prop}\label{prop:s62s53minusone}
 We have
    \[
	\frac{5}{2} S_{6,2}(-1) - S_{5,3}(-1) \= 
	-\frac{917}{768} \zeta(8) 
	+\frac{1}{2}\zeta (3) \zeta (5)
	+\frac{1}{4}\zeta(2)\zeta (3)^2
	\, .
    \]
 
 \begin{proof}
 This follows from the MZV Data Mine \cite{mzvDM}, since each $S_{n,p}(-1)$ already has the form of an alternating MZV. 
 \end{proof}
 \end{Prop}
 
A reduction of \( S_{5,3}(\tfrac{1}{2}) \) to \( S_{6,2}(\tfrac{1}{2}) \), \( 
S_{6,2}(-1) \), polylogs and products also exists, and  follows from the 
reduction of the reflection identity \( S_{5,3}(z) + S_{5,3}(1-z) \) in 
Proposition \ref{prop:s53ref} below.  However we should not expect a reduction 
of \( S_{6,2}(\tfrac{1}{2}) \) to anything of lower depth, since the cobracket 
contains the factor \( \{\tfrac{1}{2}\}_5 \neq 0 \).

\subsection{\texorpdfstring{Strategy for evaluating \( S_{6,2}(-1) \)}{Strategy for evaluating S\textunderscore{}\{6,2\}}}
\label{sec:s62minusone}
	Since \( \delta \zeta(3,5) = -5 \zeta(3) \wedge \zeta(5) = -5 \{1\}_3 
	\wedge \{1\}_5 \) it should be possible to reduce \( S_{6,2}(-1) \) and~\( 
	S_{5,3}(-1) \) individually to \( \Li_8 \) and products, if we allow also the more 
	familiar (conjecturally irreducible) constant~\( \zeta(3,5) \).
    
	More precisely, the following combination, with trivial coboundary, should 
	be expressible in terms of classical polylgoarithms and products of lower 
	weight terms
	\[
		S_{6,2}(-1)  - \frac{3}{80} \zeta(3,5) \overset{?}{\=} 0 \Mod{$\Li_8$, products} \, .
	\]
	However, such a reduction is likely to be \emph{much} more complicated than 
	the corresponding reduction for \( S_{4,2}(-1) \).  
	The complicated part of the \( S_{4,2}(-1) \) reduction stems from requiring terms \( \sum_j \alpha_j \Li_6(x_j) \) such that the \( (5,1) \)-part of their coproduct gives
	\[
		\sum\nolimits_j \alpha_j \Li_5(x_j) \otimes \log(x_j) \= \zeta(5) \otimes \log(2) \, .
	\]
	For weight 6, this was already possible using only arguments \( \pm 2^j \), since one has the identity \cite[p. 419]{Za1}
	\[
		 \LiZS_5\Big({-}\frac{1}{8}\Big) - 162 \LiZS_5\Big({-}\frac{1}{2}\Big) - 126 \LiZS_5\Big(\frac{1}{2}\Big) \= \frac{403}{16} \zeta(5) \, .
	\]
	
	To match the \( \zeta(7) \otimes \log(2) \) term in
	\[
		\Delta' S_{6,2}(-1) \= -\frac{15}{16} \zeta (3)\otimes \zeta (5)-\frac{3}{4}\zeta (5)\otimes \zeta (3)-\frac{127}{64} \zeta(7) \otimes \log (2)
	\]
	one should try to find a \( \Li_8 \) combination \( \sum \alpha_j 
	\Li_8(x_j) \) such that the \( (7,1) \)-part of their coproduct simplifies 
	to \( \zeta(7) \otimes \log(2) \).  Unfortunately, the simplest such \( 
	\Li_7 \) combination which gives a non-zero multiple of~\( \zeta(7) \) 
	already involves all the 29 exceptional \( \{2,3\} \)-units \cite[p. 
	420]{Za1}.  In fact we simultaneously require, for \( \nu_p(x_j) \) the 
	\(p\)-adic valuation of \( x_j \), that
	\begin{align*}
		&\sum\nolimits_j \alpha_j \LiZS_7(x_j) \nu_2(x_j) \:\in\: \zeta(7) 
		\mathbb{Q}^\times \, ,  \\
		&\sum\nolimits_j \alpha_j \LiZS_7(x_j) \nu_p(x_j) \= 0 \, , \quad 
		\text{\( p > 2 \)} \, ,
	\end{align*}
	in order to match \( \zeta(7) \otimes \log(2) \) in the coproduct, and to avoid generating extraneous terms \( \zeta(7) \otimes \log(p) \), \( p > 2 \).
	
	To find such a combination, we can slightly adapt the procedure from 
	\cite{Za1} for inductively computing elements in the Bloch groups \( 
	\mathcal{B}_n(F) \).  Take a set of elements \( X = \{ x_j \} \), each \( 
	x_j \) of the form \( \pm p_1^{k_1} \cdots p_\ell^{k_\ell} \).  Firstly, we 
	find the combinations \( \sum\nolimits_j n_j [x_j] \) in \( \ker \beta_8 
	\).  Here $\beta_m\colon \ZZ[F]\to\Sym^{m-2}(F^{\times}_{\QQ})\otimes 
	\bigwedge^2(F^{\times}_{\QQ})$ is as in~\cite{Za1} given by
    $[x]\mapsto [x]^{m-2}\otimes ([x]\wedge [1-x])$. 
    We then can impose the conditions
	\[
		\LiZS_k \Big( \sum\nolimits_j n_j \nu_{p_1}(x_j)^{\mu_1} \cdots \nu_{p_\ell}(x_j)^{\mu_\ell} [x_j] \Big) = 0 \,,
	\]
	for \( \mu_1 + \cdots + \mu_\ell = 8-k \), with \( k = 3, 5 \), to obtain combinations which give \( 0 \cdot \zeta(3) \) and \( 0 \cdot \zeta(5) \) under \( \LiZS_3 \) and \( \LiZS_5  \) respectively.  Assuming that \( p_1 = 2 \), we only need to impose the conditions
	\[
			\LiZS_7 \Big( \sum\nolimits_j n_j \nu_{p_i}(x_j) [x_j] \Big) = 0 \,,
	\]
	for \( i = 2, \ldots, \ell \), and then the combination \( \Lambda = 
	\sum\nolimits_j n_j [x_j] \) has the property we desire.  The same 
	observation as in \cite{Za1} shows that it is possible to satisfy these 
	conditions by taking \( X = X(S) \) to be some set of \( S \)-units, for a 
	sufficiently large set of primes \( S \).  Specifically, the number of 
	conditions imposed grows polynomially in the size of \( S \), but the 
	Erd\H{o}s-Stewart-Tijdeman Theorem (cf.~\cite{Za1}, p.425) shows that the size of \( X(S) \) grows 
	exponentially in the size of \( S \). \medskip

	In the case where \( x_j = \pm 2^a 3^b \), and \( 1-x_j \) contains only factors \( 2, 3, 5, \ldots, 23 \), (the original \( p = 2 \) , plus \( q = 7 \) new extra factors) we are in fact guaranteed to find such a solution.  The set of such \( x_j \) in \( (-1, 1) \), excluding squares, consists of 75 elements.  In weight \( w = 8 \), to be \( \ker \beta_8 \) we must impose
	\[
		63 = q \binom{w  + p - 2}{p-1} + \binom{w + p - 1}{p-1} - p
	\]
	conditions.  To force \( \LiZS_3 \) and \( \LiZS_5 \) images to be 0, we must impose a further
	\[
		10 = 6 + 4 = \sum_{k \in \{3, 5\}} \binom{w - k + p -1}{p-1} \, ,
	\]
	conditions.  Finally, we have only \( 1 = p - 1 \) more condition to force for the desired behaviour for the \( \LiZS_7 \) image.  In total we have 75 elements, and only 74 conditions, so the linear space of such combinations is (at least) 1-dimensional.
	
	After performing the linear algebra, we find exactly one combination of 60 of these elements (the full expression is given in Appendix \ref{app:reduction})
	\[
		\Lambda \ceq 50\,508\,755\,462\,288\,597\,796 \bigg[-\frac{2048}{2187} \bigg] + \cdots + 2\,651\,619\,475\,018\,716\,827\,904 \bigg[ \frac{243}{256} \bigg] \,,
	\]
	which satisfies
	\begin{align*}
		(\LiZS_7 \cdot \nu_2) (\Lambda) & \overset{?}{\=} -175\,442\,386\,671\,378\,179\,202\,538\,515 \zeta(7) \,, \\
		(\LiZS_7 \cdot \nu_3) (\Lambda) & \overset{?}{\=} 0 \,.
	\end{align*}
	Here we write \( (\LiZS_7 \cdot \nu_p) (x) \ceq \LiZS_7(x) \nu_p(x) \), and extend by linearity to formal linear combinations as usual.  \medskip
	
	With this combination, we can match the \( \log(2) \otimes \zeta(7) \) term 
	in \( \Delta' S_{6,2}(-1) \), and we obtain a candidate reduction of the 
	form
	
	\begin{Prop} \label{prop:s62minusone} We have the following conjectural evaluation of \( S_{6,2}(-1) \)
	\begin{align*}
		S_{6,2}(-1) \overset{?}{\=} 
		& -\frac{127}{64} \Li_8\Big( -(175\,442\,386\,671\,378\,179\,202\,538\,515)^{-1} \Lambda \Big) \\
		& + \frac{3}{80} \zeta(3,5)+\frac{15}{16} \zeta (3) \zeta (5)+\frac{127}{64} \log (2) \zeta (7) \\
		& + \sum_{ 2k + a + b = 8 }\lambda_{2k,a,b} \zeta(2k) \log^a(2) \log^b(3)
	\end{align*}
	for some \( \lambda_{2k,a,b} \in \QQ \) which come from the terms in 
	coproduct of \( \Li_8(\lambda) \) which arise from the product terms in the 
	analytic identities for \( (\Li_7 \cdot \nu_2)(\Lambda) \) and \( (\Li_7 
	\cdot \nu_3)(\Lambda) \), and their lower weight analogues.
	
	The full candidate for this reduction is given in Appendix \ref{app:reduction}, and has been verified to 20,000 decimal places in PARI/GP \cite{PARI2}.
\end{Prop}

\subsection{\texorpdfstring{Functional equations for \( S_{5,3} \)}{Functional equations for S\textunderscore{}\{5,3\}}}
\label{sec:s53alg}

We expect that \( S_{5,3} \) behaves like \( \Li_2 \) modulo lower depth Nielsen polylogarithms.  From Proposition \ref{prop:nielseninverse}, we have the inversion relation.  The reflection relation for \( S_{5,3} \) also holds, as the following shows.

\begin{Prop}\label{prop:s53ref}
	\( S_{5,3} \) satisfies the two-term reflection relation, modulo lower depth Nielsen polylogarithms
	\begin{align*}
	S_{5,3}(1-z) + S_{5,3}(z) \= & 2 S_{6,2}(1-z)+2 S_{6,2}(z)+S_{6,2}\Big(\frac{z}{z-1}\Big) \\&
	 -3 \Li_8(1-z)-3 \Li_8(z)-3 \Li_8\Big(\frac{z}{z-1}\Big) \Mod{products}\,.
	\end{align*}
	
	\begin{proof}
		The polynomial invariant is
		\begin{align*}
			-15 X^2 Y^4 + 15 X^4 Y^2 - \big( & -12 X Y^5 + 12 X^5 Y + 6 (X - 
			Y)^5 Y \\
			& + 3 Y^6  -3 X^6 + 3 (X - Y)^6 \big) \= 0 \, . \qedhere
		\end{align*}
	\end{proof}
\end{Prop}

By working out the product terms, one obtains a reduction
\begin{align*}
	S_{5,3}\Big(\frac{1}{2}\Big) \= &
	S_{6,2}\bigg(2 \bigg[\frac{1}{2}\bigg] + \frac{1}{2} \big[ -1 \big] \bigg)
	-3 \Li_8\Big(\frac{1}{2}\Big)
	-2\Li_7\Big(\frac{1}{2}\Big) \log (2)
	+S_{5,2}\Big(\frac{1}{2}\Big) \log(2)
	\\& 
	-\frac{1}{2} \Li_6\Big(\frac{1}{2}\Big) \log ^2(2)
	+\frac{2311}{768} \zeta(8)
	+\frac{1}{4} \zeta (2) \zeta (3)^2
	-\frac{1}{2} \zeta (3) \zeta (5)
	\\& 
	-\Big( \frac{255}{128} \zeta (7) 
	-\frac{1}{8} \zeta (4) \zeta (3) 
	-\frac{1}{2} \zeta(2) \zeta (5) \Big) \log (2)
	+\frac{1}{2!} \Big( \frac{23}{16} \zeta(6) 
	-\ \zeta (3)^2 \Big) \log ^2(2)
	\\&
	-\frac{1}{3!} \Big(  2 \zeta (5)
	-  \zeta (2) \zeta (3) \Big) \log ^3(2)
	+\frac{5}{4 \cdot 4!} \zeta (4) \log ^4(2)
	\\&
	-\frac{3}{5!} \zeta (3) \log ^5(2)
	 +\frac{3}{6!} \zeta (2) \log ^6(2)
	-\frac{10}{8!}\log	^8(2)
	 \, .
\end{align*}
This confirms the reduction suggested above following Proposition \ref{prop:s62s53minusone}.
\medskip

Naturally, one would hope to find a reduction for \( S_{5,3} \) of the five-term relation.  Using the result for \( S_{3,2} \), we can eliminate the \( \{1\}_3 \wedge  S_{3,2}(z) \) component of the coproduct, from \( S_{5,3}(\text{five-term}) \).  Unfortunately, we are still left with the non-trivial task of matching the remainder with \( S_{6,2} \) terms, with rational arguments.  The difficult part is to match the \( \sum\nolimits_i \{1\}_3 \wedge \{f_i(z)\}_5 \) and \( \sum\nolimits_{j} \{1\}_5 \wedge \{g_j(z)\}_3 \) components simultaneously with a combination of \( S_{6,2} \) terms. One could apply the idea of Section \ref{sec:limlin} and use the duplication relation, to obtain
\[
	\delta S_{6,2}\Big( \frac{1}{4} \big[z^2\big] - \big[z\big] - \big[{-}z\big]  \Big) = \frac{3}{16} \{1\}_3 \wedge \{z^2\}_5 \, .
\]
By substituting \( \sqrt{f_i(z)} \) into this, one can match by brute force the full motivic cobracket of \( S_{5,3}(\text{5-term}) \).  But then one is left with the more difficult task of matching the mod-products symbol by \( \Li_8 \) terms of arbitrary \emph{algebraic} arguments. \medskip

On the other hand, \( S_{3,2} \) of the algebraic \( \Li_2 \) equation from Section \ref{sec:algfe} has a relatively simple expression in terms of \( \Li_5 \).  So matching the \( S_{5,3} \) combination is more straightforward.  We have, noting that Proposition \ref{prop:s53ref} covers the special case \( a=b=1 \) in more detail, that
\begin{Prop} \label{prop:s53alg}
	Let \(a,b,c\in \ZZ\sm\{0\}\), with \( a+b+c=0 \), and let \( \{p_i(t)\}_{i=1}^r \)  be the roots of \( x^a(1-x)^b = t \).  For convenience take \( a > 0 \).
	Then the following functional equation holds on the level of the mod-products symbol
	\begin{align*}
		\sum_{i=1}^{r} S_{5,3}(p_i(t)) \modsh & \sum_{i=1}^{r} \bigg\{ \frac{2b - a}{b} S_{6,2}(p_i(t)) + \frac{b}{a} S_{6,2}(1 - p_i(t)) + \frac{b}{a+b} S_{6,2}(1 - p_i(t)^{-1}) \\
		&- \frac{a^2 - 2ab + 3b^2}{b^2} \Li_8(p_i(t)) - \frac{2ab - b^2}{a^2} \Li_8(1-p_i(t)) - \frac{2ab + 3b^2}{(a+b)^2} \Li_8(1 - p_i(t)^{-1}) \bigg\} \,.
	\end{align*}	
\end{Prop}

	\begin{Cor}
		We have the clean single-valued identity
		\begin{align*}
		\sum_{i=1}^{r} \SCS_{5,3}(p_i(t)) {}= & \sum_{i=1}^{r} \bigg\{ \frac{2b - a}{b} \SCS_{6,2}(p_i(t)) + \frac{b}{a} \SCS_{6,2}(1 - p_i(t)) + \frac{b}{a+b} \SCS_{6,2}(1 - p_i(t)^{-1}) \\
		&- \frac{a^2 - 2ab + 3b^2}{b^2} \LiCS_8(p_i(t)) - \frac{2ab - b^2}{a^2} \LiCS_8(1-p_i(t)) - \frac{2ab + 3b^2}{(a+b)^2} \LiCS_8(1 - p_i(t)^{-1}) \bigg\} \,.
		\end{align*}
	\end{Cor}

	\begin{proof}
		Consider the limit \( t \to 0 \) and use \( \LiCS_8(0) = \LiCS_8(1) = \LiCS_8(\infty) = 0 \), \( \SCS_{6,2}(0) = \SCS_{6,2}(1) = \SCS_{6,2}(\infty) = 0 \), and \( \SCS_{5,3}(0) = \SCS_{5,3}(1) = \SCS_{5,3}(\infty) = 0 \).
		
		If \( b > 0 \), we obtain roots \( p_i = 0 \) with multiplicity \( a \) and \( p_i = 1 \) with multiplicity \( b \), giving the constant is \( 0 \).  If \( -a < b < 0 \), we obtain roots \( p_i = 0 \) with multiplicity \( a \), also giving the constant is \( 0 \).  Otherwise \( b < -a \), an we obtain roots \( p_i = 0 \) with multiplicity \( a \) and roots \( p_i = \infty \) with multiplicity \( -a -b \), still giving the constant is \( 0 \).
	\end{proof}
	
	\begin{proof}[Proof of Proposition]
		The proof strategy is the same as in the previous cases.  It reduces 
        to weight 7 functional equations, including the one in Proposition \ref{prop:s52alg}.
	\end{proof}

\subsection{Nielsen ladders} The concept of a `ladder' in some given weight \( N\) (already used above) was introduced by Lewin in order to account for 
identities of the type\  \( \Li_N \big( \sum_i \sum_{k=0}^1 {n_{i,k}} [(-1)^k \theta^i]\big) =0 \) \ with \( n_{i,k}\in \ZZ\), \( i\geq 0 \), for some algebraic 
number \( \theta \). Due to the duplication relation for  \( \Li_N \) one can actually reduce each such to a linear combination where all signs
\( (-1)^k \) have been dropped.
\begin{Def} We call an identity of the type 
 \( \sum_j \sum_{k=0}^1 S_{N-j,j} \big( \sum_i \sum_{k=0}^1 {n_{i,j,k}} [(-1)^k \theta^i]\big) =0 \) with \( n_{i,j,k}\in \ZZ\) 
a {\em Nielsen ladder} of weight \( N \).
\end{Def}

\begin{Rem}
We have a  non-trivial example of a Nielsen ladder in weight~8 for the algebraic number \( -\omega \) from \S  \ref{s32ladderomega}. The exact same procedure as used there applies, except for the final use of the duplication relation which is not known for \( S_{5,3} \) and we can depth reduce \( S_{5,3}(2 [\om]+[-\om] ) \) on the level of clean single-valued functions.
\end{Rem}

\section{\texorpdfstring{A family of depth reductions in general weight with arguments \( z \),  \( 1 - z \) and  \( 1 - z^{-1}\)}{A family of depth reductions in general weight with arguments z,  1 - z and  1 - 1/z}} 
	\subsection{Depth reduction in general weight}
	We end with the following result generalising the \( \Li_2 \)-behaviour of 
	\( S_{3,2} \) and \( S_{5,3} \) modulo lower depth from Propositions 
	\ref{prop:s32twoterm} and \ref{prop:s53ref}, and the \( \Li_3 \)-behaviour 
	of \( S_{4,2} \) modulo lower depth from Proposition \ref{prop:s42three}. 
	Moreover, it supports the claim about the behaviour of \( S_{2m-\varepsilon,m} \) for \( \varepsilon\in\{0,1,2\} \)
	alluded to in Remark \ref{rem:nielsendepthbehaviour}. More precisely, 
	we prove that \( S_{2m-2,m} \) reduces to lower depth, and we expect that \( S_{2m-1,m} \) behaves like \( S_{1,1} = \Li_2 \),
	and \( S_{2m,m} \) behaves like \( S_{2,1} = \Li_3 \). For other cases, the cobracket potentially 
	involves several terms of maximal depth. 
	
	\begin{Thm}\label{thm:depthreduce}
		For all \( m \geq 1 \) the following depth reductions, and 2-term and 3-term identities hold. \bigskip
		\begin{enumerate}
		\item	
		\abovedisplayskip=0pt\abovedisplayshortskip=0pt~\vspace*{-1.1\baselineskip}
		\begin{align*}
		& \hspace{-10em} S_{2m,m}([z] + [1-z] + [1-z^{-1}]) \modsh  \\
		 & \hspace{-10em} \sum_{j=1}^{m-1} (-1)^{j+1} \binom{m-1+j}{j} S_{2m+j,m-j}\Big( [z] + [1-z] + [1-z^{-1}] \Big)\,. 
		\end{align*}
		\bigskip
		\item 
		\abovedisplayskip=0pt\abovedisplayshortskip=0pt~\vspace*{-1.1\baselineskip}
		\begin{align*}
			& \hspace{-6em } S_{2m-1,m}([z] + [1-z]) \modsh \\
			& \hspace{-6em} \sum_{j=1}^{m-1} (-1)^{j+1} \binom{m-2+j}{j} S_{2m-1+j,m-j}\Big([z] + [1-z] + \frac{j}{m-1} [1 - z^{-1} ] \Big)\,.\hfill
		\end{align*}\bigskip
		\item 
		\abovedisplayskip=0pt\abovedisplayshortskip=0pt~\vspace*{-1.1\baselineskip}
		\begin{align*}
		&S_{2m-2,m}(z) \modsh \\
		& \sum_{j=1}^{m-1} (-1)^{j+1} \binom{m-2+j}{j} S_{2m-2+j,m-j}\Big([z] - \frac{j}{m+j-2} [1-z] - \frac{j}{m+j-2} [1 - z^{-1} ] \Big)\,.
		\end{align*}
		\end{enumerate}
	\end{Thm}

	\begin{proof}
		Under the polynomial invariant from Section~\ref{sec:spanningset}, Part 
		(i) is equivalent to the identity
		\begin{equation}\label{eqn:p3term}
			P_{m-1}(X,Y) - P_{m-1}(Y,X) + P_{m-1}(Y-X,-X) = 0 \,,
		\end{equation}
		where
		\[
			P_n(X,Y) \ceq \sum_{j=0}^n (-1)^j \binom{n + j}{j} \binom{3n + 1}{n-j} X^{2n+1+j} Y^{n-j} \,.
		\]
		A routine calculation shows that
		\[
			P_n(X,Y) = X^{n+1}\sum_{a+b=n} \binom{n+a}{a} (XY)^a \binom{n+b}{b} 
			(X(Y-X))^{b} \, ,
		\]
		which implies that
		\[
			P_n(X,Y) = [w^n] \Big( \frac{X}{(1 - X Y w) (1 - X(Y-X) w)} 
			\Big)^{n+1} \,,
		\]
        where $[z^n]f$ denotes the coefficient of~$z^n$ in a power series~$f$.
		Therefore, using the Lagrange inversion formula, we see that the 
		generating series
		\[
			U(X,Y,z) \ceq \sum_{n=0}^\infty P_n(X,Y) \frac{z^{n+1}}{n+1} \, 
		\]
		satisfies the cubic equation
		\[
			U(1 - XY U)(1 - X(Y-X) U) \= X z \,.
		\]
		
		This cubic equation in \( U \) has three solutions, and it is easy to check that they are given by
		\begin{align*}
			U_1(X,Y,z) &= U(X,Y,z) \,, \\
			U_2(X,Y,z) &= \frac{1}{X Y} - U(Y,X,z) \,, \\
			U_3(X,Y,z) &= \frac{1}{X (Y-X) } + U(Y-X,-X,z) \, .
		\end{align*}
		Since the coefficient of \( -U^2 \) in the associated monic cubic equation is 
		\(
			\frac{1}{X Y} + \frac{1}{X (Y-X)}
		\)
		we get that
		\[
			U(X,Y,z) - U(Y,X,z) + U(Y-X,-X,z) \= 0\,,
		\]
		which proves \eqref{eqn:p3term}.
        
        The other two parts are proved similarly, so we only sketch the proof 
        of~(ii) which is slightly more complicated than~(i). 
        In this case we need to prove
        \begin{equation}\label{eqn:p2term}
        Q_{m-1}(X,Y) - Q_{m-1}(Y,X) + \widetilde{Q}_{m-1}(Y-X,-X) = 0 \,,
        \end{equation}
        where 
        \begin{align*}
        Q_n(X,Y) &\ceq \sum_{j=0}^n (-1)^j \binom{n - 1 + j}{j} \binom{3n}{n-j} 
        X^{2n+j} Y^{n-j} \,, \\
        \widetilde{Q}_n(X,Y) &\ceq \sum_{j=0}^n (-1)^j \binom{n - 1 + j}{j} 
        \binom{3n}{n-j}\frac{j}{n} X^{2n+j} Y^{n-j}\,,\quad n>0 \,,
        \end{align*}
        and $\widetilde{Q}_0\ceq0$. We claim that the generating series
        \[
        V(X,Y,z) \ceq \sum_{n=0}^\infty Q_n(X,Y) \frac{z^{n+1}}{n+1} \,,
        \quad \quad
        \widetilde{V}(X,Y,z) \ceq \sum_{n=0}^\infty \widetilde{Q}_n(X,Y) 
        \frac{z^{n+1}}{n+1} \,,
        \]
        can be expressed in terms of $U(X,Y,z)$ as
        \begin{equation} \label{eqn:p2termgen}
        V(X,Y,z) \= X^{-1}U+\frac{1}{2}(X-Y)U^2\,,\quad\quad
        \widetilde{V}(X,Y,z) = -\frac{1}{2}XU^2\,,
        \end{equation}
        from which after a simple calculation we see that~\eqref{eqn:p2term} 
        is implied by
        \[
        U_1^k+U_2^k+U_3^k \= \frac{1}{(XY)^k}+\frac{1}{(X(Y-X))^k}\,,\quad 
        k=1,2\,,
        \]
        which again follows from Vieta's formulas for the cubic equation 
        satisfied by $U$. To prove~\eqref{eqn:p2termgen} we use Lagrange
        inversion in a more general form 
            \[H(g(z)) \= 
            \sum_{n\ge0}[w^n](H'(w)\phi^{n+1}(w))\frac{z^{n+1}}{n+1}\,,\]
        where $g(z)$ satisfies $g(z)=z\phi(g(z))$ and $H(w)$ is a formal 
        power series without a constant term ($\phi$ is a power series with 
        $\phi(0)\ne0$). 
        To obtain~\eqref{eqn:p2termgen} we use the following simple identities
        \[
        Q_n(X,Y) \= [w^n]\frac{X^{2n}(1+Yw)^{3n}}{(1+Xw)^{n}}\,,\quad\quad
        \widetilde{Q}_n(X,Y) \= 
        [w^n]\frac{-wX^{2n+1}(1+Yw)^{3n}}{(1+Xw)^{n+1}}\,,
        \]
        and the analogous identity for $P_n$
        \[
        P_n(X,Y) \= [w^n] \frac{X^{2n+1}(1+Yw)^{3n+1}}{(1+Xw)^{n+1}} \,,
        \]
        together with the Lagrange inversion formula for    
        $$\phi(w) \=\frac{X^2(1+Yw)}{(1+Xw)}\,,$$ and the following three choices
        for~$H$:
        \begin{equation*}
        	H_1(w)  \= \frac{w(2+(X+Y)w)}{2X^2(1+Yw)^2} \,, \quad
        	H_2(w)  \= \frac{-w^2}{2X(1+Yw)^2} \text{\,, and} \quad
        	H_3(w)  \= \frac{w}{X(1+Yw)} \,.\qedhere
        \end{equation*}
	\end{proof}

\begin{Cor} There are Nielsen ladders in arbitrary weight.
\end{Cor}
\begin{proof} Denote by  \( \theta \) a root of \( (1-x) \pm x^r\) for some \( r \in \ZZ_{>0} \), so that
\( 1 - \theta = \mp\theta^{r} \)  and \( 1 - \theta ^{-1} = \pm\theta^{r-1} \).
 Specialising to \( z= \theta \)  in the theorem, all the arguments become, up to sign, powers of \( \theta \). Terms with 
negative powers of \( \theta \) can be replaced, via inversion (Proposition \ref{prop:nielseninverse}), by positive powers.
\end{proof}

\appendix

\section{\texorpdfstring{Evaluation of $S_{3,2}$ at values involving the golden ratio}{Evaluation of S\textunderscore{}{3,2} at values involving the golden ratio}}\label{app:reductiongoldenratio} 

	Recall the following evaluations  involving the golden ratio \( \phi = \frac{1}{2} (1 + \sqrt{5}) \) for \( \Li_2 \) (see \cite[Equations 1.20 and 1.21]{Le}, or \cite[Section 1.1]{Za3}):
	\begin{alignat*}{2}
		\Li_2(\phi^{-2}) &\= \phantom{{}+{}}\frac{2}{5}\zeta(2) - \phantom{\frac{1}{1}} \log ^2(\phi) \,, \quad\quad & \Li_2(\phi^{-1}) &\= \phantom{{}+{}}\frac{3}{5}\zeta(2) - \log ^2(\phi) \,, \\[1ex]
		\Li_2(-\phi^{-1}) &\= -\frac{2}{5} \zeta(2) + \frac{1}{2}\log ^2(\phi) \,, \quad\quad &	\Li_2(-\phi ) &\= -\frac{3}{5} \zeta(2)-\log ^2(\phi) \,.
	\end{alignat*}

    Corresponding to these \( \Li_2 \) evaluations, we have the following evaluations for the clean single-valued Nielsen polylogarithm \( \SCS_{3,2} \):
    \begin{alignat*}{7}
	\SCS_{3,2}(\phi ^{-2})\=\frac{1}{33} \LiCS_5\big(&&{}-8 [\phi ^{-3}]{}+{}&&780 [\phi ^{-1}]{}+{}&&804 [-\phi ]{}+{}&&8 [-\phi ^3]\big){}+{}&&\zeta (5) \,, \\[1ex]
	\SCS_{3,2}(-\phi ^{-1})\=\frac{1}{33}\LiCS_5\big(&& 8 [\phi ^{-3}]{}+{}&&243 [\phi ^{-1}]{}+{}&&318 [-\phi ]{}-{}&&8 [-\phi ^3]\big){}+{}&&\zeta (5) \,,  \\[1ex]
	\SCS_{3,2}(\phi ^{-1})\=\frac{1}{33}\LiCS_5\big(&& 8 [\phi ^{-3}]{}-{}&&219 [\phi ^{-1}]{}-{}&&243 [-\phi]{}-{}&&8 [-\phi ^3]\big){}+{}&&\zeta (5)  \,, \\[1ex]
	\SCS_{3,2}(-\phi )\=\frac{1}{33}\LiCS_5\big(&&{}-8 [\phi ^{-3}]{}-{}&&243 [\phi ^{-1}]{}-{}&&219 [-\phi ]{}+{}&&8 [-\phi ^3] \big) {}+{}&&\zeta (5) \,.
    \end{alignat*}
    For the complex analytic Nielsen polylogarithm \( S_{3,2} \) we have:
    \begin{align*}
	    S_{3,2}(\phi ^{-2})\={}&
	    \frac{1}{33} \Li_5\big({}-8 [\phi ^{-3}]{}+{}780 [\phi ^{-1}]{}+{}804 [-\phi ]{}+{}8 [-\phi ^3]\big)
	    + \Li_4(\phi ^{-2})\log (\phi ) \\ &
	    +\frac{1}{2}\zeta (5)
	    +\frac{481}{11} \zeta (4) \log (\phi )
	    -\zeta(3)\Li_2(\phi ^{-2})
	     +\frac{50}{11} \zeta (2) \log ^3(\phi )
	    +\frac{14}{15}\log ^5(\phi ) \,, \\[2ex]
	    S_{3,2}(-\phi ^{-1})\={}&
	    \frac{1}{33}\Li_5\big( \phantom{{}+{}} 8 [\phi ^{-3}]{}+{}243 [\phi ^{-1}]{}+{}318 [-\phi ]{}-{}8 [-\phi ^3]\big)
	    -\Li_4(-\phi ^{-1}) \log (\phi )\\ &
	    +\frac{1}{2}\zeta (5)
	    +\frac{325}{22} \zeta (4) \log (\phi )
	    -\zeta(3)\Li_2(-\phi ^{-1})
	     +\frac{21}{22} \zeta (2) \log ^3(\phi )
	    -\frac{7}{12} \log ^5(\phi ) \,, \\[2ex]
	    S_{3,2}(\phi ^{-1})\={}&
	    \frac{1}{33}\Li_5\big( \phantom{{}+{}} 8 [\phi ^{-3}]{}-{}219 [\phi ^{-1}]{}-{}243 [-\phi]{}-{}8 [-\phi ^3]\big)
	    +2\Li_4(\phi ^{-1}) \log (\phi ) \\ &
	    +\frac{1}{2}\zeta (5)
	    -\frac{335}{22} \zeta (4) \log (\phi )
	    -\zeta(3)\Li_2(\phi ^{-1})
	     -\frac{28}{11} \zeta (2) \log ^3(\phi )
	    -\frac{8}{15} \log ^5(\phi ) \,, \\[2ex]
	    S_{3,2}(-\phi )\={}&
	    \frac{1}{33}\Li_5\big({}-8 [\phi ^{-3}]{}-{}243 [\phi ^{-1}]{}-{}219 [-\phi ]{}+{}8 [-\phi ^3] \big)
	    -2 \Li_4(-\phi )  \log (\phi )\\ &
	    +\frac{1}{2}\zeta (5)
	    -\frac{325}{22} \zeta (4) \log (\phi)
	     -\zeta(3)\Li_2(-\phi)
	    -\frac{16}{11} \zeta (2) \log ^3(\phi )
	    +\frac{8 }{15}\log ^5(\phi ) \,.
    \end{align*}

\section{\texorpdfstring{Reduction of $S_{3,2}(\big[\frac{1}{9}\big] - 6 \big[\frac{1}{3}\big]))$ to $\Li_5$ and lower depth}{Reduction of S\textunderscore{}{3,2}([1/9] - 6 [1/3]) to Li\textunderscore{}5 and lower depth}}\label{app:reductions32ladder}
\begin{align*}
	S_{3,2}\bigg(\bigg[\frac{1}{9} \bigg]-6 \bigg[\frac{1}{3}\bigg] \bigg) \= 
	\Li_5\bigg( & \frac{1}{16}  \bigg[\frac{1}{9}\bigg]+\frac{21}{2} \bigg[\frac{1}{4}\bigg]+36 \bigg[\frac{1}{3}\bigg]
	 -100 \bigg[\frac{1}{2}\bigg] -60 \bigg[\frac{2}{3}\bigg]+\frac{69}{2} \bigg[\frac{3}{4}\bigg]-2 \bigg[\frac{8}{9}\bigg] \bigg) \\ {}&
	 + \frac{1855 }{12}\zeta (5)
	 + 6\Li_4\Big(\frac{1}{3}\Big) \log\Big(\frac{2}{3}\Big) - \Li_4\Big(\frac{1}{9}\Big) 
	\log \Big(\frac{8}{9}\Big)
	\\&
	+\zeta (2) \Big(\frac{128}{3}\log ^3(2)
	-84 \log^2(2) \log (3)
	+54 \log (2) \log ^2(3)
	-\frac{61}{6} \log ^3(3)\Big)
	\\&
	- \zeta(3) \Li_2\bigg(\bigg[\frac{1}{9} \bigg]-6 \bigg[\frac{1}{3}\bigg] \bigg)
	 +\zeta (4) \Big(52 \log (2)-\frac{239}{4}\log(3) \Big) 
	\\& -\frac{67}{6}\log ^5(2)
	+ 23 \log ^4(2) \log (3)
	 -23 \log ^3(2) \log ^2(3) 
	\\& +17\log ^2(2) \log ^3(3)
	-\frac{33}{4} \log (2) \log ^4(3)
	 +\frac{19}{12}\log ^5(3)
	 \,.
\end{align*}

\section{\texorpdfstring{Reduction of $S_{6,2}(-1)$ to $\zeta(3,5)$ and lower depth}{Reduction of S\textunderscore{}{6,2}(-1) to zeta(3,5) and lower depth}}\label{app:reduction}

The following combination which we abbreviate as $\Lambda$ below {\footnotesize

 \scriptsize
\begin{alignat*}{4}
& \phantom{{}+{}} 50508755462288597796 \, \big[{-}2^{11} 3^{-7}\big]
&&+69841566365930200554764814 \, \big[{-}2^{-2} 3^1\big]
&&-775364232778811798418105642 \, \big[{-}2^1 3^{-1}\big] 
\\[-1ex]
&+9614338651927197388368 \, \big[{-}2^{-7} 3^4\big]
&&+356655652241330545382160 \, \big[{-}2^{-4} 3^2\big]
&&+22509382601419271262985124160 \, \big[{-}2^{-1}\big]
\\[-1ex]
&+126912035059272811134 \, \big[{-}2^{10} 3^{-7}\big]
&&-94164506374654687219920 \, \big[{-}2^2 3^{-2}\big]
&&-14944644124416655455996 \, \big[{-}2^{-6} 3^3\big]
\\[-1ex]
&+578363469392155525327836 \, \big[{-}2^{-3} 3^1\big]
&&-1389650271294609827123449194 \, \big[{-}3^{-1}\big]
&&+142150983452642605772646 \, \big[{-}2^3 3^{-3}\big]
\\[-1ex]
&+234866506563215285097901896 \, \big[{-}2^{-2}\big]
&&-2156235930824838852840480 \, \big[{-}2^1 3^{-2}\big]
&&+545472150324080280895440 \, \big[{-}2^{-4} 3^1\big]
\\[-1ex]
&+52935763185068637077963640 \, \big[{-}2^{-1} 3^{-1}\big]
&&-1637817842535871022208 \, \big[{-}2^5 3^{-5}\big]
&&-3832788189554116913056832 \, \big[{-}2^{-3}\big]
\\[-1ex]
&-7534735430532974309850624 \, \big[{-}3^{-2}\big]
&&+229760377972981805891088 \, \big[{-}2^{-5} 3^1\big]
&&+1824486564349437387795018 \, \big[{-}2^{-2} 3^{-1}\big]
\\[-1ex]
&+146288680291430373180960 \, \big[{-}2^{-1} 3^{-2}\big]
&&-1841986300588118314548 \, \big[{-}2^{-9} 3^3\big]
&&-449734750753601709254502 \, \big[{-}2^{-3} 3^{-1}\big]
\\[-1ex]
&+1026645718908856249515210 \, \big[{-}3^{-3}\big]
&&+104005977148977093591408 \, \big[{-}2^{-5}\big]
&&+98806808342364061998789 \, \big[{-}2^{-4} 3^{-1}\big]
\\[-1ex]
&-90845492233250820003624 \, \big[{-}2^{-1} 3^{-3}\big]
&&-856806635887547864148 \, \big[{-}2^2 3^{-5}\big]
&&-1618278846243184730952 \, \big[{-}2^{-6}\big]
\\[-1ex]
&-19002715158472937734824 \, \big[{-}2^1 3^{-5}\big]
&&-865828810038222668088 \, \big[{-}2^{-2} 3^{-4}\big]
&&-6222083524060876926624 \, \big[{-}2^{-7} 3^{-1}\big]
\\[-1ex]
&-1110630706006093486416 \, \big[{-}2^{-9}\big]
&&-738916774978949856954 \, \big[{-}2^{-6} 3^{-3}\big]
&&+256696765017519764574 \, \big[{-}2^{-1} 3^{-7}\big]
\\[-1ex]
&-538995368726709238620 \, \big[2^{-3} 3^{-6}\big]
&&+2659063284848174620104 \, \big[3^{-5}\big]
&&+26750471369678154671328 \, \big[2^{-5} 3^{-1}\big]
\\[-1ex]
&-22869536297787068698224 \, \big[2^{-7} 3^1\big]
&&+26750471369678154671328 \, \big[3^{-3}\big]
&&+1178782871405313104861940 \, \big[2^{-1} 3^{-2}\big]
\\[-1ex]
&+521062220884439910260592 \, \big[2^1 3^{-3}\big]
&&+2935816641298266366693024 \, \big[2^{-2} 3^{-1}\big]
&&-310513457035175924880 \, \big[2^{-11} 3^5\big]
\\[-1ex]
&-1295961764172934408392024 \, \big[2^{-3}\big]
&&+47361088156862575216815120 \, \big[2^{-1} 3^{-1}\big]
&&+96143386519271973883680 \, \big[2^{-4} 3^1\big]
\\[-1ex]
&+6405316545334014067721724 \, \big[2^1 3^{-2}\big]
&&+25072646886079199648640 \, \big[2^3 3^{-3}\big]
&&-1384706396500391342516779656 \, \big[3^{-1}\big]
\\[-1ex]
&+8128823582861345906142336 \, \big[2^{-3} 3^1\big]
&&-7000414961399434681344 \, \big[2^5 3^{-4}\big]
&&+22794254041869638225651427336 \, \big[2^{-1}\big]
\\[-1ex]
&+116644982485749618488112 \, \big[2^4 3^{-3}\big]
&&-791724010495502232049202784 \, \big[2^1 3^{-1}\big]
&&+48384399276356552737368768 \, \big[2^{-2} 3^1\big]
\\[-1ex]
&-118427126740566034100976 \, \big[2^{-5} 3^3\big]
&&+421717883947747928801820 \, \big[2^3 3^{-2}\big]
&&+2651619475018716827904 \, \big[2^{-8} 3^5\big] \,,
\end{alignat*}
}
satisfies
\begin{align*}
(\LiZS_7 \cdot \nu_2) (\Lambda) & \overset{?}{\=} -175442386671378179202538515 \zeta(7) \,, \\*
(\LiZS_7 \cdot \nu_3) (\Lambda) & \overset{?}{\=} 0 \,,
\end{align*}
where \( (\LiZS_7 \cdot \nu_p)(x) \ceq \LiZS_7(x) \nu_p(x) \), extended to formal linear combinations by linearity as usual.

 This combination \( \Lambda\) gives the following candidate reduction of \( S_{6,2}(-1) \) to \( \zeta(3,5) \), polylogs and products.  It has been verified to 20,000 decimal places in PARI/GP \cite{PARI2}.
\begin{align*}
& S_{6,2}(-1) \overset{?}{\=} 
 -\frac{127}{64} \Li_8\Big( (-175442386671378179202538515)^{-1} \Lambda \Big) \\
& + \frac{3}{80} \zeta(3,5)+\frac{15}{16} \zeta (3) \zeta (5)+\frac{127}{64} \log (2) \zeta (7) 
 -\tfrac{19402627481307724677394420487}{5658362329180826944988958720}	\zeta (8) \\[1ex]
& 
+ \big( -\tfrac{614191658599101564926973}{2508139330310650241573120} \log ^2(2)
-\tfrac{2886590561748348125557503}{5894127426230028067696832} \log (2) \log (3) \\
&\quad\quad
+\tfrac{156730163328824575308249123}{943060388196804490831493120}	\log ^2(3) \big) \zeta (6) \\
&
+\big( \tfrac{4333162707367886155376747}{58941274262300280676968320} \log ^4(2) 
-\tfrac{5659544278330031492868951}{29470637131150140338484160} \log ^3(2) \log (3) 
\\&\quad\quad
+\tfrac{14956296074710188358483053}{58941274262300280676968320} \log ^2(2) \log ^2(3) 
-\tfrac{4904619219145722353667}{40041626536888777633810} \log (2) \log ^3(3) 
\\& \quad\quad
+\tfrac{18073086979999367134054497 }{943060388196804490831493120}\log ^4(3) \big) \zeta	(4) \\
&
+ \big( -\tfrac{91321852215653745732887 }{55257444620906513134657800}\log ^6(2) 
-\tfrac{1295653753433384401869 }{75372473481202404957760}\log ^5(2) \log (3)
\\& \quad\quad
 +\tfrac{339723349179163751737503}{5894127426230028067696832} \log ^4(2) \log ^2(3) 
-\tfrac{1084240497917160885726101}{14735318565575070169242080} \log ^3(2) \log ^3(3) 
\\& \quad\quad
 +\tfrac{2552445606672179607323583}{58941274262300280676968320} \log ^2(2) \log ^4(3) 
-\tfrac{176801938693662694517853}{14735318565575070169242080} \log (2) \log ^5(3) 
\\& \quad\quad
+\tfrac{3078249389028940034364037}{2357650970492011227078732800} \log ^6(3) \Big) \zeta	(2)
\\& + \big( \tfrac{66439339984835171875512961}{7072952911476033681236198400} \log	^8(2)
-\tfrac{93928768086503200785586131}{2062944599180509823693891200} \log ^7(2) \log	(3)
\\& \quad\quad
+\tfrac{30360617039303735894889903 }{294706371311501403384841600}\log ^6(2) \log	^2(3)
-\tfrac{40045462381407097906386587 }{294706371311501403384841600}\log ^5(2) \log	^3(3)
\\& \quad\quad
+\tfrac{6559870959604648612738713 }{58941274262300280676968320}\log ^4(2) \log	^4(3)
-\tfrac{16638301041582761143511231 }{294706371311501403384841600}\log ^3(2) \log	^5(3)
\\& \quad\quad
+\tfrac{10057243964473014455324631}{589412742623002806769683200} \log ^2(2) \log	^6(3)
-\tfrac{1437747378492833287500083}{515736149795127455923472800} \log (2) \log	^7(3)
\\& \quad\quad
+\tfrac{74287595521730329182293}{401302292849704038651699200} \log ^8(3) \big)
 \, .
\end{align*}

\end{document}